%% file: main.tex
\documentclass{article} 

\PassOptionsToPackage{round}{natbib}

\usepackage{authblk,breakcites}
\usepackage[round]{natbib}
\setcitestyle{authoryear}
\usepackage{microtype}
\usepackage{graphicx}
\usepackage{sidecap}
\usepackage{caption}
\usepackage{subfigure}
\usepackage{subcaption}
\usepackage{booktabs} 
\usepackage{multirow}
\usepackage[svgnames,dvipsnames]{xcolor}
\usepackage{float} 
\usepackage{wrapfig}
\usepackage{geometry}

\usepackage{hyperref}
\hypersetup{
    colorlinks=true,
    linkcolor=blue,
    filecolor=magenta,      
    urlcolor=cyan,
    citecolor=Emerald
    }

\usepackage{amsmath}
\usepackage{amssymb}
\usepackage{amsfonts}
\usepackage{mathtools}
\usepackage{amscd}
\usepackage{mathrsfs}
\usepackage{amsthm}
\usepackage[capitalize,noabbrev]{cleveref}
\usepackage{enumerate}

\usepackage{dsfont}

\usepackage{ulem}
\input{notation.tex}

\title{Controlled Interacting Branching Diffusion Processes: A Viscosity Approach}

\author{Antonio Ocello}

\affil{{\small Centre de Recherche en Économie et de Statistiques (CREST), Groupe ENSAE-ENSAI, ENSAE Paris, Institut Polytechnique de Paris, 91120 Palaiseau, France}}

             \date{\today}

\usepackage[colorinlistoftodos]{todonotes} 

\begin{document}

\maketitle

\begin{abstract}
    We study optimal control problems for interacting branching diffusion processes, a class of measure-valued dynamics capturing both spatial motion and branching mechanisms. From the perspective of the dynamic programming principle, we establish a rigorous connection between the control problem and an infinite system of coupled Hamilton--Jacobi--Bellman (HJB) equations, obtained through a bijection between admissible particle configurations and the disjoint topological union of countable Euclidean spaces.
    Under natural coercivity conditions on the cost functionals, we show that these growth conditions transfer to the value function and yield a viscosity characterization in the class of functions satisfying the same bounds. We further prove a comparison principle, which allows us to fully characterize the control problem through the associated HJB equation.
    Finally, we show that the problem simplifies in the mean-field regime, where the model coefficients exhibit symmetry with respect to the indices of the individuals in the population. This permutation invariance allows us to restrict attention to a reduced class of symmetric admissible controls, a reduction established by combining the viscosity characterization of the value function with measurable selection arguments.
\end{abstract}

\noindent \MSC{93E20, 60J60, 60J80, 35K10, 60J70}

\noindent \keywords{Stochastic control, branching diffusion processes, viscosity solutions, mean-field interactions, comparison principle}


\input{main_text/intro.tex}
\input{main_text/setting.tex}

\input{main_text/diff-characterization.tex}
\input{main_text/mean-field-setting.tex}
\input{main_text/examples.tex}
\input{main_text/conclusion.tex}

\paragraph{Acknowledgements.} This work is supported by Hi! PARIS and ANR/France 2030 program (ANR-23-IACL-0005). I am gratefully acknowledge Idris Kharroubi, Julien Claisse, Loïc Bethencourt, Étienne Tanré, and Rémi Catellier for many enriching discussions.

\appendix

\input{appendix/well-posed-ctrl-pb.tex}

\input{appendix/dpp.tex}

\input{appendix/verification-thm.tex}

\bibliographystyle{apalike}
\bibliography{main}

\end{document}

%% file: notation.tex
\def \I{\mathbb{I}}
\def \N{\mathbb{N}}
\def \R{\mathbb{R}}

\def \E{\mathbb{E}}
\def \F{\mathbb{F}}
\def \G{\mathbb{G}}

\def \d{\mathbf{d}}

\def \P{\mathbb{P}}

\def \D{\mathbb{D}}
\def \S{\mathbb{S}}

\def \1{\mathds{1}}
\def \0{\mathds{O}}

\def \Bc{\mathcal{B}}

\def \Dc{\mathcal{D}}
\def \Ec{\mathcal{E}}
\def \Fc{\mathcal{F}}
\def \Gc{\mathcal{G}}

\def \Ic{\mathcal{I}}

\def \Nc{\mathcal{N}}

\def \Pc{\mathcal{P}}

\def \Sc{\mathcal{S}}
\def \Tc{\mathcal{T}}

\def \Vc{\mathcal{V}}

\def \Hb{\textbf{H}}
\def \Lb{\textbf{L}}
\def \Sb{\textbf{S}}

\def \Pb{\textbf{P}}

\def \Igen{{\mathscr{P}_{\mathrm{adm}}(\Ic)}}
\def \rme{{\rm e}}

\newcommand{\restr}[2]{#1_{\mkern 1mu \vrule height 2ex\mkern2mu #2}}

\def \eps{\varepsilon}

\def \eqsp{\;}

\def \l{\left}
\def \r{\right}

\providecommand{\keywords}[1]
{
  \small	
  \textbf{Keywords---} #1.
}
\providecommand{\MSC}[1]
{
  \small	
  \textbf{MSC Classification---} #1.
}

\theoremstyle{plain}
\newtheorem{theorem}{Theorem}[section]
\newtheorem{proposition}[theorem]{Proposition}
\newtheorem{lemma}[theorem]{Lemma}
\newtheorem{corollary}[theorem]{Corollary}
\theoremstyle{definition}
\newtheorem{definition}[theorem]{Definition}

\theoremstyle{remark}
\newtheorem{remark}[theorem]{Remark}

\newcounter{hypH}
\newenvironment{hypH}{
    \refstepcounter{hypH}
    \begin{itemize}
    \item[{\bf H\arabic{hypH}}]
    }
{\end{itemize}}

\newcommand{\eg}{\textit{e.g.}}
\newcommand{\ie}{\textit{i.e.}}

\def \CtrlStandard{\Sc}

%% file: main_text/intro.tex
\section{Introduction}
\label{section:intro}

    Interacting particle systems are at the core of many models of collective dynamics. They are naturally described through measure-valued processes, which provide a flexible probabilistic framework to capture both individual behaviors and their aggregate effects. Such models have found applications across a wide range of disciplines, including biology and ecology \citep[see, e.g.,][]{champagnat2006unifying,champagnat2008individual}, genetics \citep{fleming1979some}, and finance \citep{carmona2013mean,grbac2025propagation}. Among them, branching diffusion processes stand out as a particularly rich class, since they combine the spatial dynamics of diffusion with stochastic birth-death mechanisms. Their controlled versions open new directions for both applied and theoretical investigations, as they connect measure-valued population dynamics with stochastic control theory. A related perspective is provided by the mean-field control (MFC) framework \citep{book:Carmona-Delarue_1,book:Carmona-Delarue_2}, which shares a similar spirit by modeling the collective behavior of large populations through their empirical distribution, under the assumption of anonymity and homogeneity in the interactions. 
    
    The study of controlled branching diffusion processes has received growing attention in the recent literature. Early work by \citet{Ustunel} introduced a weak control formalism for measure-valued branching processes. Later, \citet{Nisio} and \citet{claisse18} investigated the strong control problem, focusing on cost structures of product form with particle-wise dependence. More recently, \citet{kharroubi2024stochastic,kharroubi2024optimal} explored stochastic target problems and optimal stopping for branching diffusions. The general connection between control theory and measure-valued processes has also been highlighted in the recent contribution of \citet{cox2024controlled}, underlining the breadth of applications and the methodological challenges that arise in combining these two domains.
    
    In this work, we revisit the strong control problem for branching diffusion processes under general conditions on the interaction scheme.
    This perspective is consistent with recent developments in the literature on heterogeneous models and their scaling limits \citep{caines2021graphon,lacker2023label,de2024mean,coppini2025nonlinear,de2025linear}.
    Specifically, we allow the model coefficients to depend simultaneously on the index of the particle and on the empirical distribution of the population's index and location, thereby capturing the full generality of heterogeneous and non-symmetric interactions. Our analysis establishes the dynamic programming principle and derives the associated Hamilton--Jacobi--Bellman (HJB) system, formulated through a Euclidean bijection of the configuration space. We then prove that the value function admits a viscosity characterization, and we establish a comparison principle that fully characterizes the control problem via its HJB system.

    In addition, we study the mean-field setting, where the coefficients depend only on the population’s empirical distribution of positions and no longer on particle indices. In this regime, the HJB formulation together with measurable selection arguments allows us to establish invariance with respect to permutations of particle labels. This invariance naturally propagates to the optimizer: it is optimal to restrict attention to symmetric controls, namely controls that assign the same action to any two particles occupying the same position.

    Finally, when attempting to optimize trajectories, we focus on the spatial motion of the particles in the population, a concept that is naturally captured through the system's kinetic energy.
    This is the case of the Schrödinger bridge problem, as in \citet{follmer2006random}, where one seeks to identify the random evolution (\ie, a probability measure on path-space) that is closest to a prior Markov diffusion evolution in the relative entropy sense, while also satisfying certain initial and final marginals. It has been noted that this problem can be framed as a stochastic control problem \citep[see, \eg,][]{DaiPra:Pavon:SB_Ctrl_sto:1990,DaiPra:SB_Ctrl_sto:1991,chen2016relation,chen2021stochastic}, where the kinetic energy plays a fundamental role in the cost function. Continuing along this line of reasoning, we present an example involving a comparable cost function and proceed to solve it with the help of the verification theorem.
    
    In a companion paper \citep{ocello2023controlled}, we develop the relaxed formulation of this problem, showing its equivalence with the strong control setting under suitable assumptions on the coefficients; this formulation is a key step towards scaling limits, such as the superprocess limits studied in \citet{ocello2025controlled}.
    
    The paper is organized as follows. In \Cref{section:setting}, we introduce the model, assumptions, and moment estimates ensuring well-posedness of the cost functional. In \Cref{section:HJB}, we establish the dynamic programming principle, derive the HJB system via the Euclidean bijection, and prove both the viscosity characterization and the comparison principle. In \Cref{section:MF-regime}, we specialize to the mean-field regime, where permutation invariance of the coefficients allows us to reduce the analysis to symmetric controls. Finally, the appendices collect technical proofs, including details on the DPP and the verification theorem.

%% file: main_text/setting.tex
\section{Setting}
\label{section:setting}

\subsection{Notation}

\paragraph{Finite measures.}
For a Polish space $(\Ec,d)$ with $\Bc(\Ec)$ its Borelian $\sigma$-field, we write $C_b(\Ec)$ (resp. $C_0(\Ec)$) for the subset of the continuous functions that are bounded (resp. that vanish at infinity), and $M(\Ec)$ (resp. $\Pc(\Ec)$) for the set of Borel positive finite measures (resp. probability measures) on $\Ec$. We equip $M(\Ec)$ with weak* topology, \ie, the weakest topology that makes continuous the maps $M(\Ec)\ni\lambda\mapsto\int_\Ec \varphi(x) \lambda(dx)$, for $\varphi\in C_b(\Ec)$. We denote $\langle \varphi,\lambda\rangle := \int_\Ec \varphi(x) \lambda(dx)$, for $\lambda\in M(\Ec)$ and $\varphi\in C_b(\Ec)$.

Denote also by $M^1(\Ec)$ the subspace of measures with finite first order moment, \ie, the collection of all $\lambda\in M(\Ec)$ such that $\int_\Ec d(x,x_0)\lambda(dx)<\infty$, for some $x_0\in \Ec$. The weak* topology can be metrized in $M^1(\Ec)$ by the Wasserstein type metric $\d_{1,\Ec}$, as introduced in Appendix B of \citet{Claisse:Tan:MFbranching}. This means that, if $\partial$ is a cemetery point, we consider first $\bar \Ec$ the enlarged space $\bar \Ec := \Ec\cup \{\partial\}$. Defining $d(x,\partial):=d(x,x_0)+1$, we have that $(\bar \Ec,d)$ is Polish. For $m\in\R_+$, we consider the Wasserstein distance $\d_{1,\Ec,m}$, on the space $M^1_m(\bar \Ec)$ defined as
\begin{align*}
    M^1_m(\bar \Ec) := \{\lambda \in M^1(\bar \Ec) :  \lambda(\bar \Ec) = m\}\eqsp,
\end{align*}
as follows
\begin{align*}
    \d_{1,\Ec,m}(\lambda,\lambda') = \inf_{\pi\in\Pi(\lambda,\lambda')}\int_{\bar \Ec \times \bar \Ec}d(x,y)\pi(dx,dy)
    \eqsp, \quad \text{ for }\lambda,\lambda'\in M^1_m(\bar \Ec)\eqsp,
\end{align*}
with $\Pi(\lambda,\lambda')$ the collection of all non-negative measures on $\bar \Ec \times \bar \Ec$ with marginals $\lambda$ and $\lambda'$.
The distance $\d_{1,\Ec}$ on $M^1(\Ec)$ is now defined as
\begin{align*}
    \d_{1,\Ec}(\lambda,\lambda') = \d_{1,\Ec,m}\left(\bar\lambda_m,\bar\lambda'_m\right)\eqsp, \quad \text{ for }\lambda,\lambda'\in M^1_m(\Ec)\eqsp,
\end{align*}
with $m \geq \lambda(\Ec)\vee \lambda'(\Ec)$, $\bar\lambda_m(\cdot):=\lambda(\cdot\cap \Ec) + (m-\lambda(\Ec))\delta_\partial(\cdot)$,and $\bar\lambda'_m(\cdot):=\lambda'(\cdot\cap \Ec) + (m-\lambda'(\Ec))\delta_\partial(\cdot)$.
As proven in Lemma B.1 of \citet{Claisse:Tan:MFbranching}, this definition does not depend on the choice of $m$. Moreover, for some $x_0\in \Ec$, we have the natural bound
\begin{align}\label{eq:bound_d_p_E}
    \d_{1,\Ec}(\lambda,\delta_{x_0})\leq \int_\Ec d(x,x_0)\lambda(dx) + \langle1,\lambda\rangle, \quad \text{ for }\lambda\in M^1(\Ec)\eqsp.
\end{align}

Finally, we write $\Nc(\Ec)$ for the space of finite atomic measures on $Ec$, \ie,
\begin{align*}
    \Nc(\Ec) := \left\{\sum_{i=1}^m\delta_{x_i} ~:~m\in\N, x_i\in \Ec \text{ for }i\leq m\right\}\eqsp,
\end{align*}
a weakly* closed subset of $M(\Ec)$.

\paragraph{Label set.} We use Ulam--Harris--Neveu labelling to consider the genealogy of the particles.
Consider the set of labels 
\begin{align*}
    \Ic := \{\varnothing\}\cup\bigcup_{n=1}^{+\infty}\N^n\eqsp.
\end{align*}
Denote by $\varnothing$ the \textit{mother particle}, and $i=i_1\cdots i_n$ the multi-integer $i=(i_1,\ldots,i_n)\in\N^n$, $n\geq 1$.  For $i=i_1\cdots i_n\in \N^n$ and $j=j_1\cdots j_m\in \N^m$, we define their concatenation is $i j\in\N^{n+m}$ by $i j =i_1\cdots i_n j_1\cdots j_m$, and extend it to the entire $\Ic$ by $\varnothing i = i \varnothing=i$, for all $i\in\Ic$. When a particle $i=i_1\cdots i_n\in \N^n$ gives birth to $k$ particles, the off-springs are labelled $i0,\ldots,i(k-1)$. Moreover, if $\Vc\subset\Ic$ was the set of alive particles, after the branching event on the branch $i\in\Vc$, we have that the new set of alive particles become $\Vc^i_k$, with
\begin{align}
\label{eq:def:V_i_k}
    \Vc^i_k := \Vc\setminus\{i\}\cup \l\{i0,\dots,i(k-1)\r\}\eqsp.
\end{align}

Consider the partial ordering $\preceq$ (resp. $\prec$) by
\begin{align*}
    i\preceq j ~ \Leftrightarrow ~ \exists \ell\in\Ic~:~j=i\ell
    \qquad
    \left(\textrm{resp.}~i\prec j ~ \Leftrightarrow ~\exists \ell\in\Ic\setminus \{\varnothing\}~:~j=i\ell\right)\eqsp,
\end{align*}
for $i,j\in\Ic$.
We endow $\Ic$ with the discrete topology, generated by the distance
\begin{align*}
d^\Ic(i,j) := \sum_{\ell = p+1}^n (i_\ell +1) + \sum_{\ell' = p+1}^m (j_{\ell'} +1)\;,
\qquad\text{ for }
i=i_1\cdots i_{n}\in\N^n, \;j=j_1\cdots j_{m} \in \N^m\eqsp,
\end{align*}
with $p = \max\{\ell\geq1:i_\ell=j_\ell\}$ the generation of the greatest common ancestor. Denote $i\wedge j= i_0\cdots i_p$ and write $|i|:= d^\Ic(i,\varnothing)$, for $i\in\Ic$. Moreover, define the total ordering $\leq$ on $\Ic$ as $i\leq j$ if $i\preceq j$ or $i_{p+1}<j_{p+1}$.


From the definition of $\Vc^i_k$, note that not all possible combinations of indeces are considered when describing a population. Let $\Igen$ be the space of admissible configurations of indeces for a branching population to exists, defined as
\begin{align*}
    \Igen
    :=&\Big\{
        \Vc\eqsp:\eqsp \Vc\subseteq \Ic \text{ finite},\eqsp i\nprec  j,\text{ for }i,j\in\Vc
    \Big\}\eqsp.
\end{align*}
As $\Igen$ is a subset of $\mathscr{P}_{\mathrm{fin}}(\Ic)$ the set of all finite subsets of $\Ic$, it is a countable set.
For $\Vc\in\Igen$, denote $\mathfrak{S}_\Vc$, the set of permutations of $\Ic$ that send $\Vc$ to an admissible configuration in $\Igen$, defined as
\begin{align*}
    \mathfrak{S}_\Vc := 
    \bigl\{
        \mathfrak{s}\in \mathrm{Sym}(\Ic)\;:\;\mathfrak{s}\cdot \Vc \in \Igen
    \bigr\}\eqsp,
\end{align*}
where $\mathrm{Sym}(\Ic)$ is the permutation group of $\Ic$ and the action on subsets is $\mathfrak{s}\cdot \Vc \;:=\; \{\mathfrak{s}(i): i\in \Vc\}\subset \Ic$. Moreover, denote $\mathfrak{s}\cdot\lambda$ to be $\mathfrak{s}\cdot\lambda := \sum_{i\in\Vc}\delta_{(\mathfrak{s}(i),x_i)}$, for $\lambda = \sum_{i\in\Vc}\delta_{(i,x_i)}\in E$.

\paragraph{State and control space.} Take $E\subset \Nc(\Ic\times\R^d)$ as
\begin{align*}
    E:=&\l\{
        \sum_{i\in \Vc}\delta_{(i,x_i)}\eqsp : \eqsp \Vc\in \Igen,\; x_i\in\R^d
    \r\}\eqsp.
\end{align*}
Note that $\Nc(\R^d)$ is a closed set of $M^1(\R^d)$ with respect to the distance $\d_{1,\R^d}$. This is due to the fact that $\Nc(\R^d)$ is weakly*-closed and, from Lemma B.2 in \citet{Claisse:Tan:MFbranching}, convergence in $M^1(\R^d)$ entails weak*-convergence to some $\lambda\in \Nc(\R^d)\subseteq M^1(\R^d)$. Therefore, combining this with the fact the $E$ is weakly*-closed \citep[see, $e.g$, Proposition A.7,][]{kharroubi2024stochastic} and $\Ic$ is equipped with discrete topology, we have that $E$ is also a closed set of $M^1(\Ic\times \R^d)$.

Define now this projection map $\pi:E\to \Nc(\R^d)$ 
as
\begin{align*}
    \pi:E\ni \sum_{i\in \Vc}\delta_{(i,x_i)}\mapsto \sum_{i\in \Vc}\delta_{x_i}
    \eqsp.
\end{align*}
Fix $\lambda=\sum_{i\in \Vc}\delta_{(i,x_i)},\lambda^\prime=\sum_{i\in \Vc}\delta_{(i,y_i)}\in E$. Using the characterisation of the distance $\d_{1,\Ic\times \R^d}$ of Lemma B.1 in \citet{Claisse:Tan:MFbranching}, we obtain
\begin{align}
\label{eq:bound_d_1_Rd-measures-to_R_d}
    \d_{1,\Ic\times \R^d}\left(\lambda,\lambda^\prime\right) = 
    \sup_{\varphi\in\text{Lip}^0_1(\Ic\times \R^d)}\sum_{i\in\Vc}\left|\varphi(i,x_i) - \varphi(i,y_i) \right| \leq \sum_{i\in\Vc}|x_i-y_i| = \|\Vec{x}_\Vc-\Vec{y}_\Vc\|_{1,d|\Vc|}
    \eqsp,
\end{align}
where $\Vec{x}_\Vc=(x_i)_{i\in \Vc}$ is the vector of $\R^{d|\Vc|}$ taken in the order induced by the total ordering $\leq$ on $\Ic$,
$\text{Lip}^0_1(\R^d)$ denote the collection of all functions $\varphi : \Ic\times \R^d\to \R$ with Lipschitz constant smaller or equal to $1$ and such that $\varphi(0) = 0$, and $\|\cdot\|_{1,n}$ denotes the $L^1$-distance in $\R^{n}$, for $n\in\N$.
Using Cauchy--Schwarz inequality, we can also bound the distance $\d_{1,\Ic\times \R^d}$ by
\begin{align}
\label{eq:bound_d_1_Rd-measures-to_R_d-2}
    \d_{1,\R^d}\left(\lambda,\lambda^\prime\right) \leq \sqrt{|\Vc|}\eqsp \|\Vec{x}_\Vc-\Vec{y}_\Vc\|_{2,d|\Vc|}
    \eqsp,
\end{align}
where $\|\cdot\|_{2,n}$ denotes the $L^2$-distance in $\R^{n}$, for $n\in\N$.

Let $\Tc_{t,s}$ denotes the
collection of all stopping times valued in $[t,s]$.
Take the set of actions $A$ to be a closed subset of an Euclidean space.

\paragraph{Càdlàg paths.}
Denote by $\D([0,T];E)$ the space of càdlàg, right continuous with left limits, functions from $[t, +\infty)$ to $E$, equipped with the Skorokhod topology $d_{E}$ associated with the metric $\d_{1,\Ic\times \R^d}$, which makes it complete \citep[see, \eg,][]{billingsley2013convergence}.
    
\subsection{Branching diffusion processes}
\label{Section:strong_form}
Fix a finite time horizon $T > 0$. 
Let $(\Omega,\Fc,\P)$ be a probability space supporting two independent families $\{W^i\}_{i\in \Ic}$ and $\{Q^i\}_{i\in \Ic}$ of mutually independent processes. Let $W^i$ be a $d^\prime$-dimensional Wiener processes, and $Q^i(ds dz)$ a Poisson random measure on $[0,T]\times\R_+$ with intensity measure $dsdz$.
Let $\F=\{\Fc_t\}_{t\geq 0}$ be the filtration generated by these processes, \ie, the (right-continuous) completion of the $\sigma$-algebra $\G=\{\Gc_t\}_{t\geq 0}$ with
\begin{align*}
    \Gc_t := \sigma\l(W^ i_s, Q^i([0,s]\times C)~:~ s \leq t,~ i\in \Ic,~ C\in \Bc(\R_+)\r)
    \eqsp.
\end{align*}
Moreover, let $\Fc_\infty$ (resp. $\Gc_\infty$) be the $\sigma$-algebra generated by $\bigcup_{t\geq 0}\Fc_t$ (resp. $\bigcup_{t\geq 0}\Gc_t$).

Consider the following parameter of models
\begin{align*}
(b,\sigma, \gamma, p_k
)\eqsp :\eqsp \Ic\times \R^d\times E \times A \to \R^d \times \R^{d\times d^\prime}\times \R_+ \times [0,1]\eqsp,
\end{align*}
for $k\geq0$, such that $\sum_{k \geq 0}p_k(i,x,\lambda,a)=1$, for $(i,x,\lambda,a)\in\Ic\times \R^d\times E\times A$.
Let $\Phi$ be the generating function of $(p_k)_k$, \ie,
\begin{align*}
    \Phi(s,i,x,\lambda,a) = \sum_{k=0}^\infty p_k(i,x,\lambda,a) s^k\eqsp, \quad \text{ for }(s,i,x,\lambda,a)\in [0,1]\times \Ic\times\R^d\times E\times A\eqsp.
\end{align*}
We now introduce the following assumptions on these parameters.

\begin{hypH}
\label{hypH:model_parameters}
    \begin{enumerate}[(i)]
        \item Suppose that $b$ and $\sigma$ are Lipschitz continuous in $(x,\lambda)$ uniformly in $(i,a)$, \ie, there exists $L>0$ such that
        \begin{align}
        \label{eq:bound_b_sigma_Lipschitz}
            \left|b(i,x,\lambda,a)-b(i,x',\lambda',a)\right|+\left|\sigma(i,x,\lambda,a)-\sigma(i,x',\lambda',a)\right|\leq
        L (\|x-x'\|_{2,d}+\d_{1,\Ic\times \R^d}(\lambda,\lambda'))\eqsp,
        \end{align}
        for $x,x'\in\R^d$, $\lambda,\lambda'\in E$, and $a\in A,i\in\Ic$.
        \item Suppose that $\sigma$ and $\gamma$ are uniformly bounded, and $b$ has linear growth in $(x,a)$ while bounded in $(\lambda,i)$, \ie, there exists $C_\sigma, C_\gamma, C_b >0$ such that
        \begin{align}\label{eq:bound_b_sigma_gamma}
            \left|b(i,x,\lambda,a)\right|\leq C_b (1+ |x| + |a|)\eqsp,\qquad
            \left|\sigma(i,x,\lambda,a)\right|\leq C_\sigma\eqsp,\qquad
            \gamma(i,x,\lambda,a)\leq C_\gamma\eqsp,
        \end{align}
        for $(i,x,\lambda,a)\in \Ic\times\R^d\times E \times A$.
        \item Suppose that the first and second order moments related to $(p_k)_k$ are uniformly bounded, \ie, there exist two constants $C^1_\Phi, C^2_\Phi>0$ such that
        \begin{align}
        \label{eq:bound:order1_2_Phi}
            \begin{split}
                \partial_s \Phi(1,i,x,\lambda,a)&=\sum_{k\geq1}k p_k(i,x,\lambda,a)\leq C^1_\Phi\eqsp,
                \qquad
                \partial^2_{ss} \Phi(1,i,x,\lambda,a)=\sum_{k\geq1}k(k-1) p_k(i,x,\lambda,a) \leq C^2_\Phi\eqsp,
            \end{split}
        \end{align}
        for $(i,x,\lambda,a)\in \Ic\times\R^d\times E\times A$.
    \end{enumerate}
\end{hypH}

The extension to time-dependent coefficients is straightforward and is not addressed explicitly here in order to avoid heavier notation. This setting will be used later in~\Cref{Section:LQ}.


\begin{definition}[Admissible control]
    We say that $\beta = (\beta^i)_{i\in \Ic}$ is a \emph{admissible control}, and we denote $\beta\in\CtrlStandard$, if $\beta$ is an $\G$-predictable process valued in $A^\Ic$, such that
    \begin{align}\label{eq:bound_sup_beta_2}
        \E\left[\int_t^T \sup_{i\in\Ic} |\beta_s^i|^2 ds\right]<\infty\eqsp.
    \end{align}
\end{definition}

Fix an initial condition $(t,\lambda)\in[0,T]\times E$ and an admissible control $\beta = (\beta^i)_{i\in\Ic}\in\CtrlStandard$. We describe the \textit{controlled branching diffusion} $\xi^{t,\lambda;\beta}$ as the measure-valued process
\begin{align*}
    \xi^{t,\lambda;\beta}_s = \sum_{i\in \Vc^{t,\lambda;\beta}_s}\delta_{(i,Y^{i,\beta}_s)}\eqsp,
\end{align*}
where $Y^{i,\beta}_s$ is the position of the member with label $i\in\Ic$, and $\Vc^{t,\lambda;\beta}_s$ the set of alive particles at time $s$.
This process takes values in $E$ and the behaviour of each alive particle $i$ is characterized by the following three properties:
\begin{itemize}
    \item \textit{Spatial motion}: during its lifetime, it moves in $\R^d$ according to the following stochastic differential equation
    \begin{align*}
        dY^{i,\beta}_s = b\left(i,Y^{i,\beta}_s,\xi^{t,\lambda;\beta}_s,\beta_s^i\right)ds + \sigma\left(i,Y^{i,\beta}_s,\xi^{t,\lambda;\beta}_s,\beta_s^i\right)dW^i_s\eqsp;
    \end{align*}
    \item \textit{Branching rate} $\gamma$: given a position $Y^{i,\beta}_s$ at time $s$, conditionnally to $\Fc_s$, the probability it dies in the time interval $[s,s+\delta s)$ is $\gamma(i,Y^{i,\beta}_{s-},\xi^{t,\lambda;\beta}_{s-},\beta_s^i)\delta s + o(\delta s)$.
    \item \textit{Branching mechanism}: when it dies at a time $s$, conditionnally to $\Fc_s$, it leaves behind (at the location where it died) a random number of offspring with probability $(p_k(i,Y^{i,\beta}_{s-},\xi^{t,\lambda;\beta}_{s-},\beta_{s}^i))_{k\in\N}$.
\end{itemize}
If the control is constant, \ie, we are in the uncontrolled setting, conditionally on time and place of birth, offspring evolve independently of each other.

We emphasize that the dependence of the model's parameters on both the individual particle index $i$ and the empirical measure $\xi^\beta$ represents the most general form of branching interactions. This framework encompasses fully non-symmetric and particle-specific dynamics, allowing for highly heterogeneous systems. As a result, it generalizes many classical models that assume exchangeability or symmetry among particles, and captures a broader range of real-world applications where individuals may behave differently based on their position or role in the population.

Let $L$ be the generator (associated with the spatial motion of each particle) defined on $\varphi\in C^2_b(\Ic\times\R^d)$ as
\begin{align*}
    L \varphi (i,x,\lambda,a) = b(i,x,\lambda,a)^\top D \varphi(i,x) + \frac{1}{2}\text{Tr}\left(
    \sigma\sigma^\top(i,x,\lambda,a)D^2 \varphi(i,x)
    \right)\eqsp,
\end{align*}
where $D\varphi(i,\cdot)$ and $D^2\varphi(i,\cdot)$ denote gradient and Hessian of the function $\varphi(i,\cdot)$, for a fixed index $i\in\Ic$.
A possible representation of previous properties is given by the following SDE 
\begin{align}
\label{SDE:strong}
    \begin{split}
        \langle\varphi,\xi^{t,\lambda;\beta}_s\rangle =&~ \langle\varphi,\lambda\rangle+
        \int_t^s\sum_{i\in \Vc^{t,\lambda;\beta}_u}D \varphi(i,Y^{i,\beta}_u)^\top\sigma\left(i,Y^{i,\beta}_u,\xi^{t,\lambda;\beta}_u,\beta^i_u\right)dB^i_u
        +
        \int_t^s\sum_{i\in \Vc^{t,\lambda;\beta}_u}L \varphi\left(i,Y^{i,\beta}_u,\xi^{t,\lambda;\beta}_u,\beta^i_u\right)du \\
        & + \int_{(t,s]\times\R_+}\sum_{i\in \Vc^{t,\lambda;\beta}_{u-}}\sum_{k\geq0}
        (k-1)\varphi(i,Y^{i,\beta}_{u-}) \1_{I_k\left(i,Y^{i,\beta}_{u-},\xi^{t,\lambda;\beta}_{u-},\beta^i_{u}\right)}(z)
        Q^i(dudz)\eqsp,
    \end{split}
\end{align}
with
\begin{align*}
    I_k(i,x,\lambda,a) = \Bigg[\gamma(i,x,\lambda,a)\sum_{\ell=0}^{k-1}p_\ell(i,x,\lambda,a),\gamma(i,x,\lambda,a)\sum_{\ell=0}^{k}p_\ell(i,x,\lambda,a)\Bigg)\eqsp,
\end{align*}
for all $(i,x,\lambda,a)\in \Ic\times\R^d\times E \times A$, $k\in\N$, with the value of an empty sum being zero by convention. Notice that $(I_k(i,x,\lambda,a))_{k\in \N}$ forms a partition of the interval $[0, \gamma(i,x,\lambda,a))$.

The evolution of the piecewise constant process \((\Vc^{t,\lambda;\beta}_s)_{s\geq t}\) follows the same structural dynamics as those introduced in \citet{claisse18}. In particular, the process evolves through successive branching and interaction events that preserve the admissibility of the configuration. An explicit construction of this evolution is provided in the proof of \Cref{prop:existence_strong_branching} in \Cref{Appendix:Proof:existence_strong_branching}, where the well-posedness of the process under our generalized framework is established.

\paragraph{Existence and moment estimates.}
We now prove the existence of controlled branching diffusions under admissible controls. We also provide bounds on their moments, which are crucial for the well-posedness of the control problem defined in~\Cref{subsection:strong_form}. The proof of this result follows the same lines as Proposition 2.1 from~\citet{claisse18} and is deferred to~\Cref{Appendix:Proof:existence_strong_branching}.

\begin{proposition}
\label{prop:existence_strong_branching}
    Let $(t,\lambda) \in [0,T]\times E$ and $\beta\in\CtrlStandard$. Suppose Assumption H\ref{hypH:model_parameters} holds. Then, there exists a unique (up to indistinguishability) càdlàg and adapted process $(\xi^{t,\lambda;\beta}_s)_{s\geq t}$ satisfying~\eqref{SDE:strong} such that $\xi^{t,\lambda;\beta}_t=\lambda$. In addition, there exists a constant $C>0$ depending only on $T$ and on the coefficients $b$, $\sigma$, $\gamma$, and $(p_k)_k$ such that, for $h>0$,
    \begin{align}
        \label{eq:non-explosion-moment1_mass}
        \E\left[\sup_{u\in[t,t+h]}|\Vc^{t,\lambda;\beta}_u|\right]\leq&~ \langle1,\lambda\rangle~ e^{C_\gamma C^1_\Phi h}\eqsp,
        \\
        \label{eq:non-explosion-moment2_mass}
        \E\left[\sup_{u\in[t,t+h]}|\Vc^{t,\lambda;\beta}_u|^2\right]\leq&~ \langle1,\lambda\rangle^2~ e^{C_\gamma (C^1_\Phi+C^2_\Phi) h}\eqsp,
        \\
        \label{eq:non-explosion-moment1_control}
        \E\left[\int_t^{t+h}\sum_{i\in \Vc^{t,\lambda;\beta}_{u}}|\beta^i_u|du\right]\leq&~ C\eqsp,
        \\
        \label{eq:non-explosion-moment1_population} 
        \E\left[\sup_{u\in[t,t+h]}\sum_{i\in \Vc^{t,\lambda;\beta}_u}\left|Y^{i,\beta}_u\right|\right]\leq&~ 
        C \Bigg(\sum_{i\in \Vc}|x^i| +\E\left[ \int_t^{t+h} |\Vc^{t,\lambda;\beta}_{u}|du \right]
        +\E\left[ \int_t^{t+h} \sum_{i\in \Vc^{t,\lambda;\beta}_{u}}\left|\beta^{i}_{u}\right|du \right] \Bigg)
        \eqsp.
    \end{align}
\end{proposition}

\subsection{Control problem}
\label{subsection:strong_form}

Let $\psi: \Ic\times \R^d\times E\times A \to \R$ and $\Psi: E \to \R$ be continuous functions, and consider the following assumption.

\begin{hypH}
\label{hypH:coercivity_hyp}
    Suppose that there exists $C_\Psi, c_\psi >0$ such that
    \begin{gather}
        - C_\Psi \left(1+\int_{\R^d}|y| \lambda(dy) + \langle 1\eqsp,\lambda\rangle\right)
        \leq
        \Psi(\lambda)
        \leq
        C_\Psi \left(1+\int_{\R^d}|y|^2 \lambda(dy) + \langle 1,\lambda\rangle^2 \right)\eqsp,
        \label{eq:coercivity_hyp:big_Psi}
        \\
        -C_\Psi \left(1+|x| \right)  + c_\psi|a|^2
        \leq
        \psi(i,x,\lambda, a )
        \leq
        C_\Psi \left(1+|x|^2 + 
        |a|^2 \right)
        \eqsp,
        \label{eq:coercivity_hyp:psi}
    \end{gather}
    for $(i,x,\lambda, a )\in \Ic\times \R^d\times E\times A$.
\end{hypH}

Fix an admissible control $\beta\in\CtrlStandard$ and a starting condition $(t,\lambda) \in[0,T]\times E$. The cost and  value functions are defined as follows:
\begin{align}
\label{eq:def:cost_function}
    J(t,\lambda;\beta) := \E\left[
        \int_t^T \sum_{i\in \Vc^{t,\lambda;\beta}_s} \psi\left(i,Y^{i,\beta}_s,\xi^{t,\lambda;\beta}_s,\beta_s^i\right)ds +  \Psi\left(\xi^{t,\lambda;\beta}_T\right)\Bigg|\xi^{t,\lambda;\beta}_t=\lambda
    \right]
    \qquad
    \text{ and }
    \qquad
    v(t,\lambda) := \inf_{\beta\in\CtrlStandard} J(t,\lambda;\beta)
    \eqsp.
\end{align}

\paragraph{Well-posedness of the control problem.}
To establish the well-posedness of the control problem~\eqref{eq:def:cost_function}, it remains to prove the finiteness of the second moment of the branching processes, at least near an optimal value. To this end, we apply similar techniques used to prove~\Cref{prop:existence_strong_branching} in the proof of the following lemma, which is deferred to~\Cref{Appendix:Proof:bound_moment2_population}.

\begin{lemma}
\label{Lemma:bound_moment2_population}
    Let $(t,\lambda) \in [0,T]\times E$ and $\beta\in\CtrlStandard$. Suppose Assumption H\ref{hypH:model_parameters}-H\ref{hypH:coercivity_hyp} hold. Then, there exists a constant $C>0$ depending only on $T$ and on the coefficients $b$, $\sigma$, $\gamma$, and $(p_k)_k$ such that, for $h>0$,
    \begin{align}
    \label{eq:non-explosion-moment2_population}
        \E\left[\sup_{u\in[t,t+h]}\sum_{i\in \Vc^{t,\lambda;\beta}_u}\left|Y^{i,\beta}_u\right|^2\right]\leq&~ 
        C \Bigg(\langle|\cdot|^2,\lambda\rangle +\E\left[ \int_t^{t+h} |\Vc^{t,\lambda;\beta}_{u}|du \right]
        +\E\left[ \int_t^{t+h} \sum_{i\in \Vc^{t,\lambda;\beta}_{u}}\left|\beta^{i}_{u}\right|^2 du \right] \Bigg)\eqsp.
    \end{align}
\end{lemma}

This lemma shows that if $\E\left[ \int_t^{T} \sum_{i\in \Vc^{t,\lambda;\beta}_{u}}\left|\beta^{i}_{u}\right|^2 du \right]<\infty$, then $|J(t,\lambda;\beta)|< \infty$ by coercivity bounds. This condition indicates that any control close to optimal must satisfy it, as demonstrated in the following proposition.

\begin{proposition}
\label{prop:bound_eps_opt_control}
    Fix $(t,\lambda) \in [0,T]\times E$. Let $\eps>0$, and let $\CtrlStandard^{\eps}_{(t,\lambda)}$ be the set of $\beta\in\CtrlStandard$ satisfying
    \begin{align*}
        J(t,\lambda;\beta) \leq v(t,\lambda)+ \eps\eqsp.
    \end{align*}
    Then 
    \begin{align}\label{eq:bound_eps_opt_control}
        \sup_{\beta\in \CtrlStandard^{\eps}_{(t,\lambda)}} \E\left[ \int_t^{T} \sum_{i\in \Vc^{t,\lambda;\beta}_{u}}\left|\beta^{i}_{u}\right|^2 du \right]<\infty\eqsp.
    \end{align}
    Moreover, $v(t,\lambda)>-\infty$.
\end{proposition}
\begin{proof}
    We use the l.h.s.\ of~\eqref{eq:coercivity_hyp:big_Psi} and~\eqref{eq:coercivity_hyp:psi} along with~\Cref{Lemma:bound_moment2_population} to find a constant $C > 0$ (which may change from line to line) such that, for all $\beta \in \CtrlStandard$,
    \begin{align}
        J(t,\lambda;\beta)
        \geq&\eqsp
        -C \E\left[1 + \sup_{u\in[t,T]}|V_u|^2 +\sup_{u\in[t,T]}\sum_{i\in \Vc^{t,\lambda;\beta}_{u}}\left|Y^{i,\beta}_{u}\right|\right] + c_\psi \E\left[ \int_t^{T} \sum_{i\in \Vc^{t,\lambda;\beta}_{u}}\left|\beta^{i}_{u}\right|^2 du \right]
        \nonumber
        \\
        \geq&\eqsp
        -C \E\left[1 + \int_t^{T} \sum_{i\in \Vc^{t,\lambda;\beta}_{u}}\left|\beta^{i}_{u}\right| du\right]
        + c_\psi \E\left[ \int_t^{T} \sum_{i\in \Vc^{t,\lambda;\beta}_{u}}\left|\beta^{i}_{u}\right|^2 du \right]
        \eqsp.
        \label{eq:bound_a-C_a^2}
    \end{align}
    This already proves $v(t,\lambda)>-\infty$, as the function $a \mapsto c_\psi|a|^2- C|a|$ is bounded from below. To prove the first claim, fix arbitrarily a constant control $\beta^{a_0,i}_s :=a_0 \in A$.~\Cref{Lemma:bound_moment2_population} and~\Cref{prop:existence_strong_branching} imply
    \begin{align*}
        \E\left[\sup_{u\in[t,t+h]}\sum_{i\in \Vc^{t,\lambda;\beta}_u}\left|Y^{i,\beta^{a_0}}_u\right|^2\right]
        \leq&\eqsp 
        C \left(1 + \E\left[ \int_t^{t+h} \sum_{i\in \Vc^{t,\lambda;\beta}_{u}}\left|\beta^{a_0,i}_{u}\right|^2 du \right] \right) \leq C \left(1 + \left|a_0\right|^2 \right)
        \eqsp.
    \end{align*}
    Then, from the r.h.s.\ of~\eqref{eq:coercivity_hyp:big_Psi} and~\eqref{eq:coercivity_hyp:psi}, we have show $J(t,\lambda;\beta^{a_0})<\infty$. Therefore, for  $\beta \in \CtrlStandard^{\eps}_{(t,\lambda)}$, we have $J(t,\lambda;\beta)\leq J(t,\lambda;\beta^{a_0})+\eps$. Combining this with~\eqref{eq:bound_a-C_a^2}, it yields
    \begin{align}
    \label{eq:uniform-bound-beta}
        \sup_{\beta\in \CtrlStandard^{\eps}_{(t,\lambda)}} \E\left[ \int_t^{T} \sum_{i\in \Vc^{t,\lambda;\beta}_{u}}\left(\left|\beta^{i}_{u}\right|^2- C \left|\beta^{i}_{u}\right|\right) du \right]
        <
        \infty\eqsp,
    \end{align}
   which implies~\eqref{eq:bound_eps_opt_control}, by~\Cref{prop:existence_strong_branching}.
\end{proof}

%% file: main_text/diff-characterization.tex
\section{Differential characterization of the control problem}
\label{section:HJB}
\subsection{Dynammic programming principle (DPP)}
A key property for deriving the differential characterization of the control problem \eqref{eq:def:cost_function} is the \emph{Dynamic Programming Principle (DPP)}. This principle asserts that the value function of the optimal cost can be expressed in terms of controlled subproblems, enabling a recursive formulation of the original problem.

The proof of such a result typically relies on a \emph{pseudo-Markov property} together with \emph{measurable selection} result, which allow the application of the DPP in settings with controlled branching dynamics. The property—developed in \citet{claisse2016pseudo} for controlled diffusions and later extended in \citet{claisse18-v1} to branching diffusions—ensures that the system can be ``restarted'' at any stopping time using a control that is independent of the past. The second one allows to construct a measurable function that selects the control from a set of admissible controls, ensuring that the control can be applied at any time an dits proof follows the lines of \citet{kharroubi2024stochastic}.
These two steps are crucial for constructing $\varepsilon$-optimal controls and establishing the DPP.


We provide this result while postponing to \Cref{Appendix:DPP} the technical developments. These are adapted to our setting from the results in \citet{claisse18-v1} and \citet{kharroubi2024stochastic}.

\begin{proposition}[DPP]
\label{Prop:DPP}
    Let $(t,\lambda)\in [0,T]\times E$ and $\beta\in\CtrlStandard$. Suppose Assumption H\ref{hypH:model_parameters}-H\ref{hypH:coercivity_hyp} holds. Fix $\tau\in\Tc_{t,T}$. Then,
    \begin{align}
    \label{eq:DPP}
        v(t,\lambda) = \inf_{\beta\in\CtrlStandard} \E\left[
            \int_t^\tau \sum_{i\in \Vc^{t,\lambda;\beta}_s} \psi\left(i,Y^{i,\beta}_s,\xi^{t,\lambda;\beta}_s,\beta^i_s\right)ds +  v\left(\tau, \xi^{t,\lambda;\beta}_\tau \right)\middle|\xi^{t,\lambda;\beta}_t=\lambda
        \right]\eqsp.
    \end{align}
\end{proposition}
\begin{proof}
    \emph{Step 1: Upper bound.}  
    Let $\varepsilon > 0$. By definition of the value function, there exists an $\varepsilon$-optimal control ${}^\varepsilon\beta \in \CtrlStandard$ such that
    \begin{align*}
        v(t,\lambda) + \varepsilon \geq J\l(t,\lambda;{}^\varepsilon\beta\r).
    \end{align*}
    From \Cref{corollary:App-DPP:pseudo_markov}, we get
    \begin{align*}
        v(t,\lambda) + \varepsilon
        \geq&~
        \int_\Omega \left(
            \int_t^{\tau(\omega)} 
            \sum_{i\in \Vc^{t,\lambda;{}^\varepsilon\beta}_s(\omega)}
            \psi\left(
                i,Y^{i,{}^\varepsilon\beta}_s(\omega),\xi^{t,\lambda;{}^\varepsilon\beta}_s(\omega),({}^\varepsilon\beta^{\tau(\omega), \omega})^i_s
            \right)ds + J\left(
                \tau(\omega),\xi^{t,\lambda;{}^\varepsilon\beta}_\tau(\omega);
                {}^\varepsilon\beta^{\tau(\omega), \omega}
            \right)
        \right)\P(d\omega)
        \\
        \geq&~
        \E\left[ 
            \int_t^{\tau} 
            \sum_{i\in \Vc^{t,\lambda;{}^\varepsilon\beta}_s}
            \psi\left(
                i,Y^{i,{}^\varepsilon\beta}_s,\xi^{t,\lambda;{}^\varepsilon\beta}_s,({}^\varepsilon\beta)^i_s
            \right)ds
            + v\left(\tau, \xi^{t,\lambda;{}^\varepsilon\beta}_\tau \right)
        \right]\eqsp.
    \end{align*}

    \emph{Step 2: Lower bound.} We now address the converse inequality. The standard strategy consists in constructing an admissible control on $[t,T]$ by gluing together any given control on $[t,\tau]$ with suitably chosen controls after $\tau$.

    Denote $\hat{v}(t,\lambda)$ the r.h.s.\ of~\eqref{eq:DPP}. Let now ${}^\varepsilon\beta \in \CtrlStandard$ be such that
    \begin{align}
    \label{eq:DPP:upper_bound_hat_v}
        \E\left[
            \int_t^\tau \sum_{i\in \Vc^{t,\lambda;{}^\varepsilon\beta}_s} \psi\left(i,Y^{i}_s,\xi^{t,\lambda;{}^\varepsilon\beta}_s,{}^\varepsilon\beta^i_s\right)ds +  v\left(\tau, \xi^{t,\lambda;{}^\varepsilon\beta}_\tau \right)\middle|\xi^{t,\lambda;{}^\varepsilon\beta}_t=\lambda
        \right] \leq \hat{v}(t,\lambda) + \varepsilon\eqsp.
    \end{align}
    Consider the probability measure $\nu$ induced on $[0,T]\times E$ by $\omega\mapsto (\tau(\omega), \xi^{t,\lambda;{}^\varepsilon\beta}_\tau(\omega))$ and let $\phi_\nu$ be the Borel-measurable function provided by \Cref{prop:App-DPP:measurable_selection}. From the definition of $\mathcal{U}_\varepsilon$, we have
    \begin{align}
    \label{eq:DPP:upper_bound}
        v\left(\hat{t},\hat{\lambda}\right) + \varepsilon \geq J\left(\hat{t}, \hat{\lambda}; \phi_\nu(\hat{t}, \hat{\lambda})\right)\eqsp,
    \end{align}
    for all $(\hat{t}, \hat{\lambda}) \in [0,T]\times E\setminus N^\nu$, where $N^\nu$ is a negligible set with respect to $\nu$. 
    Define $\Theta:=(\tau, \xi^{t,\lambda;{}^\varepsilon\beta}_\tau)$
    and $N := \Theta^{-1}\l(N^\nu\r)$.
    Using \eqref{eq:DPP:upper_bound}, we get
    \begin{align*}
        &v\left(\Theta(\omega)\right) + \varepsilon
        \\
        &\geq
        J\left(\Theta(\omega); \phi_\nu(\Theta(\omega))\right)
        \\
        &=
        \int_\Omega
        \left(
            \int_{\tau(\omega)}^T
            \sum_{i\in \Vc^{\Theta(\omega); \phi_\nu(\Theta(\omega))}_s(\omega^\prime)}
            \psi\left(
                i,\eqsp
                Y^{i,\phi_\nu(\Theta(\omega))}_s(\omega^\prime),\eqsp
                \xi^{\Theta(\omega); \phi_\nu(\Theta(\omega))}_s(\omega^\prime),\eqsp
                \phi_\nu(\Theta(\omega))^i_s(\omega^\prime)
            \right)ds + g\left(
                \xi^{\Theta(\omega); \phi_\nu(\Theta(\omega))}_T(\omega^\prime)
            \right)
        \right)\P(d\omega^\prime)
        %
        \eqsp,
    \end{align*}
    for $\omega\in\Omega\setminus N$.
    Since $\phi_\nu(\Theta)$ is independent of $\Fc_\tau$, we can apply the pseudo-Markov property from \Cref{lemma:App-DPP:pseudo_markov} together with \eqref{eq:DPP:upper_bound_hat_v} to obtain
    \begin{align*}
        \hat{v}(t,\lambda) + 2\varepsilon\geq
        \E\left[
            \int_t^T \sum_{i\in \Vc^{t,\lambda;{}^\varepsilon\tilde\beta}_s} \psi\left(i,Y^{i}_s,\xi^{t,\lambda;{}^\varepsilon\tilde\beta}_s,{}^\varepsilon\tilde\beta^i_s\right)ds +  g\left(\xi^{t,\lambda;{}^\varepsilon\tilde\beta}_T \right)\middle|\xi^{t,\lambda;{}^\varepsilon\tilde\beta}_t=\lambda
        \right]
        \geq v(t,\lambda)
        \eqsp,
    \end{align*}
    with ${}^\varepsilon\tilde\beta$ the following control
    \begin{align*}
        {}^\varepsilon\tilde{\beta}^i_s(\omega)
        :=
        {}^\varepsilon\beta^i_s(\omega) \eqsp\1_{[t,\tau(\omega))}+
        \phi_\nu^i(\Theta(\omega))(\omega)\1_{[\tau(\omega),T]}
        \eqsp,\qquad \P\text{--a.s.},
    \end{align*}
    for $s \in [t,T]$ and $i \in \Ic$.
\end{proof}

\subsection{Verification theorem}
\label{Section:verification_theorem}
Before establishing a verification theorem based on the differential characterization of the control problem, we first present a version that relies solely on \emph{martingality conditions}. This approach is standard in the control literature \citep[see, \eg, Lemma 2.1 of][]{Pham:Conditional_LQ} and provides a more general and potentially less technical route for validating candidate value functions. Specifically, the goal is to identify a suitably regular function of the form $(t, \lambda) \mapsto \varphi(t, \lambda)$ such that, when applying the semimartingale decomposition given by~\eqref{SDE:strong}, the finite variation term in the corresponding Itô-type formula is nonnegative for all admissible controls $\beta \in \CtrlStandard$, and vanishes for at least one control $\bar{\beta}$. The proof of the following is deferred to \Cref{Appendix:verification-thm}.


\begin{proposition}\label{Prop:verification2}
    Let $w\in C^0([0,T]\times E)$ such that there exists a constant $C_w>0$ such that
    \begin{align}\label{eq:verification_thm2:growth_w}
        -C_w\left(1+\langle1,\lambda\rangle+\langle|\cdot|,\lambda\rangle\right) \leq w(t,\lambda)\leq C_w\left(1+\langle1,\lambda\rangle^2+\langle|\cdot|^2,\lambda\rangle\right)
        \eqsp,
        \qquad\text{ for } (t,\lambda)\in [0,T]\times E
        \eqsp.
    \end{align}
    Fix $(\bar{t},\bar \lambda)\in [0,T]\times E$. If we have that
    \begin{enumerate}[(i)]
        \item $w (T,\lambda) = \Psi(\lambda)$, for $\lambda \in E$;
        \item $
            \left\{
                w\l(s,\xi^{\bar{t},\bar \lambda;\beta}_s\r)
                +\displaystyle
                \int_{\bar{t}}^s \sum_{i\in \Vc^{\bar{t},\bar \lambda;\beta}_u}
                \psi\l(i,Y^{i,\beta}_u,\xi^{\bar{t},\bar \lambda;\beta}_u,\beta^i_u\r)du\eqsp:\eqsp s\in \l[\bar{t},T\r]
            \right\}$  is a $\P$-local submartingale, for $\beta \in \CtrlStandard$;
        \item there exists $\bar{\beta}\in \CtrlStandard$ such that
        $\left\{
            w\left(s,\xi^{\bar{t},\bar \lambda;\bar\beta}_s\right)
            +\displaystyle
            \int_{\bar{t}}^s \sum_{i\in \Vc^{\bar{t},\bar \lambda;\beta}_u}
            \psi\left(i,Y^{i,\bar\beta}_u,\xi^{\bar{t},\bar \lambda;\bar\beta}_u,\bar \beta^i_u\right)du\eqsp:\eqsp s\in \l[\bar{t},T\r]
        \right\}$ is a $\P$-local martingale.
    \end{enumerate}
    Then, $\bar\beta$ is an optimal control for $v(\bar{t},\bar\lambda)$, \ie, $v(\bar{t},\bar\lambda) = J(\bar{t},\bar\lambda;\bar\beta)$, and $v(\bar{t},\bar\lambda)=w(\bar{t},\bar\lambda)$.
\end{proposition}

While this martingale-based condition can, in principle, be used to verify optimality without requiring a full differential characterization of the value function, in practice, constructing such an optimal control $\bar{\beta}$ almost always leverages a more explicit analytical (typically PDE-based) characterization.

\paragraph{Homeomorphisms with $\sqcup_{\Vc\in\Igen}\R^{d|\Vc|}$.}

The DPP paves the way for a differential characterization of the value function, which is essential for analyzing the associated optimization problem. By considering measures that belong to the space $E$, one can leverage the differential structure of different Euclidean spaces $\mathbb{R}^\ell$ for some $\ell \in \mathbb{N}$, coupled through the underlying tree structure typical of these processes. This approach, already employed in, \eg, \citet{claisse18}, \citet{kharroubi2024stochastic}, and \citet{kharroubi2024optimal} arises as a specific instance of a more general bijection—namely, the one connecting the space $E$ to the disjoint topological union $\sqcup_{\Vc\in\Igen}\R^{d|\Vc|}$, similarly to the description of the branching process used in \citet{Ustunel}.

Define the map $\iota$ by
\begin{align*}
    \iota: E \ni \sum_{i \in \bar{\Vc}} \delta_{(i, x_i)} &\mapsto \Vec{x}_{\bar{\Vc}} = (x_i)_{i \in \bar{\Vc}}
    \in \bigsqcup_{\Vc \in \Igen} \mathbb{R}^{d|\Vc|}
    \eqsp,
\end{align*}
where the vector $\Vec{x}_\Vc=(x_i)_{i\in \Vc}$ is ordered according to the total order $\leq$ on $\Ic$.
This map associates to each element $\lambda = \sum_{i \in \Vc} \delta_{(i, x_i)} \in E$, a vector in $\mathbb{R}^{d|\Vc|}$ at index with $\Vc \in \Igen$, that contains the positions of the points in $\Vc$. Viewing the processes as functions defined on $\sqcup_{\Vc \in \Igen} \mathbb{R}^{d|\Vc|}$ allows us to leverage the differential structure of these spaces, opening for a differential analysis in infinite coupled system indexed in $\Igen$.





For $\Vc\in\Igen$, define $v_\Vc:[0,T]\times \R^{d|\Vc|}\to \R$ as
\begin{align}\label{eq:def:v_Vc}
    v_\Vc(t,x_1,\dots,x_{|\Vc|}):=v\left(t,\sum_{i\in\Vc}\delta_{(i,x_i)}\right) = v\left(t,\iota^{-1}(\Vec{x}_\Vc)\right)\eqsp.
\end{align}
Analogously, we define $(\mathfrak{b}_\Vc,\Sigma_\Vc) : \R^{d|\Vc|}\times A^{|\Vc|}\to \R^{d|\Vc|}\times\R^{d|\Vc|\times d^\prime |\Vc|}$ as
\begin{align*}
    \mathfrak{b}_\Vc\left(\Vec{x}_\Vc, \Vec{a}_\Vc\right) := 
    \bigg(
        b\left(i,x_i, \iota^{-1}(\Vec{x}_\Vc),a_i\right)
    \bigg)_{i\in\Vc}, \qquad
    \Sigma_\Vc\left(\Vec{x}_\Vc, \Vec{a}_\Vc\right) :=
    \text{Diag}_{|\Vc|}\l(\bigg(
        \sigma\left(i,x_i, \iota^{-1}(\Vec{x}_\Vc),a_i\right)
    \bigg)_{i\in\Vc}\r)\eqsp,
\end{align*}
where the matrix $\text{Diag}_m$ is a diagonal matrix of size $dm\times d^\prime m$, for $m\in\N$.
For any $\Vc\in\Igen$, we define the generator $\Lb_\Vc$ as
\begin{align*}
    \Lb_\Vc v_\Vc\left(t,\Vec{x}_\Vc, \Vec{a}_\Vc\right)
    &:= \mathfrak{b}_\Vc\left(\Vec{x}_\Vc, \Vec{a}_\Vc\right)^\top D v_\Vc\left(t,\Vec{x}_\Vc\right) +  \frac{1}{2}\text{Tr}\left(
        \Sigma_\Vc(\Sigma_\Vc)^\top\left(\Vec{x}_\Vc, \Vec{a}_\Vc\right)D^2 v_\Vc\left(t,\Vec{x}_\Vc\right)
    \right)
    \\
    &~~~+\sum_{i\in\Vc} \gamma\left(i,x_i, \iota^{-1}(\Vec{x}_\Vc),a_i\right) \l(\sum_{k\geq 0}v_{\Vc^i_k}\Biggl(t,\mathfrak{e}_{\Vc,k}^i (\Vec{x}_\Vc)\r)
        p_k\left(i,x_i, \iota^{-1}(\Vec{x}_\Vc),a_i\right) - v_\Vc\left(t,\Vec{x}_\Vc\right)\Biggr)
        \eqsp,
\end{align*}
with $\mathfrak{e}_{\Vc,k}^i$ is defined as
\begin{align*}
    \mathfrak{e}_{\Vc,k}^i:\R^{|\Vc|}\to&~\R^{|\Vc|+k-1}
    \\
    \vec{x}_\Vc \mapsto &~\big(x_1,\dots,x_{i-1}, \underbrace{x_i, \dots,x_i}_{k-\text{times}},x_{i+1},\dots,x_m\big)^\top 
\end{align*}
and $\Vc^i_k$ as defined in \eqref{eq:def:V_i_k}. Moreover, define also its associated Hamiltonian $\Hb_\Vc$ as
\begin{align}
\label{eq:def-Hamiltonian}
    \begin{split}
        \Hb_\Vc: \R^{d|\Vc|}\times \R\times \R^{d|\Vc|}\times \S^{d|\Vc|}\times \R^{\Vc\times \N} &\to \R\\
        \big(\Vec{x}_\Vc,r, q_\Vc,M_\Vc,(r_{(i,\ell)})_{i\in\Vc,\ell\in\N}\big)&\mapsto 
        \inf_{\Vec{a}_\Vc\in A^{|\Vc|}}\Bigg\{
            \mathfrak{b}_\Vc\left(\Vec{x}_\Vc, \Vec{a}_\Vc\right)^\top q_\Vc
            +  \frac{1}{2}\text{Tr}\left(
                \Sigma_\Vc(\Sigma_\Vc)^\top\left(\Vec{x}_\Vc, \Vec{a}_\Vc\right)D^2 M_\Vc
            \right)
        \\&\qquad\qquad\qquad
            + \sum_{i\in\Vc} \gamma\left(i,x_i, \iota^{-1}(\Vec{x}_\Vc),a_i\right)\Bigg(\sum_{k\geq 0}r_{(i,k)}\eqsp p_k\left(i,x_i, \iota^{-1}(\Vec{x}_\Vc),a_i\right)- r \Bigg)
        \\&\qquad\qquad\qquad\qquad\qquad\qquad\qquad\qquad\qquad\qquad\qquad
            +\sum_{i\in\Vc} \psi\left(i,x_i, \iota^{-1}(\Vec{x}_\Vc),a_i\right)
        \Bigg\}
        \eqsp,
    \end{split}
\end{align}
with $\S^{d|\Vc|}$ being the set of symmetric matrices of dimension $d|\Vc|\times d|\Vc|$.
This generator will be used in the HJB equation \eqref{eq:HJB} that characterizes the value function as we will prove in \Cref{prop:viscosity_solution} and \Cref{prop:comparison_principle}.

\begin{remark}\label{Rmk:diffusion_between_branching_events}
    These notations look like the one used in the proof of~\Cref{prop:existence_strong_branching}. As seen in their construction, branching processes behave as diffusion processes between two different branching events, that are defined via a Poisson random measure independent of each Brownian motion. This is why the first two terms of $\Lb_\Vc$ are Itô's-like terms while the last one takes into account what happens in the branching events.
\end{remark}

\begin{theorem}\label{Thm:Verification_Theorem}
    Let $w\in C^0\left([0,T]\times E\right)$ such that
    \begin{align}\label{eq:verification_thm:growth_w}
        -C_w\left(1+\langle1,\lambda\rangle+\langle|\cdot|,\lambda\rangle\right)
        \leq
        w(t,\lambda)\leq C_w\left(1+\langle1,\lambda\rangle^2+\langle|\cdot|^2,\lambda\rangle\right)\eqsp.
    \end{align}
    for some constant $C_w>0$.
    Assume that $(w_\Vc)_{\Vc\in\Igen}$, defined as in~\eqref{eq:def:v_Vc}, is in $C^{1,2}\left([0,T]\times \R^{d|\Vc|}\right)$, for $\Vc\in\Igen$.
    \begin{enumerate}[(i)]
        \item If we have that
            \begin{align}
                \label{eq:verification_thm:terminal_cond}
                \begin{split}
                    -\partial_t w_\Vc\left(t,\Vec{x}_\Vc\right)
                    - \inf_{\Vec{a}_\Vc\in A^{|\Vc|}}\left\{\Lb_\Vc w_\Vc\left(\Vec{x}_\Vc, \Vec{a}_\Vc\right) + \sum_{i\in\Vc} \psi\left(i,x_i, \iota^{-1}(\Vec{x}_\Vc),a_i\right) \right\}\leq&\eqsp 0\eqsp,
                    \\
                    w_\Vc\left(T,\Vec{x}_\Vc\right) \leq&\eqsp \Psi\left(\iota^{-1}\left(\Vec{x}_\Vc\right)\right)\eqsp,
                \end{split}
            \end{align}
        for $\Vc\in\Igen$, $t\in [0,T]$, and $\Vec{x}_\Vc\in \R^{d|\Vc|}$, then $w \leq v$ on $[0,T]\times E$.
        \item Suppose, in addition to \eqref{eq:verification_thm:terminal_cond}, that $w_\Vc(T,\Vec{x}_\Vc) =\Psi(\iota^{-1}(\Vec{x}_\Vc))$, for $\Vc\in\Igen$, and $\Vec{x}_\Vc\in \R^{d|\Vc|}$, and there exist measurable functions $\Vec{\mathfrak{a}}_\Vc:[0,T)\times \R^{d|\Vc|}\to A^{|\Vc|}$, for $\Vc\in\Igen$, such that
        \begin{align}
        \label{eq:verification_thm:inf_condition}
            \begin{split}
                &-\partial_t w_\Vc\left(t,\Vec{x}_\Vc\right)
                - \inf_{\Vec{a}_\Vc\in A^{|\Vc|}}\left\{\Lb_\Vc w_\Vc\left(\Vec{x}_\Vc, \Vec{a}_\Vc\right) + \sum_{i\in\Vc} \psi\left(i,x_i, \iota^{-1}(\Vec{x}_\Vc),a_i\right) \right\} \\
                &=\eqsp-\partial_t w_\Vc\left(t,\Vec{x}_\Vc\right)
                - \left\{\Lb_\Vc w_\Vc\left(\Vec{x}_\Vc, \Vec{\mathfrak{a}}_\Vc\left(t,\Vec{x}_\Vc\right)\right) + \sum_{i\in\Vc} \psi\left(i,x_i, \iota^{-1}(\Vec{x}_\Vc)\eqsp,(\Vec{\mathfrak{a}}_\Vc)_i\left(t,\Vec{x}_\Vc\right)\right) \right\}=0\eqsp.
            \end{split}
        \end{align}
        Defining $\bar{\beta}$ as 
        \begin{align*}
            \bar{\beta}^i_s
            := 
            \sum_{\Vc\in\Igen\eqsp :\eqsp i\in\Vc}(\Vec{\mathfrak{a}}_\Vc)_i\left(t,\iota\left(\xi^{t,\lambda;\bar{\beta}}_s\right)\right)\1_{\xi^{t,\lambda;\bar{\beta}}_s = \Vc}
            +
            a_0\1_{i\notin \xi^{t,\lambda;\bar{\beta}}_s}
            \eqsp,
        \end{align*}
        we assume that the following SDE admits a unique solution
        \begin{align*}
            \langle\varphi,\xi^{\bar{\beta}}_s\rangle =& \langle\varphi,\lambda\rangle+ \int_t^s\sum_{i\in \Vc^{t,\lambda;\beta}_u}D \varphi(Y^{i,\bar{\beta}}_u)^\top\sigma\left(i,Y^{i,\bar{\beta}}_u,\xi^{\bar{\beta}}_u,
            \hat\beta^i_u
            \right)dB^i_u
            +\int_t^s\sum_{i\in \Vc^{t,\lambda;\beta}_u}L \varphi\left(Y^{i,\bar{\beta}}_u,\xi^{\bar{\beta}}_u,\hat\beta^i_u \right)du\\
            & + \int_{(t,s]\times\R_+}\sum_{i\in \Vc^{t,\lambda;\beta}_{u-}}\sum_{k\geq0}
            (k-1)\varphi(Y^{i,\hat\beta}_{u-}) \1_{I_k\left(i,Y^{i,\bar{\beta}}_{u-},\xi^{\bar{\beta}}_{u-}, \hat\beta^i_{u} \right)}(z)
            Q^i(dudz)\eqsp,
        \end{align*}
       and, that $\bar{\beta}\in \CtrlStandard$, for $(t,\lambda)\in [0,T]\times E$. Then, $w=v$ on $[0,T]\times E$, and $\bar{\beta}$ is an optimal Markov control.
    \end{enumerate}
\end{theorem}

The proof of the following is deferred to \Cref{Appendix:verification-thm}. This verification theorem has the advantage to prove not only the optimality of a solution but also showing some function is smaller than the value function. This description is the generalization of Theorem II.3.1 of \citet{Ustunel} for general value functions.
As its proof is a straightforward adaptation in our setting of Theorem 3.5.3 of \citet{pham2009continuous}, we provide it in \Cref{Appendix:verification-thm}.

\subsection{Viscosity solutions}
We now introduce the partial differential equation (PDE) associated with the control problem as follows \begin{align}
\label{eq:HJB}
    \begin{cases}
        -\partial_t w_\Vc\left(t,\Vec{x}_\Vc\right)
        - \inf_{\Vec{a}_\Vc\in A^{|\Vc|}}\left\{\Lb_\Vc w_\Vc\left(\Vec{x}_\Vc, \Vec{a}_\Vc\right) + \sum_{i\in\Vc} \psi\left(i,x_i, \iota^{-1}(\Vec{x}_\Vc),a_{\Vc,i}\right) \right\}
        =  0 \eqsp,
        \\
        w_\Vc\left(T,\Vec{x}_\Vc\right) = \Psi\l(\iota^{-1}(\Vec{x}_\Vc)\r)\eqsp,
    \end{cases}
\end{align}
for $\Vc\in\Igen$, $(t,\Vec{x}_\Vc)\in [0,T]\times \R^{d|\Vc|}$.
This corresponds to an infinite system of coupled HJB equations indexed by \(\Igen\). In \Cref{Thm:Verification_Theorem}, we demonstrated how an optimal control can be constructed from a sufficiently regular solution of this system. However, such regular solutions are not guaranteed to exist in general, particularly when the value function lacks smoothness.

In this context, \emph{viscosity solutions} provide a powerful tool to analyze the HJB equation. They allow us to define solutions in a weak sense. This approach, previously adopted in the context of branching processes by \citet{claisse18,kharroubi2024stochastic,kharroubi2024optimal}, employs viscosity solutions to rigorously connect the control problem to its PDE characterization, ensuring both existence and uniqueness via a comparison principle. Following these works, we adapt their methodology to our framework and provide the corresponding results.

The standard methodology in the stochastic control literature typically follows three main steps: (1) establish \emph{growth conditions} on the value function derived from the growth assumptions on the cost functions; (2) prove that the value function is a \emph{viscosity solution} to the associated HJB equation; and (3) demonstrate that this solution is unique within the class of functions satisfying the prescribed growth bounds, by establishing a \emph{comparison principle}.


\paragraph{Regularity and growth conditions.}
Under Assumptions~H\ref{hypH:model_parameters}--H\ref{hypH:coercivity_hyp}, the regularity and growth conditions imposed on the cost functions $\psi$ and $\Psi$ are naturally inherited by the value function $v$. This transfer of growth behavior follows from standard estimates on the controlled branching diffusion and the structure of the cost functional, and ensures that $v$ satisfies similar polynomial bounds, crucial for the well-posedness of \eqref{eq:HJB} and the comparison arguments.
The following result is a generalization, \eg, of Theorem II.10.1 and Theorem II.10.2 of \citet{FS}, to controlled processes solution to \eqref{SDE:strong}.

\begin{proposition}
\label{prop:continuity-value-fct}
    Assume that Assumptions~H\ref{hypH:model_parameters}--H\ref{hypH:coercivity_hyp} hold. Then, the value function $[0,T]\times E\ni (t,x)\mapsto v(t,x)$ is continuous in $[0,T]\times E$.
    Moreover, there exists a constant $C>0$, depending only on the time horizon $T$ and the coefficients $(b,\sigma,\gamma,(p_k)_{k\geq0}, \psi,\Psi)$, such that
    \begin{align}
    \label{eq:growth-condition:value-fct}
        -C\left(1+\langle1,\lambda\rangle+\langle|\cdot|,\lambda\rangle\right)\leq v(t,\lambda)\leq C\left(
            1+ \langle1,\lambda\rangle^2 + \langle|\cdot|^2,\lambda\rangle
        \right)
        \eqsp,\qquad\text{ for }(t,\lambda)\in[0,T]\times E
        \eqsp.
    \end{align}
\end{proposition}
\input{main_text/growth-proof.tex}

\paragraph{Viscosity solutions.} First, we define the notion of viscosity solution in our setting.

\begin{definition}[Viscosity solution]
\label{def:viscosity_solution}
    Let $v: [0,T] \times E \to \mathbb{R}$ be a continuous function, with $(v_\Vc)_{\Vc\in\Igen}$ its decomposition as in~\eqref{eq:def:v_Vc}.
    We say that
    \begin{itemize}
        \item $v$ is a viscosity subsolution of \eqref{eq:HJB} if
        \begin{enumerate}
            \item $v(T,\lambda) \leq \Psi(\lambda)$, for $\lambda\in E$;
            \item for all $(t_0,\lambda_0) \in [0,T) \times E$, with $\lambda_0= \sum_{i\in\Vc_0}\delta_{x^0_i}$, and all test function $(\varphi_\Vc)_{\Vc\in\Igen}$ such that $\varphi_\Vc\in C^{1,2}([0,T] \times \mathbb{R}^{d|\Vc|})$, $(t,\Vc,\vec{x}_\Vc)\mapsto v_\Vc(t,\vec{x}_\Vc) - \varphi_\Vc(t,\vec{x}_\Vc)$ attains a local maximum at $(t_0,\Vc_0,\iota(\lambda_0))$, it holds that
            \begin{align*}
                -\partial_t \varphi_{\Vc_0}\left(t,\iota(\lambda_0)\right)
                - \inf_{\Vec{a}_{\Vc_0}\in A^{|\Vc_0|}}\left\{\Lb_\Vc \varphi_{\Vc_0}\left(\iota(\lambda_0), \Vec{a}_{\Vc_0}\right) - \sum_{i\in \Vc_0} \psi\left(i,x^0_i, \lambda_0,a_{\Vc_0,i}\right) \right\}
                \leq  0
                \eqsp.
            \end{align*}
        \end{enumerate}
        \item $v$ is a viscosity supersolution of \eqref{eq:HJB} if
        \begin{enumerate}
            \item $v(T,\lambda) \geq \Psi(\lambda)$, for $\lambda\in E$;
            \item for all $(t_0,\lambda_0) \in [0,T) \times E$, with $\lambda_0= \sum_{i\in\Vc_0}\delta_{x^0_i}$, and all test function $(\varphi_\Vc)_{\Vc\in\Igen}$ such that $\varphi_\Vc\in C^{1,2}([0,T] \times \mathbb{R}^{d|\Vc|})$, $(t,\Vc,\vec{x}_\Vc)\mapsto v_\Vc(t,\vec{x}_\Vc) - \varphi_\Vc(t,\vec{x}_\Vc)$ attains a local minimum at $(t_0,\Vc_0,\iota(\lambda_0))$, it holds that
            \begin{align*}
                -\partial_t \varphi_{\Vc_0}\left(t,\iota(\lambda_0)\right)
                - \inf_{\Vec{a}_{\Vc_0}\in A^{|\Vc_0|}}\left\{\Lb_\Vc \varphi_{\Vc_0}\left(\iota(\lambda_0), \Vec{a}_{\Vc_0}\right) - \sum_{i\in \Vc_0} \psi\left(i,x^0_i, \lambda_0,a_{\Vc_0,i}\right) \right\}
                \geq  0
                \eqsp.
            \end{align*}
        \end{enumerate}
        \item $v$ is a viscosity solution of \eqref{eq:HJB} if it is both a viscosity subsolution and supersolution.
    \end{itemize}
\end{definition}

\begin{proposition}
\label{prop:viscosity_solution}
    Let $v$ be the value function defined in \eqref{eq:def:cost_function}. Suppose that Assumption H\ref{hypH:model_parameters}--H\ref{hypH:coercivity_hyp} hold.
    Then, $v$ is a viscosity solution of \eqref{eq:HJB}.
\end{proposition}

Once the continuity of the value function is established in \Cref{prop:continuity-value-fct}, it becomes straightforward to adapt the verification arguments of \Cref{Thm:Verification_Theorem} to the viscosity solution framework. In particular, one can follow the classical approach used in viscosity theory for stochastic control, adapting the proof techniques outlined in Chapter 4.3 of \citet{pham2009continuous} or Theorem II.5.1 in \citet{FS}.
The approach relies on the DPP \eqref{Prop:DPP} and derive the infinitesimal properties of the test functions from the semimartingale decomposition \eqref{SDE:strong}, together with the continuity and growth properties previously established. For these reasons, we omit the detailed proof.


Using the growth condition in \eqref{eq:growth-condition:value-fct} together with the continuity of the value function, we can establish the following comparison principle.

\begin{proposition}[Comparison principle]
\label{prop:comparison_principle}
    Let $w$ (resp. $u$) be a l.s.c. (resp. u.s.c.) viscosity supersolution (resp. subsolution) to \eqref{eq:HJB} satisfying the growth condition \eqref{eq:growth-condition:value-fct}. Then, we have that
    \begin{align*}
        u(t,\lambda) \leq w(t,\lambda)\eqsp,
        \qquad \text{ for } (t,\lambda)\in [0,T]\times E\eqsp.
    \end{align*}
    In particular, if \(v\) is the value function defined in \eqref{eq:def:cost_function}, then \(v\) is the unique viscosity solution of the HJB equation \eqref{eq:HJB}.
\end{proposition}

\input{main_text/comparison-proof.tex}

%% file: main_text/growth-proof.tex
\begin{proof}[Proof of \Cref{prop:continuity-value-fct}]

\emph{Step 1: Continuity of $v$.}
Let \( (t_n, \lambda_n) \to (t,\lambda) \) in \( [0,T] \times E\). We aim to prove that
\begin{align*}
\lim_{n \to \infty} v(t_n, \lambda_n) = v(t,\lambda).
\end{align*}
Since the index set \(\Ic\) is equipped with the discrete topology, convergence of the sequence \(\lambda_n = \sum_{i_n \in \Vc_n} \delta_{(i_n, x_{i_n,n})}\) to \(\lambda = \sum_{i \in \Vc} \delta_{(i, x_i)}\) in the vague topology implies that, for sufficiently large $n$, the supports \(\Vc_n\) must coincide with \(\Vc\). Indeed, the convergence of measures in the vague topology requires convergence of the locations of atoms and the preservation of their labels in the discrete topology.

Therefore, there exists $N \in \mathbb{N}$ such that for all $n \geq N$, we have \(\Vc_n = \Vc\). Without loss of generality, we assume that $N=0$. Moreover, for each \(i \in \Vc\), the continuity of the evaluation functionals implies that $x_{i,n} \to x_i$ in $\mathbb{R}^d$ as $n \to \infty$. 

\vspace{.5em}
\emph{Step 1.1: Upper semicontinuity.} Fix \( \varepsilon > 0 \). By definition of the infimum, there exists \( \beta \in \CtrlStandard \) such that
\begin{align*}
    v(t,x) + \varepsilon \geq J(t,x;\beta)
    \eqsp.
\end{align*}
Now consider the same control $\beta$ used from initial condition $(t_n, \lambda_n)$. Under Assumption H\ref{hypH:model_parameters}, we have that \eqref{eq:stability-Ys} hold.
Since $\psi$ and $\Psi$ are continuous and satisfy the polynomial growth \eqref{eq:coercivity_hyp:big_Psi}-\eqref{eq:coercivity_hyp:psi}, and \( \beta \) is fixed, we apply dominated convergence to get
\begin{align*}
    \limsup_{n \to \infty} J(t_n, \lambda_n; \beta) = J(t,x,\beta) \leq v(t,\lambda) + \varepsilon\eqsp.
\end{align*}
Thus,
\begin{align*}
    \limsup_{n \to \infty} v(t_n, \lambda_n) \leq v(t,\lambda) + \varepsilon\eqsp.
\end{align*}
Sending \( \varepsilon \to 0 \), we obtain
\begin{align*}
    \limsup_{n \to \infty} v(t_n, \lambda_n) \leq v(t,\lambda)\eqsp.
\end{align*}

\emph{Step 1.2: Lower semicontinuity.} Fix \( \varepsilon > 0 \). For each \( n \), choose \( \beta_n \in \CtrlStandard \) such that
\begin{align*}
    v(t_n, \lambda_n) + \varepsilon \leq J(t_n, \lambda_n; \beta_n)\eqsp.
\end{align*}
We focus now on each branch $i\in\Ic$. Under Assumption H\ref{hypH:model_parameters}, following classical SDE estimates, we have that
\begin{align*}
    \sup_n \mathbb{E}^\P \left[ \sup_{s \in [t_n, T]} \left|Y^{i,\beta_n}_s\right|^2\1_{i\in\Vc^{t_n,\lambda_n;\beta_n}_s} \right] < \infty
    \eqsp.
\end{align*}
Hence, the sequence $(Y^{i,\beta_n}_\cdot \eqsp \1_{i\in\Vc^{t_n,\lambda_n;\beta_n}_\cdot})_n$ is tight in the Skorokhod space $\D([0,T]; \mathbb{R}^d)$. By Prokhorov’s theorem \citep[see, \eg,][]{billingsley2013convergence}, tightness implies that there is a converging subsequence in law. Moreover, we also have that the controls $\beta^i_n\1_{i\in\Vc^{t_n,\lambda_n;\beta_n}_\cdot}$ are tight as a consequence of \eqref{eq:uniform-bound-beta}. 
Therefore, by the same arguments, we can you extract a converging subsequence $\beta_n \to \beta$ along with $Y^{i,\beta_n}_\cdot \eqsp \1_{i\in\Vc^{t_n,\lambda_n;\beta_n}_\cdot}\to Y^{i,\beta}_\cdot \eqsp \1_{i\in\Vc^{t,\lambda;\beta}_\cdot}$.

Therefore, by lower semicontinuity of the cost functional and Fatou’s lemma:
\begin{align*}
    \liminf_{n \to \infty} v(t_n, \lambda_n) + \varepsilon \geq \liminf_{n \to \infty} J(t_n, \lambda_n; \beta_n) \geq J(t,\lambda;\beta) \geq v(t,x).
\end{align*}
Thus,
\begin{align*}
    \liminf_{n \to \infty} v(t_n, \lambda_n) \geq v(t,\lambda) - \varepsilon\eqsp.
\end{align*}
Sending \( \varepsilon \to 0 \), we get
\begin{align*}
    \liminf_{n \to \infty} v(t_n, \lambda_n) \geq v(t,\lambda)\eqsp.
\end{align*}

Combining the two previous steps yields the continuity of $v$.

\vspace{.5em}
\emph{Step 2: Growth conditions.} This step is a consequence of \Cref{prop:bound_eps_opt_control}. First, we bound the l.h.s.\ of \eqref{eq:growth-condition:value-fct}. Using the l.h.s.\ of~\eqref{eq:coercivity_hyp:big_Psi} and~\eqref{eq:coercivity_hyp:psi} together with \Cref{prop:existence_strong_branching}, as in \eqref{eq:bound_a-C_a^2}, we see that there exists a constant $C>0$ (which may change from line to line) such that
\begin{align*}
        J(t,\lambda;\beta)
        \geq&\eqsp
        -C \E\left[1 + \sup_{u\in[t,T]}|V_u|^2 +\sup_{u\in[t,T]}\sum_{i\in \Vc^{t,\lambda;\beta}_{u}}\left|Y^{i,\beta}_{u}\right|\right] + c_\psi \E\left[ \int_t^{T} \sum_{i\in \Vc^{t,\lambda;\beta}_{u}}\left|\beta^{i}_{u}\right|^2 du \right]
        \\
        \geq&\eqsp
        -C\left(
            1+\langle1,\lambda\rangle^2+\langle|\cdot|,\lambda\rangle +
            \E\left[ \int_t^{T} \sum_{i\in \Vc^{t,\lambda;\beta}_{u}}\left|\beta^{i}_{u}\right| du\right]
        \right) 
        + c_\psi \E\left[
            \int_t^{T} \sum_{i\in \Vc^{t,\lambda;\beta}_{u}}\left|\beta^{i}_{u}\right|^2 du
        \right]
        \geq 
        -C\left(
            1+\langle1,\lambda\rangle^2+\langle|\cdot|,\lambda\rangle
        \right) 
        \eqsp,
\end{align*}
since the function $a \mapsto c_\psi|a|^2- C|a|$ is bounded from below.

Secondly, fix an arbitrary constant control $\beta^{a_0,i}_s :=a_0 \in A$, for $s\in[t,T]$, $i\in\Ic$. For $\varepsilon>0$, we have that
\begin{align*}
    v(t,\lambda)\leq J(t,\lambda;\beta)\leq J(t,\lambda;\beta^{a_0})+\eps
    \eqsp,\qquad\text{ for }\beta \in \CtrlStandard^{\eps}_{(t,\lambda)}
    \eqsp.
\end{align*}
Then, from the r.h.s.\ of~\eqref{eq:coercivity_hyp:big_Psi} and~\eqref{eq:coercivity_hyp:psi}, together with \Cref{prop:existence_strong_branching}, we get
\begin{align*}
    J(t,\lambda;\beta^{a_0})
    \leq&~ C\left(
        1 + \E^\P\left[
            \sup_{s\in[t,T]}\langle|\cdot|^2,\xi^{t,\lambda;\beta^{a_0}}_{s}\rangle +
            \sup_{s\in[t,T]}\langle1,\xi^{t,\lambda;\beta^{a_0}}_{s}\rangle^2 + |a_0|^2\eqsp\sup_{s\in[t,T]}\langle|\cdot|,\xi^{t,\lambda;\beta^{a_0}}_{s}\rangle
        \right]
    \right)
    \\
    \leq&~C\left(1+\langle1,\lambda\rangle^2+\langle|\cdot|^2,\lambda\rangle\right)
    \eqsp.
\end{align*}
\end{proof}

%% file: main_text/comparison-proof.tex
\begin{proof}
    In this proof, we omit the subscript $\|\cdot\|_{2,n}$ and only keep $\|\cdot\|$ for the Euclidean norm on $\R^n$, since the appropriate dimension of the vector under consideration can be inferred from the context.

    Fix $\kappa>0$. Then, a straightforward derivation shows that $\tilde{w}(t,\lambda) = \rme^{\kappa t} w(t,\lambda)$ (resp. $\tilde{u}(t,\lambda) = \rme^{\kappa t} u(t,\lambda)$), for $(t,\lambda) \in[0,T]\times E$, satisfies the growth condition \eqref{eq:growth-condition:value-fct} and is a l.s.c. (resp. u.s.c.) viscosity supersolution (resp. subsolution) to the following HJB equation
    \begin{align}
    \label{eq:comparison:new-HJB}
        \begin{cases}
            \kappa \varphi_\Vc\left(t,\Vec{x}_\Vc\right) -\partial_t \varphi_\Vc\left(t,\Vec{x}_\Vc\right)
            - \inf_{\Vec{a}_\Vc\in A^{|\Vc|}}\left\{\Lb_\Vc \varphi_\Vc\left(\Vec{x}_\Vc, \Vec{a}_\Vc\right) - \sum_{i\in\Vc} \rme^{\kappa t} \psi\left(i,x_i, \iota^{-1}(\Vec{x}_\Vc),a_{\Vc,i}\right) \right\}
            =  0 \eqsp,
            \\
            \varphi_\Vc\left(T,\Vec{x}_\Vc\right) = \rme^{\kappa T} \Psi\l(\iota^{-1}(\Vec{x}_\Vc)\r)\eqsp.
        \end{cases}
    \end{align}

    We assume to the contrary that there exists there exists $\Vc_0\in\Igen$ and $(t_0,\vec{x}^0_{\Vc_0})\in [0,T]\times \R^{d|\Vc_0|}$ such that ${u}_{\Vc_0}(t_0,\vec{x}^0_{\Vc_0})-{w}_{\Vc_0}(t_0,\vec{x}^0_{\Vc_0}) \geq  \delta$, for some $\delta>0$. This implies that
    \begin{align}
    \label{eq:comparison:absurdHP}
        \tilde{u}_{\Vc_0}(t_0,\vec{x}^0_{\Vc_0})-\tilde{w}_{\Vc_0}(t_0,\vec{x}^0_{\Vc_0}) 
        \geq 
        \rme^{t_0\kappa}\delta
        \geq 
        \delta
        \eqsp.
    \end{align}
    
    Consider the function $\Lambda:\Igen \to \mathbb{N}$ that assigns to each $\Vc \in \Igen$ the index of largest norm, \ie, $\Lambda(\Vc):=\max_{i\in\Vc}|i|$.
    Since $w$ and $u$ satisfy \eqref{eq:growth-condition:value-fct}, the same holds for $\tilde{w}$ and $\tilde{u}$.

    Fix $\epsilon > 0$. It then follows that, for any $\Vc \in \Igen$ and $(t,\vec{x}\Vc) \in [0,T] \times \mathbb{R}^{d|\Vc|}$, the function
    \[
    \tilde{w}_\Vc(t,\vec{x}_\Vc) + \tilde{u}_\Vc(t,\vec{x}_\Vc) + \epsilon (\|\vec{x}_\Vc\|^3 + |\Vc|^3)
    \]
    is uniformly bounded. In particular, this boundedness holds uniformly with respect to the index \(\Lambda(\Vc)\).
    For $n\geq0$, define the following penalized function
    \begin{align*}
        \phi_{\epsilon,n}(\Vc, t,s,\vec{x}_\Vc,\vec{y}_\Vc) 
        :=
        \tilde{u}_{\Vc}(t,\vec{x}_{\Vc}) - \tilde{w}_{\Vc}(s,\vec{y}_{\Vc})
        -\frac{n}{2}|t-s|^2
        -\frac{n}{2}\|\vec{x}_{\Vc}-\vec{y}_{\Vc}\|^2-\epsilon\l(\|\vec{x}_{\Vc}\|^3+\|\vec{y}_{\Vc}\|^3 + |\Vc|^3 + \Lambda(\Vc)\r)
        \eqsp.
    \end{align*}
    for some small $\epsilon>0$. Define $M_{\epsilon,n}$ as
    \begin{align*}
        M_{\epsilon,n} := &~
        \sup_{\Vc\in\Igen,\eqsp (t,s,\vec{x}_\Vc,\vec{y}_\Vc)\in [0,T]^2\times (\R^{d|\Vc|})^2}
        \phi_{\epsilon,n}(\Vc, t,s,\vec{x}_\Vc,\vec{y}_\Vc)\eqsp.
    \end{align*}
    Since $u$ and $w$ satisfy the growth condition \eqref{eq:growth-condition:value-fct}, there exists $\Vc_n\in\Igen$ and $(t_n,s_n,\vec{x}^n_{\Vc_n},\vec{y}^n_{\Vc_n})$ such that 
    \begin{align*}
        M_{\epsilon,n} =  \phi_{\epsilon,n}(\Vc_n, t_n,s_n,\vec{x}^n_{\Vc_n},\vec{y}^n_{\Vc_n}) 
        \eqsp.
    \end{align*}
    Take $\epsilon$ small enough such that
    \begin{align*}
        \delta_\epsilon := \delta-\epsilon\l(2\|\vec{x}^0_{\Vc_0}\|^3 + |\Vc_0|^3 + \Lambda(\Vc_0)\r) > 0
        \eqsp.
    \end{align*}
    Moreover, combining the growth condition \eqref{eq:growth-condition:value-fct} with \eqref{eq:comparison:absurdHP}, there exists $\Vc_\infty\in\Igen$ and $(t_\infty,\vec{x}^\infty_{\Vc_\infty})\in [0,T]\times \R^{d|\Vc_\infty|}$ such that
    \begin{align}
    \label{eq:comparison:maximization_pt}
        \begin{split}
            M_{\epsilon,\infty}
            :=&~
            \sup_{
                \Vc\in\Igen,\eqsp(t,\vec{x}_{\Vc})\in [0,T]\times \R^{d|\Vc|}
            }\left(
                \tilde{u}_{\Vc}(t,\vec{x}_{\Vc})-\tilde{w}_{\Vc}(t,\vec{x}_{\Vc})
                -\epsilon\l(2\|\vec{x}_{\Vc}\|^3 + |\Vc|^3 + \Lambda(\Vc)\r)
            \right)
            \\
            =&~
            \tilde{u}_{\Vc_\infty}(t_\infty,\vec{x}^\infty_{\Vc_\infty})-\tilde{w}_{\Vc_\infty}(t_\infty,\vec{x}^\infty_{\Vc_\infty})
            -\epsilon\l(2\|\vec{x}^\infty_{\Vc_\infty}\|^3 + |\Vc_\infty|^3 + \Lambda(\Vc_\infty)\r)
            \geq \delta_\epsilon > 0\eqsp.
        \end{split}
    \end{align}

    From the definition of $M_{\epsilon,n}$, taking $x=y$ in the previous supremum, we obtain that
    \begin{align}
    \label{eq:comparison:sup_geq_0}
        0<~\delta_\epsilon \leq \bar M_{n}\leq~ C_\epsilon\eqsp,
    \end{align}
    for a constant $C_\epsilon>0$ that depends on $\epsilon$ and the growth condition \eqref{eq:growth-condition:value-fct}. 
    
    As we penalize w.r.t.\ the size of $\Vc$ and element of maximum size $\Lambda(\Vc)$, we must have that there exists $K\in\N$ such that $\max\{|\Vc_n|;\Lambda(\Vc_n)\}\leq K$, for $n\geq1$. This criterion selects a finite number of possible sets of indices, \ie, the set
    \begin{align*}
        \l\{
            \Vc\in\Igen\eqsp:\eqsp|\Vc|\leq K,\eqsp \Lambda(\Vc)\leq K
        \r\}
    \end{align*}
    is finite. Therefore, without loss of generality, the sequence $\{\Vc_n\}_{n\geq1}$ is convergent, up to a subsequence that is constantly equal to an element $\Vc_\star$.

    Using the growth condition \eqref{eq:growth-condition:value-fct}, there exists a compact set $B_\epsilon\subset\R^{d|\Vc_\star|}$ such that $\vec{x}^n_{\Vc_\star},\vec{y}^n_{\Vc_\star}\in B_\epsilon$, for $n\geq1$. Therefore, up to  a sub-sequence, we can take $(\vec{x}^n_{\Vc_\star},\vec{y}^n_{\Vc_\star})\to(\vec{x}^\star_{\Vc_\star},\vec{y}^\star_{\Vc_\star})$ and $(t_n,s_n)\to(t_\star,s_\star)$, as $n\rightarrow\infty$.
    Moreover, there exists a constant $C^\prime_\epsilon>0$ such that
    \begin{align}
    \label{eq:comparison:bound_x_y_n}
        \frac{n}{2}\|\vec{x}^n_{\Vc_\star}-\vec{y}^n_{\Vc_\star}\|^2+ \frac{n}{2}|t_n-s_n|^2 \leq  C^\prime_\epsilon \eqsp,
        \qquad\text{ for }n\geq 1\eqsp.
    \end{align}
    This yields that $\lim_{n\to\infty}\|\vec{x}^n_{\Vc_\star}-\vec{y}^n_{\Vc_\star}\|=0$ (resp. $\lim_{n\to\infty}|t_n-s_n|=0$) and $\vec{x}^\star_{\Vc_\star}=\vec{y}^\star_{\Vc_\star}$ (resp. $t_\star=s_\star$). Combining this with \eqref{eq:comparison:sup_geq_0}, we obtain 
    \begin{align*}
        \lim_{n\to\infty}\frac{n}{2}\|\vec{x}^n_{\Vc_\star}-\vec{y}^n_{\Vc_\star}\|^2 =  0
        \qquad\text{ and }
        \qquad
        \lim_{n\to\infty}\frac{n}{2}|t_n-s_n|^2 =  0
        \eqsp.
    \end{align*}

    Without loss of generality, we can take the maximization point in \eqref{eq:comparison:maximization_pt} to be $\Vc_\star\in\Igen$ and $(t_\star,\vec{x}^\star_{\Vc_\star})\in [0,T]\times \R^{d|\Vc_\star|}$, $i.e.$, $(\Vc_\infty,t_\infty,\vec{x}^\infty_{\Vc_\infty})=(\Vc_\star,t_\star,\vec{x}^\star_{\Vc_\star})$.
    
    Therefore, as $(t_n,s_n,\vec{x}^n_{\Vc_\star},\vec{y}^n_{\Vc_\star})\in [0,T]^2\times (\R^{d|\Vc_\star|})^2$ is a maximizer of $M_{\epsilon, n}$, we may apply Ishii's lemma \citep[see, $e.g.$, Theorem 8.3,][]{crandall1992users} since we consider a system of PDE in finite dimension as we took $\Vc_n$ to be constant, for $n\geq1$. Therefore, there exist $A^u_n,A^w_n\in\S^{d|\Vc_\star|}$ such that
    \begin{align}
    \label{eq:comparison:eq}
        \begin{split}
            \kappa\eqsp\tilde{u}_{\Vc_\star}(t_n,\vec{x}^n_{\Vc_\star})-n(t_n-s_n) - \Hb_{\Vc_\star}\left(
                \Vec{x}^n_{\Vc_\star},\eqsp
                \tilde{u}_{\Vc_\star}(t_n,\vec{x}^n_{\Vc_\star}),\eqsp
                q^u_n,\eqsp
                A^u_n,\eqsp 
                \l(\tilde{u}_{\Vc_{\star,k}^i}\big(t_n,\mathfrak{e}_{\Vc_\star,k}^i\l(\vec{x}^n_{\Vc_\star}\r)\big)\r)_{i\in\Vc_\star,\ell\in\N}
            \right)\leq&~0
            \eqsp,
            \\
            \kappa\eqsp\tilde{w}_{\Vc_\star}(s_n,\vec{y}^n_{\Vc_\star})-n(t_n-s_n) - \Hb_{\Vc_\star}\left(
                \Vec{y}^n_{\Vc_\star},\eqsp 
                \tilde{w}_{\Vc_\star}(s_n,\vec{y}^n_{\Vc_\star}),\eqsp
                q^w_n,\eqsp
                A^w_n,\eqsp 
                \l(
                    \tilde{w}_{\Vc_{\star,k}^i}\big(s_n,\mathfrak{e}_{\Vc_\star,k}^i\l(\vec{y}^n_{\Vc_\star}\r)\big)
                \r)_{i\in\Vc_\star,\ell\in\N}
            \right)\geq&~0
            \eqsp,
        \end{split}
    \end{align}
    with
    \begin{align*}
        q^u_n
        :=
        n(\Vec{x}^n_{\Vc_\star}-\Vec{y}^n_{\Vc_\star}) + 3\epsilon \eqsp\|\Vec{x}^n_{\Vc_\star}\|\eqsp \Vec{x}^n_{\Vc_\star}
        \eqsp,
        \qquad
        q^w_n
        :=
        n(\Vec{x}^n_{\Vc_\star}-\Vec{y}^n_{\Vc_\star}) - 3\epsilon \eqsp\|\Vec{y}^n_{\Vc_\star}\|\eqsp \Vec{y}^n_{\Vc_\star}
        \eqsp,
    \end{align*}
    and
    \begin{align}
    \label{eq:comparison:matrix-ineq}
        -(n+|\mathfrak{D}_n|) ~
        \I_{2d|\Vc_\star|} \leq&~
        \begin{pmatrix}
            A^u_n & 0\\0 & - A^w_n
        \end{pmatrix} \leq \mathfrak{D}_n+\frac{1}{n}\mathfrak{D}_n^2
        \eqsp,
    \end{align}
    where
    \begin{align*}
        \mathfrak{D}_n :=&~ n \begin{pmatrix}
            \I_{d|\Vc_\star|} & -\I_{d|\Vc_\star|}\\-\I_{d|\Vc_\star|} & \I_{d|\Vc_\star|}
        \end{pmatrix} + 3\epsilon\begin{pmatrix}
            \eta_{\Vc_\star}(\Vec{x}^n_{\Vc_\star}) & -(\eta_{\Vc_\star}(\Vec{x}^n_{\Vc_\star})+ \eta_{\Vc_\star}(\Vec{y}^n_{\Vc_\star}))\\-(\eta_{\Vc_\star}(\Vec{x}^n_{\Vc_\star})+ \eta_{\Vc_\star}(\Vec{y}^n_{\Vc_\star})) & \eta_{\Vc_\star}(\Vec{y}^n_{\Vc_\star})
        \end{pmatrix}
    \end{align*}
    and
    \begin{align*}
        \eta_{\Vc}(\Vec{x}_{\Vc}) :=&~ \|\Vec{x}_{\Vc}\|\eqsp \I_{d|\Vc|} +\frac{\Vec{x}_{\Vc}\Vec{x}_{\Vc}^\top }{\|\Vec{x}_{\Vc}\|}\eqsp,\qquad\text{ for }\Vc\in\Igen,\eqsp \Vec{x}_{\Vc}\in\R^{d|\Vc|}
        \eqsp.
    \end{align*}
    Using the fact that $\eta_{\Vc}(\Vec{x}_{\Vc})\Vec{x}_{\Vc} = 2 \|\Vec{x}_{\Vc}\|\eqsp\Vec{x}_{\Vc}$, 
    the r.h.s.\ of
    \eqref{eq:comparison:matrix-ineq} reduces to
    \begin{align*}
        &\begin{pmatrix}
            A^u_n & 0\\0 & - A^w_n
        \end{pmatrix} 
        \\
        &\leq 
        3n\begin{pmatrix}
            \I_{d|\Vc_\star|} & -\I_{d|\Vc_\star|}\\-\I_{d|\Vc_\star|} & \I_{d|\Vc_\star|}
        \end{pmatrix}
        +6\epsilon
        \begin{pmatrix}
            3\|\Vec{x}^n_{\Vc_\star}\|\eqsp \I_{d|\Vc_\star|} & 
            -(\|\Vec{x}^n_{\Vc_\star}\|+\|\Vec{y}^n_{\Vc_\star}\|)\eqsp \I_{d|\Vc_\star|}
            \\
            -(\|\Vec{x}^n_{\Vc_\star}\|+\|\Vec{y}^n_{\Vc_\star}\|)\eqsp\I_{d|\Vc_\star|} & 
            3\|\Vec{y}^n_{\Vc_\star}\|\eqsp \I_{d|\Vc_\star|}
        \end{pmatrix}
        +\frac{36 \epsilon^2}{n}\begin{pmatrix}
            \|\Vec{x}^n_{\Vc_\star}\|^2\eqsp \I_{d|\Vc_\star|} & \0_{d|\Vc_\star|}\\ \0_{d|\Vc_\star|} & \|\Vec{y}^n_{\Vc_\star}\|^2 \eqsp\I_{d|\Vc_\star|}
        \end{pmatrix}
        \eqsp,
    \end{align*}
    and the l.h.s.\ to
    \begin{align*}
        -(2n+6\epsilon(\|\Vec{x}^n_{\Vc_\star}\|^2+\|\Vec{y}^n_{\Vc_\star}\|^2)) 
        \I_{2d|\Vc_\star|} \leq\begin{pmatrix}
            A^u_n & 0\\0 & - A^w_n
        \end{pmatrix}
        \eqsp.
    \end{align*}

    From \eqref{eq:comparison:eq}, we obtain
    \begin{align}
        \label{eq:comparison:eq2}
        \begin{split}
            \kappa\left(
                \tilde{u}_{\Vc_\star}(t_n,\vec{x}^n_{\Vc_\star})
                -
                \tilde{w}_{\Vc_\star}(s_n,\vec{y}^n_{\Vc_\star})
            \right)\eqsp
            \leq&~ \Hb_{\Vc_\star}\left(
                \Vec{x}^n_{\Vc_\star},\eqsp
                \tilde{u}_{\Vc_\star}(t_n,\vec{x}^n_{\Vc_\star}),\eqsp
                q^u_n,\eqsp
                A^u_n,\eqsp 
                \l(\tilde{u}_{\Vc_{\star,k}^i}\big(t_n,\mathfrak{e}_{\Vc_\star,k}^i\l(\vec{x}^n_{\Vc_\star}\r)\big)\r)_{i\in\Vc_\star,k\in\N}
            \right)
            \\
            &~- \Hb_{\Vc_\star}\left(
                \Vec{y}^n_{\Vc_\star},\eqsp 
                \tilde{w}_{\Vc_\star}(s_n,\vec{y}^n_{\Vc_\star}),\eqsp
                q^w_n,\eqsp
                A^w_n,\eqsp 
                \l(
                    \tilde{w}_{\Vc_{\star,k}^i}\big(s_n,\mathfrak{e}_{\Vc_\star,k}^i\l(\vec{y}^n_{\Vc_\star}\r)\big)
                \r)_{i\in\Vc_\star,k\in\N}
            \right)
            \eqsp.
        \end{split}
    \end{align}
    From Assumption H\ref{hypH:model_parameters} together with the definitions of $q^u_n$, $q^w_n$, $A^u_n$, and $A^w_n$, there exists a constant $\check{C}>0$ (which may change from line to line) depending only on $b$, $\sigma$, $\gamma$, and $(p_k)_{k\geq0}$ such that, from \eqref{eq:comparison:eq2}, we get
    \begin{align}
        \label{eq:comparison:eq3}
            \kappa\left(
                \tilde{u}_{\Vc_\star}(t_n,\vec{x}^n_{\Vc_\star})
                -
                \tilde{w}_{\Vc_\star}(s_n,\vec{y}^n_{\Vc_\star})
            \right)\eqsp
            \leq&~ 
            \check{C}\left(
                n\|\Vec{x}^n_{\Vc_\star} -\Vec{y}^n_{\Vc_\star}\|^2
                +
                \epsilon(\|\Vec{x}^n_{\Vc_\star}\|^3+\|\Vec{y}^n_{\Vc_\star}\|^3)
                +
                \frac{\epsilon}{n}(\|\Vec{x}^n_{\Vc_\star}\|^4+\|\Vec{y}^n_{\Vc_\star}\|^4)
            \right)
            +\Delta_n^1 +\Delta^2_n
            \eqsp.
    \end{align}
    with
    \begin{align*}
        \Delta^1_n
        :=&~
        \sup_{\Vec{a}_{\Vc_\star}\in A^{|\Vc_\star|}}\Bigg\{
            \sum_{i\in\Vc_\star} 
            \gamma\left(i,x^n_i, \iota^{-1}(\Vec{x}^n_{\Vc_\star}),a_i\right)
            \bigg[\sum_{k\geq 0}
            \left(
                \tilde{u}_{\Vc_{\star,k}^i}\big(t_n,\mathfrak{e}_{\Vc_\star,k}^i\l(\vec{x}^n_{\Vc_\star}\r)\big)
                -
                \tilde{w}_{\Vc_{\star,k}^i}\big(s_n,\mathfrak{e}_{\Vc_\star,k}^i\l(\vec{y}^n_{\Vc_\star}\r)\big)
            \right)
            \eqsp p_k\left(i,x^n_i, \iota^{-1}(\Vec{x}^n_{\Vc_\star}),a_i\right)
        \\
        &\qquad\qquad\qquad\qquad\qquad\qquad\qquad\qquad\qquad\qquad\qquad\qquad\qquad\qquad
        - \big(
                \tilde{u}_{\Vc_{\star}}\big(t_n,\vec{x}^n_{\Vc_\star}\big)
                -
                \tilde{w}_{\Vc_{\star}}\big(s_n,\vec{y}^n_{\Vc_\star}\big)
            \big)
        \bigg]
        \Bigg\}
        \eqsp,
        \\
        \Delta^2_n
        :=&~
        \sup_{\Vec{a}_{\Vc_\star}\in A^{|\Vc_\star|}}\Bigg\{
            \sum_{i\in\Vc_\star} 
            \gamma\left(i,x^n_i, \iota^{-1}(\Vec{x}^n_{\Vc_\star}),a_i\right)
            \sum_{k\geq 0}
            \tilde{w}_{\Vc_{\star,k}^i}\big(s_n,\mathfrak{e}_{\Vc_\star,k}^i\l(\vec{y}^n_{\Vc_\star}\r)\big)
            \left(
                p_k\left(i,x^n_i, \iota^{-1}(\Vec{x}^n_{\Vc_\star}),a_i\right)
                -
                p_k\left(i,y^n_i, \iota^{-1}(\Vec{y}^n_{\Vc_\star}),a_i\right)
            \right)
        \Bigg\}
        \eqsp.
    \end{align*}

    First, we deal with $\Delta^1_n$. From the definition of $M_{\epsilon,n}$, we have
    \begin{align*}
        &\tilde{u}_{\Vc_\star}(t_n,\vec{x}^n_{\Vc_\star}) - \tilde{w}_{\Vc_\star}(s_n,\vec{y}^n_{\Vc_\star})
        -\frac{n}{2}\|\vec{x}^n_{\Vc_\star}-\vec{y}^n_{\Vc_\star}\|^2-\epsilon\l(\|\vec{x}^n_{\Vc_\star}\|^3+\|\vec{y}^n_{\Vc_\star}\|^3 + |\Vc_\star|^3 + \Lambda(\Vc_\star)\r)
        \\
        &\geq
        \tilde{u}_{\Vc_{\star,k}^i}\big(
            t_n,
            \mathfrak{e}_{\Vc_\star,k}^i\l(\vec{x}^n_{\Vc_\star}\r)
        \big)
        -
        \tilde{w}_{\Vc_{\star,k}^i}\big(s_n,\mathfrak{e}_{\Vc_\star,k}^i\l(\vec{y}^n_{\Vc_\star}\r)\big)
        -\frac{n}{2}\|\mathfrak{e}_{\Vc_\star,k}^i\l(\vec{x}^n_{\Vc_\star}\r)-\mathfrak{e}_{\Vc_\star,k}^i\l(\vec{y}^n_{\Vc_\star}\r)\|^2
        \\
        &\qquad\qquad\qquad\qquad\qquad\qquad\qquad-
        \epsilon\l(\|\mathfrak{e}_{\Vc_\star,k}^i\l(\vec{x}^n_{\Vc_\star}\r)\|^3+\|\mathfrak{e}_{\Vc_\star,k}^i\l(\vec{y}^n_{\Vc_\star}\r)\|^3 + |\Vc_{\star,k}^i|^3 + \Lambda(\Vc_{\star,k}^i)\r)
        \eqsp.
    \end{align*}
    This means that we get
    \begin{align*}
        \Delta^1_n\leq \Delta^{1,a}_n +
        \Delta^{1,b}_n +
        \Delta^{1,c}_n +
        \Delta^{1,d}_n
        \eqsp,
    \end{align*}
    with
    \begin{align*}
        \Delta^{1,a}_n 
        :=&~
        \frac{n}{2}\sup_{\Vec{a}_{\Vc_\star}\in A^{|\Vc_\star|}}
        \sum_{i\in\Vc_\star}
        \sum_{k\geq 0}
        \left(
            \|\mathfrak{e}_{\Vc_\star,k}^i\l(\vec{x}^n_{\Vc_\star}\r)-\mathfrak{e}_{\Vc_\star,k}^i\l(\vec{y}^n_{\Vc_\star}\r)\|^2 -\|\vec{x}^n_{\Vc_\star}-\vec{y}^n_{\Vc_\star}\|^2
        \right)(\gamma p_k)\left(i,x^n_i, \iota^{-1}(\Vec{x}^n_{\Vc_\star}),a_i\right)
        \eqsp,
        \\
        \Delta^{1,b}_n 
        :=&~
        \epsilon \sup_{\Vec{a}_{\Vc_\star}\in A^{|\Vc_\star|}}
        \sum_{i\in\Vc_\star}
        \sum_{k\geq 0}
        \left(
            \|\mathfrak{e}_{\Vc_\star,k}^i\l(\vec{x}^n_{\Vc_\star}\r)\|^3+\|\mathfrak{e}_{\Vc_\star,k}^i\l(\vec{y}^n_{\Vc_\star}\r)\|^3
            -
            \|\vec{x}^n_{\Vc_\star}\|^3-\|\vec{y}^n_{\Vc_\star}\|^3
        \right)(\gamma p_k)\left(i,x^n_i, \iota^{-1}(\Vec{x}^n_{\Vc_\star}),a_i\right)
        \eqsp,
        \\
        \Delta^{1,c}_n 
        :=&~
        \epsilon \sup_{\Vec{a}_{\Vc_\star}\in A^{|\Vc_\star|}}
        \sum_{i\in\Vc_\star}
        \sum_{k\geq 0}
        \left(
            |\Vc_{\star,k}^i|^3
            -
            |\Vc_{\star}|^3
        \right)(\gamma p_k)\left(i,x^n_i, \iota^{-1}(\Vec{x}^n_{\Vc_\star}),a_i\right)
        \eqsp,
        \\
        \Delta^{1,d}_n 
        :=&~
        \epsilon \sup_{\Vec{a}_{\Vc_\star}\in A^{|\Vc_\star|}}
        \sum_{i\in\Vc_\star}
        \sum_{k\geq 0}
        \left(
            \Lambda(\Vc_{\star,k}^i)
            -
            \Lambda(\Vc_{\star})
        \right)(\gamma p_k)\left(i,x_i^n, \iota^{-1}(\Vec{x}^n_{\Vc_\star}),a_i\right)
        \eqsp.
    \end{align*}
    Since $\|\mathfrak{e}_{\Vc_\star,k}^i\l(\vec{x}^n_{\Vc_\star}\r)-\mathfrak{e}_{\Vc_\star,k}^i\l(\vec{y}^n_{\Vc_\star}\r)\|^2 -\|\vec{x}_{\Vc_\star}^n-\vec{y}^n_{\Vc_\star}\|^2 = (k-1)|x^n_i-y^n_i|^2$, using Assumption H\ref{hypH:model_parameters}(ii)-(iii), we have that
    \begin{align*}
        \Delta^{1,a}_n
        \leq&~
        \frac{n}{2}\sup_{\Vec{a}_{\Vc_\star}\in A^{|\Vc_\star|}}
        \sum_{i\in\Vc_\star}
        \sum_{k\geq 0}
        (k-1)|x_i-y_i|^2(\gamma p_k)\left(i,x_i, \iota^{-1}(\Vec{x}_\Vc),a_i\right)
        \\
        \leq &~
        C_\gamma
        C^1_\Phi\eqsp
        \frac{n}{2} \sum_{i\in\Vc_\star}|x^n_i-y^n_i|^2 
        =
        C_\gamma
        C^1_\Phi\eqsp
        \frac{n}{2}\|\vec{x}_{\Vc_\star}^n-\vec{y}^n_{\Vc_\star}\|^2
        \leq
        \check{C}\eqsp\frac{n}{2}\|\vec{x}_{\Vc_\star}^n-\vec{y}^n_{\Vc_\star}\|^2
        \eqsp.
    \end{align*}
    Mean value theorem yields that
    $\|\mathfrak{e}_{\Vc_\star,k}^i(\vec{x}^n_{\Vc_\star})\|^3-\|\vec{x}^n_{\Vc_\star}\|^3\leq 3/2 \eqsp \|\mathfrak{e}_{\Vc_\star,k}^i(\vec{x}^n_{\Vc_\star})\|\eqsp (k-1)|x_i^n|^2$. Combining this with
    \begin{align*}
        \|\mathfrak{e}_{\Vc_\star,k}^i(\vec{x}^n_{\Vc_\star})\| - \sqrt{(k-1)|x_i^n|^2} 
        =
        \frac{\|\vec{x}^n_{\Vc_\star}\|^2 }{\|\mathfrak{e}_{\Vc_\star,k}^i(\vec{x}^n_{\Vc_\star})\| + \sqrt{(k-1)|x_i^n|^2}}
        \leq \frac{\|\vec{x}^n_{\Vc_\star}\|^2 }{2 \sqrt{(k-1)|x_i^n|^2}}
        \eqsp,
    \end{align*}
    we obtain
    \begin{align*}
        &\epsilon \sup_{\Vec{a}_{\Vc_\star}\in A^{|\Vc_\star|}}
        \sum_{i\in\Vc_\star}
        \sum_{k\geq 0}
        \left(
            \|\mathfrak{e}_{\Vc_\star,k}^i\l(\vec{x}^n_{\Vc_\star}\r)\|^3
            -
            \|\vec{x}^n_{\Vc_\star}\|^3
        \right)(\gamma p_k)\left(i,x^n_i, \iota^{-1}(\Vec{x}^n_{\Vc_\star}),a_i\right)
        \\
        &\leq
        \frac{3\epsilon}{2} \sup_{\Vec{a}_{\Vc_\star}\in A^{|\Vc_\star|}}
        \sum_{i\in\Vc_\star}
        \sum_{k\geq 0}
        \left(
            \frac{1}{2}\|\vec{x}^n_{\Vc_\star}\|^2\eqsp
            \sqrt{k-1}|x_i^n|
             + (k-1)^{3/2}|x_i^n|^{3}
        \right)(\gamma p_k)\left(i,x^n_i, \iota^{-1}(\Vec{x}^n_{\Vc_\star}),a_i\right)
        \\
        &\leq
        \frac{3\epsilon}{2}C_\gamma
        \left(
            \frac{C_\Phi^1}{2}\|\vec{x}^n_{\Vc_\star}\|^2
            \sum_{i\in\Vc_\star}
            |x_i^n|
            + C_\Phi^2
            \sum_{i\in\Vc_\star}|x_i^n|^{3}
        \right)
        \\
        &\leq
        \frac{3\epsilon}{2}C_\gamma
        \left(
            \frac{C_\Phi^1}{2}\sqrt{|\Vc_\star|}
            \eqsp
            \|\vec{x}^n_{\Vc_\star}\|^3
            +
            C_\Phi^2
            \|\vec{x}^n_{\Vc_\star}\|^3
        \right)\leq \check{C}\eqsp\epsilon\sqrt{|\Vc_\star|}
            \eqsp
            \|\vec{x}^n_{\Vc_\star}\|^3
        \eqsp,
    \end{align*}
    where in the last inequality we used Cauchy--Schwarz inequality to get $\sum_{i\in\Vc_\star}|x_i^n|\leq \sqrt{|\Vc_\star|}\eqsp \|\vec{x}^n_{\Vc_\star}\|$ and
    \begin{align*}
        \sum_{i\in\Vc_\star}|x_i^n|^{3}
        \leq&~
        \left(
            \sum_{i\in\Vc_\star}|x_i^n|^{2}
        \right)^{1/2}\eqsp
        \left(
            \sum_{i\in\Vc_\star}|x_i^n|^{4}
        \right)^{1/2}
        \leq
        \|\vec{x}^n_{\Vc_\star}\|\eqsp
        \left(\max_{i\in\Vc_\star}|x_i^n|^{2}\right)\eqsp
        \left(
            \sum_{i\in\Vc_\star}|x_i^n|^{2}
        \right)^{1/2}
        \\
        \leq&~
        \|\vec{x}^n_{\Vc_\star}\|\eqsp
        \left(
            \sum_{i\in\Vc_\star}|x_i^n|^{2}
        \right)^{1/2}\eqsp
        \left(
            \sum_{i\in\Vc_\star}|x_i^n|^{2}
        \right)^{1/2} = \|\vec{x}^n_{\Vc_\star}\|^3
        \eqsp.
    \end{align*}
    Inverting the roles of $\vec{x}^n_{\Vc_\star}$ and $\vec{y}^n_{\Vc_\star}$, we have that $\Delta^{1,c}_n\leq \check{C}\eqsp\epsilon\sqrt{|\Vc_\star|} \eqsp (\|\vec{x}^n_{\Vc_\star}\|^3+\|\vec{y}^n_{\Vc_\star}\|^3)$. Replacing $x_i^n$ with $1$ in the previous bound, we get that
    $\Delta^{1,c}_n \leq  \check{C}\eqsp\epsilon |\Vc_\star|^2$. Moreover, the bound $\Lambda(\Vc_{\star,k}^i)-\Lambda(\Vc_{\star})\leq k$ implies that $\Delta^{1,c}_n \leq  \check{C}\eqsp\epsilon$.

    We now focus on $\Delta^{2}_n$. It is clear that $\Delta^{2}_n\to 0$, for $n\to\infty$, as a consequence of the dominated convergence theorem, as $(\vec{x}^n_{\Vc_\star}, \vec{y}^n_{\Vc_\star})\to (\vec{x}^\star_{\Vc_\star},\vec{x}^\star_{\Vc_\star})$, for $n\to\infty$. Such a result can be applied since, from \eqref{eq:growth-condition:value-fct} and Assumption H\ref{hypH:model_parameters}, we have the following bound
    \begin{align*}
        |\Delta^{2}_n|
        \leq&~ 
        \sup_{\Vec{a}_{\Vc_\star}\in A^{|\Vc_\star|}}\Bigg\{
            \sum_{i\in\Vc_\star} 
            \gamma\left(i,x^n_i, \iota^{-1}(\Vec{x}^n_\Vc),a_i\right)
            \sum_{k\geq 0}
            \tilde{w}_{\Vc_{\star,k}^i}\big(s_n,\mathfrak{e}_{\Vc_\star,k}^i\l(\vec{y}^n_{\Vc_\star}\r)\big)
            \left(
                p_k\left(i,x^n_i, \iota^{-1}(\Vec{x}^n_\Vc),a_i\right)
                +
                p_k\left(i,y^n_i, \iota^{-1}(\Vec{y}^n_\Vc),a_i\right)
            \right)
        \Bigg\}
        \\
        \leq&~
        C\eqsp C_\gamma
        \sum_{i\in\Vc_\star}
        \sum_{k\geq 0}
        (\|\mathfrak{e}_{\Vc_\star,k}^i\l(\vec{y}^n_{\Vc_\star}\r)\|^2+ |\Vc_{\star,k}^i|^2)\left(
                p_k\left(i,x^n_i, \iota^{-1}(\Vec{x}^n_\Vc),a_i\right)
                +
                p_k\left(i,y^n_i, \iota^{-1}(\Vec{y}^n_\Vc),a_i\right)
            \right)
        \\
        \leq&~
        C\eqsp C_\gamma
        \sum_{i\in\Vc_\star}
        \sum_{k\geq 0}
        (\|\vec{y}^n_{\Vc_\star}\|^2+ (k-1)|y^i_n|^2+ 2|\Vc_\star|^2 + 2(k-1)^2)\left(
                p_k\left(i,x^n_i, \iota^{-1}(\Vec{x}^n_\Vc),a_i\right)
                +
                p_k\left(i,y^n_i, \iota^{-1}(\Vec{y}^n_\Vc),a_i\right)
            \right)
        \\
        \leq&~
        \check{C}(\|\vec{y}^n_{\Vc_\star}\|^2+|\Vc_\star|)
        \eqsp,
    \end{align*}
    which is a uniform bound since $\vec{y}^n_{\Vc_\star}\to\vec{x}^\star_{\Vc_\star}$.

    Sending $n$ to infinity, it follows from \eqref{eq:comparison:eq3}
    \begin{align*}
            \kappa\left(
                \tilde{u}_{\Vc_\star}(t_\star,\vec{x}^\star_{\Vc_\star})
                -
                \tilde{w}_{\Vc_\star}(t_\star,\vec{x}^\star_{\Vc_\star})
            \right)\eqsp
            \leq&~ 
            \check{C}\eqsp
                \epsilon(
                    \sqrt{|\Vc_\star|} \eqsp \|\Vec{x}^\star_{\Vc_\star}\|^3
                    +
                    |\Vc_\star|^2
                )
            \eqsp.
    \end{align*}
    Since the constant $\check{C}$ is independent on $\epsilon$, taking $\kappa>\check{C}$, the previous equation is a contradiction of \eqref{eq:comparison:absurdHP}.

\end{proof}

%% file: main_text/mean-field-setting.tex
\section{The mean field regime}
\label{section:MF-regime}

Modeling all possible binary interactions quickly becomes prohibitively costly, both analytically and computationally. In many applications, it is natural to consider symmetric interactions, where each individual reacts only to the overall distribution of the population and not to the labels of the other participants. This leads to the \emph{mean-field (MF) setting}, a framework widely employed in control theory due to its broad applicability in real-world problems \citep[see, \eg,][]{nourian2012nash,carmona2013mean,fornasier2014mean,seguret2021mean}.

\subsection{Symmetric controls}

The mean-field setting is built on two fundamental assumptions. The first is \emph{anonymity}, which requires that agents react only to the empirical distribution of the population’s positions, without distinguishing between individual identities. The second is \emph{homogeneity}, which stipulates that an agent’s behavior is independent of its specific label. These assumptions naturally restrict the generality of the model parameters: anonymity is reflected through dependence on the empirical distribution $\pi(\lambda)$ instead of the full configuration $\lambda$, while homogeneity is enforced by requiring invariance with respect to the particle index $i \in \Ic$. Formally, this is implied by the following assumption.

\begin{hypH}
    \label{hypH:MF-setting}
    There exists a family of coefficients
    \begin{align*}
        \bigl(b^{\text{MF}},\sigma^{\text{MF}},\gamma^{\text{MF}},(p^{\text{MF}}_k)_{k \geq 0},\psi^{\text{MF}}\bigr)
        : \R^d \times \Nc(\R^d) \times A \to 
        \R^d \times \R^{d \times d^\prime} \times \R_+ \times [0,1]^\N \times \R
        \quad\text{ and }\quad
        \Psi^{\text{MF}}:\Nc(\R^d)\to \R\eqsp,
    \end{align*}
    such that
    \begin{align}
    \label{eq:MF-criterion}
        \big(b,\sigma,\gamma,(p_k)_{k\geq 0},\psi\big)(i,x,\lambda,a)
        =
        \big(
            b^{\mathrm{MF}},\sigma^{\mathrm{MF}},\gamma^{\mathrm{MF}},(p^{\mathrm{MF}}_k)_{k \geq 0},\psi^{\mathrm{MF}},\Psi^{\mathrm{MF}}
        \big)(x,\pi(\lambda),a)
        \quad\text{ and }\quad
        \Psi(\lambda) =
        \Psi^{\mathrm{MF}}(\pi(\lambda))
        \eqsp,
    \end{align}
    for $(i,x,\lambda,a) \in \Ic\times \R^d\times E\times A$, 
    where $\pi: E \to \mathcal{P}(\mathbb{R}^d)$ denotes the projection mapping a configuration $\lambda$ to the empirical distribution of particle positions.
\end{hypH}

Consider now the following class of controls, which we call \emph{symmetric}. These are controls that assign the same action to all particles occupying the same position.
\begin{definition}[Symmetric control]
    Fix $(t,\lambda) \in [0,T]\times E$.
    We say that $\beta = (\beta^i)_{i\in \Ic}$ is an \emph{symmetric control}, and we denote $\beta\in\CtrlStandard^\mathfrak{s}_{(t,\lambda)}$, if $\beta\in\CtrlStandard$ and, for $\xi^{t,\lambda;\beta}=\sum_{i\in\Vc^{t,\lambda;\beta}}\delta_{(i,Y^{i,\beta}_s)}$ solution of \eqref{SDE:strong}, we have
    \begin{align}
        \label{eq:cyclical_strong_ctrl_condition}
        \beta^i_s = \beta^j_s\eqsp,
        \quad \text{ whenever } Y^{i,\beta}_s = Y^{j,\beta}_s\eqsp,
    \end{align}
    for $s\in[0,T]$ and $i,j\in\Vc^{t,\lambda;\beta}_s$.
\end{definition}
The MF structure significantly reduces the complexity of the problem while preserving the key probabilistic features of the general setting. In particular, since the coefficients of the HJB equation \eqref{eq:HJB} depend only on the empirical measure and not on individual indices, they are invariant under permutations of particle labels. This invariance naturally propagates to the feedback optimizer, ensuring that optimal controls are symmetric. Consequently, in formulating the control problem, it is sufficient to restrict attention to the class of symmetric admissible controls as we prove in the following proposition.

\begin{proposition}[Restriction to \emph{symmetric} controls]
    \label{prop:MF-symmetric-controls}
    Suppose Assumption H\ref{hypH:model_parameters}-H\ref{hypH:coercivity_hyp}-H\ref{hypH:MF-setting} hold. Fix $(t,\lambda=\sum_{i\in\Vc}\delta_{(i,x_i)}) \in [0,T]\times E$. Then, we have
    \begin{align}
        \label{eq:MF-symmetric-controls}
        v(t,\lambda) &= v(t,\mathfrak{s}\cdot\lambda)
        \eqsp,
        \quad\text{ for }\mathfrak{s}\in\mathfrak{S}_\Vc
        \eqsp,
        \qquad\text{ and }\qquad
        v(t,\lambda) = \inf_{\beta\in\CtrlStandard^\mathfrak{s}_{(t,\lambda)}} J(t,\lambda;\beta)
        \eqsp.
    \end{align}
\end{proposition}
    
\begin{proof}
    \emph{Step 1: Permutation invariance and symmetry of the value function.}
    In the mean-field setting the coefficients depend only on $(x,\pi(\lambda),a)$, hence they are invariant under any relabeling (permutation) $\mathfrak{s}\in\mathfrak{S}_\Vc$. If $v$ is the value function, then
    \begin{align*}
        u(t,\lambda):=v\bigl(t,\mathfrak{s}\cdot \lambda\bigr)
    \end{align*}
    solves the same HJB equation \eqref{eq:HJB} with the same terminal condition as $v$. This is a consequence of the fact that the Hamiltonian $\Hb_{\cdot}$ is built from $b^{\text{MF}}$, $\sigma^{\text{MF}}$, $\gamma^{\text{MF}}$, $(p^{\text{MF}}_k)_{k\geq 0}$, $\psi^{\text{MF}}$, and $\Psi^{\text{MF}}$ and depends only on $(x,\pi(\lambda))$. Therefore, it is permutation invariant. By the characterization of the control problem via viscosity solutions and the comparison principle (\Cref{prop:viscosity_solution} and \Cref{prop:comparison_principle}), equation \eqref{eq:HJB} has at most one viscosity solution in the admissible class; hence $u=v$ and the first equation of \eqref{eq:MF-symmetric-controls} is proved.

    \emph{Step 2: Symmetric measurable selector for the Hamiltonian.}
    We now turn into the second equation of \eqref{eq:MF-symmetric-controls}. Since $\CtrlStandard^\mathfrak{s}_{(t,\lambda)}\subset \CtrlStandard$, we have that
    \begin{align*}
        v(t,\lambda) \leq \inf_{\beta\in\CtrlStandard^\mathfrak{s}_{(t,\lambda)}} J(t,\lambda;\beta)
        \eqsp.
    \end{align*}
    To prove the reverse inequality, we need to show that there exists a symmetric control $\beta\in\CtrlStandard^\mathfrak{s}_{(t,\lambda)}$ such that
    \begin{align*}
        v(t,\lambda) + \epsilon \geq J(t,\lambda;\beta)
        \eqsp,
    \end{align*}
    for an arbitrarily chosen $\epsilon>0$.

    Fix $\epsilon>0$. 
    From the previous step, we have that the Hamiltonian $\Hb_{\cdot}$ depends only on $(t,x,\pi(\lambda))$ and not on the particle labels.
    Note that the number of particles $|\Vc|$ is a mean-field quantity, as $|\Vc|=\langle 1,\pi(\lambda)\rangle$.
    Define the set-function $\Hb^{\text{MF}}_\epsilon$ as
    \begin{align*}
        \Hb^{\text{MF}}_\epsilon
        \big(\Vec{x}_\Vc,r, q_\Vc,M_\Vc,(r_{(i,\ell)})_{i\in\Vc,\ell\in\N}\big):=\bigg\{
            \vec{a}_{|\Vc|}\in A^{|\Vc|}~:&~ 
            \mathfrak{b}_\Vc\left(\Vec{x}_\Vc, \Vec{a}_\Vc\right)^\top q_\Vc
            +  \frac{1}{2}\text{Tr}\left(
                \Sigma_\Vc(\Sigma_\Vc)^\top\left(\Vec{x}_\Vc, \Vec{a}_\Vc\right)D^2 M_\Vc
            \right)
        \\
        &+ \sum_{i\in\Vc} \gamma\left(i,x_i, \iota^{-1}(\Vec{x}_\Vc),a_i\right)\Bigg(\sum_{k\geq 0}r_{(i,k)}\eqsp p_k\left(i,x_i, \iota^{-1}(\Vec{x}_\Vc),a_i\right)- r \Bigg)
        \\
        &+ \sum_{i\in\Vc} \psi\left(i,x_i, \iota^{-1}(\Vec{x}_\Vc),a_{\Vc,i}\right)
            \leq \Hb_\Vc\big(\Vec{x}_\Vc, v_\Vc, D v_\Vc,D^2 v_\Vc,(v_{\Vc^i_\ell})_{i\in\Vc,\ell\in\N}\big)+\epsilon
        \bigg\}
        \eqsp.
    \end{align*}
    For each $(t,x,\pi(\lambda))$, the set $\Hb^{\text{MF}}_\epsilon$ is nonempty and depends only on $(t,x,\pi(\lambda))$. Following the same lines of \Cref{prop:App-DPP:measurable_selection}, there exists a Borel selector $a^\epsilon$ such that
    \begin{align*}
        a^\epsilon\big(\Vec{x}_\Vc,r, q_\Vc,M_\Vc,(r_{(i,\ell)})_{i\in\Vc,\ell\in\N}\big)\in \Hb^{\text{MF}}_\epsilon\big(\Vec{x}_\Vc,r, q_\Vc,M_\Vc,(r_{(i,\ell)})_{i\in\Vc,\ell\in\N}\big)
        \eqsp.
    \end{align*}
    Since $\Hb^{\text{MF}}_\epsilon$ is label-free, $a^\ast$ is \emph{symmetric} in the sense that if two particles share the same local arguments, they receive the same action.

    \emph{Step 3: Construction of a symmetric optimal feedback and conclusion.}
    Let $v$ be the viscosity solution of \eqref{eq:HJB}. Define the feedback control $\hat\beta$ by
    \begin{align*}
        \hat\beta^i_s \;=\; a^\ast \Big(s-,\eqsp Y^{i,\hat\beta}_{s-},\eqsp \pi(\xi^{\hat\beta}_{s-}),\eqsp D v_\Vc(\cdot),\eqsp D^2 v_\Vc(\cdot),\eqsp \big(v_{\Vc^i_\ell}(\cdot)\big)_{i\in\Vc,\ell\in\N}\Big),
    \end{align*}
    for $s\in[t,T]$ and $i\in\Vc^{t,\lambda;\hat\beta}_s$, where $\xi^{\hat\beta}$ is the solution of \eqref{SDE:strong} with initial condition $\xi^{t,\lambda;\hat\beta}_t = \lambda$ and control $\hat\beta$. The feedback $\hat\beta$ is well-defined and measurable, as it is constructed from the viscosity solution $v$ and the selector $a^\epsilon$, which are both measurable functions. Moreover, it is symmetric by construction, as it depends only on the empirical distribution $\pi(\xi^{\hat\beta}_s)$ and the local arguments of the particles.
    We consider the left limit $s-$ to ensure that the feedback is a predictable process.

    Moreover, by \Cref{prop:viscosity_solution} and \Cref{prop:comparison_principle}, the feedback $\hat\beta$ is $\epsilon$-optimal for \eqref{eq:def:cost_function}. Therefore, since $\epsilon>0$ was arbitrarily chosen, the search for optimizers can be restricted, without loss of generality, to the class of symmetric controls.
\end{proof}

From the invariance with respect to permutations established in \eqref{eq:MF-symmetric-controls}, the control problem can be reformulated in terms of the empirical distribution of the particles. In particular, the value function, cost functional, and associated dynamics depend only on the empirical measure of positions and not on the specific configuration of labeled particles. Hence, the natural state space in the MF framework is $\Nc(\R^d)$, rather than the configuration space $E$.


%% file: main_text/examples.tex
\subsection{Examples}
\label{Section:LQ}
We next present two examples of regular solutions within the linear-quadratic framework in the mean-field regime. 

\paragraph{Standard linear-quadratic case.}
We follow the path outlined in \citet{Pham:Conditional_LQ} and \citet{Pham_Basei:LQ}. Let $A:=\R^q$, $d^\prime=d$ and let the coefficients be as follows:
\begin{align*}
    b_t(x,\lambda,a)= B_t x+ \bar B_t a\eqsp,
    \qquad
    \sigma_t(x,\lambda,a)= \sigma_t\I\eqsp,
    \qquad
    \gamma_t(x,\lambda,a)= \gamma_t\eqsp, 
    \qquad
    p_k(x,\lambda,a)= p_k\eqsp,
\end{align*}
with $\I$ being the identity matrix, and $B$, $\bar B$, $\bar \sigma$, $\bar \gamma$ are bounded valued in $\R^{d \times d}$, $\R^{d \times p}$, $\R^{d \times d}$ and $\R_+$ respectively.

Let $\psi$ and $\Psi$ be as
\begin{align*}
    \psi_t(x,\lambda,a)= x^\top C_t x+ c_t\langle1,\lambda\rangle +a^\top \bar C_t a\eqsp,
    \qquad
    \Psi(\lambda) = \int_{\R^d}x^\top H x + h\langle1,\lambda\rangle^2\eqsp,
\end{align*}
where $t\mapsto C_t$ (resp. $t\mapsto \bar C_t$) is a bounded function in $\mathbb{S}^d$ (resp. $\mathbb{S}^q$), 
$t\mapsto c_t\in\R_+$ is bounded, $H\in \mathbb{S}^d$, and $h\geq0$.

We shall make the following assumptions on the coefficients of the model:
\begin{enumerate}[(i)]
    \item $C$ and $H$ are non-negative a.s.;
    \item $\bar C$ is uniformly positive definite, \ie, $\bar C_t\geq \epsilon\I_q$, for some $\epsilon>0$.
\end{enumerate}

We are now ready to use \Cref{Prop:verification2} by seeking a field $\big\{w_t(\lambda): \lambda\in\Nc(\R^d), t\in[0,T]\big\}$ that satisfies the local (sub)martingality conditions.

Let $w$ be as follows
\begin{align*}
    w_t(\lambda)= w^1_t(\lambda)+ w^2_t(\lambda) + w^3_t(\lambda), \qquad\text{ with }
    w^1_t(\lambda)=\int_{\R^d}x^\top Q_t x \lambda(dx)\eqsp,\quad
    w^2_t(\lambda) =p_t\langle1,\lambda\rangle^2\eqsp,
    \quad w^3_t(\lambda) = \bar p_t\langle1,\lambda\rangle\eqsp,
\end{align*}
for some funnctions $(Q,p,\bar p)$ with values in $\Sb^d\times\R\times \R$ such that
\begin{align*}
\begin{cases}
    dQ_t=\dot{Q}_t dt,\quad \text{ for }t\in[0,T], \quad Q_T = H\eqsp,\\
    dp_t=\dot{p}_tdt,\quad \text{ for }t\in[0,T], \quad p_T = h\eqsp,\\
    d\bar p_t=\dot{\bar p}_tdt,\quad \text{ for }t\in[0,T], \quad \bar p_T = 0\eqsp.
\end{cases}
\end{align*}
The terminal conditions ensure that $w_t(\lambda)= \Psi(\lambda)$. Now, we need to determine the generators $\dot{Q}$, $\dot{p}$ and $\dot{\bar p}$ to satisfy~\eqref{eq:verification_thm:inf_condition}.
From \eqref{SDE:strong}, Itô's formula yields
\begin{align}
\label{eq:LS:diffusion_w}
    \begin{split}
        w\left(t, \mu_{t}\right) &+ \int_0^t\int_{\R^d} \psi(x,\mu_u,\mathfrak{a}_u(x))\mu_u(dx)du
        = w\left(0, \mu_{0}\right) +
        \int_0^{t}\int_{\R^d}\Dc_u(x,\mu_u,\mathfrak{a}_u(x), Q_u, p_u, \bar p_u)\mu_u(dx)du + \mathbb{M}_t\eqsp,
    \end{split} 
\end{align}
with
\begin{align*}
    \Dc_u(x,\lambda,a, Q, p, \bar p):=&~ x^\top \dot{Q} x + \dot{p}\langle1,\lambda\rangle + \dot{\bar p} + \left( B_u x + \bar B_u a\right)^\top Q x 
    + x^\top Q \left( B_u x + \bar B_u a\right) 
    + \sigma_u^2 \text{Tr}(Q)\\
    &
    + (\gamma_u M_1) x^\top Q x + p \gamma_u\left(M_2 + M_1\langle1,\lambda\rangle\right) +  \bar p \gamma_u M_1 + x^\top C_u x+ c_u\langle1,\lambda\rangle +a^\top \bar C_u a
    \eqsp,
\end{align*}
$\mathbb{M}$ a martingale (after an eventual localization), $M_1:=\sum_{k\geq0}(k-1)p_k$, and $M_2:=\sum_{k\geq0}(k-1)^2p_k$. Completing the square in $\Dc$, we obtain
\begin{align*}
    \Dc_u(x,\lambda,a, Q, p, \bar p)
    :=&~
     \left(\dot{p} + p \gamma_u M_1 + c_u\right)\langle1,\lambda\rangle
    +\left(\dot{\bar p}+ \sigma_u^2 \text{Tr}(Q) + \bar p \gamma_u M_1 + p \gamma_u M_2\right) 
    + (a-\hat a_u(x, Q))^\top\bar C_u(a-\hat a_u(x, Q))
    \\
    &+x^\top \left(\dot{Q} + B_u^\top Q+ Q B_u + (\gamma_u M_1)Q + C_u 
    + \left(\bar B_u Q+\bar B_u^\top Q \right)^\top\bar C_u^{-1}\left(\bar B_u Q +\bar B_u^\top Q \right)\right) x 
    \eqsp,
\end{align*}
where
\begin{align*}
    \hat a_u(x, Q):= - \bar C_u^{-1}\left(\bar B_u Q +\bar B_u^\top Q \right)x\eqsp.
\end{align*}
Therefore, whenever
\begin{align}
\label{eq:LQ:Riccati}
    \begin{split}
        \dot{Q} + B_u^\top Q+ Q B_u + (\gamma_u M_1)Q + C_u 
            + 2Q\left(\bar B_u \bar C_u^{-1}\bar B_u + \bar B_u^\top \bar C_u^{-1}\bar B_u\right)Q=&~0\eqsp,\\
        \dot{p} + p \gamma_u M_1 + c_u=&~0\eqsp,
        \\
        \dot{\bar p}+ \sigma_u^2 \text{Tr}(Q) + \bar p \gamma_u M_1 + p \gamma_u M_2=&~0\eqsp,
    \end{split}
\end{align}
holds for $t\in[0,T]$, we have
\begin{align*}
    \Dc_u(x,\lambda,a, Q, p, \bar p)=\left(a-\hat a_u(x, Q)\right)^\top\bar C_u(a-\hat a_u(x, Q))\eqsp.
\end{align*}
Therefore, $\Dc\geq 0$, for $a\in A$ and it is zero for $a=\hat a_u(x, Q)$. Additionally, it is worth noting~\eqref{eq:LQ:Riccati} 
admit a solution since the first equation is a conventional Riccati equation, while the remaining two are linear ODEs.

This means that if the system of equations~\eqref{eq:LQ:Riccati} is satisfied, from
\eqref{eq:LS:diffusion_w} and the fact that $\Dc\geq 0$, we get the local submartingale property \textit{(ii)} of~\Cref{Prop:verification2}. Moreover, it is clear that it is zero for $\mathfrak{a}_u(x):= \hat a_u(x, Q)$, with $Q$ solution to the first equation in~\eqref{eq:LQ:Riccati}, satisfying the local martingale property \textit{(iii)} of~\Cref{Prop:verification2}. Therefore, such a control is an optimal one.

\paragraph{A Kinetic Example.}

In the case of a standard diffusion, we consider controls $\beta$ such that the diffusion satisfies the following SDE
\begin{align*}
    dX_t=\big(b(t,X_t) + \beta_s\big)dt + \sigma dB_t\eqsp,
\end{align*}
with $b$ Lipschitz in $x$ uniformly in $t$ and $\sigma$ a positive constant. In this setting, we look for a minimization of the cost function $\E\left[\frac{1}{2}\int_0^T|\beta_s|^2\right]$, which is usually called the kinetic energy for the controlled diffusion.

We adapt this problem to the case of branching processes. Let $A:=\R^q$, $d^\prime=d$ and let the coefficients be as follows:
\begin{align*}
    b_t(x,\lambda,a)= b(t,x) + a\eqsp,
    \qquad \sigma_t(x,\lambda,a)= \I\eqsp,
    \qquad
    \gamma_t(x,\lambda,a)= \gamma_t(x)\eqsp,
    \qquad
    p_k(x,\lambda,a)= p_k(x)\eqsp,
\end{align*}
with $b$, $\gamma$ and $p_k$ satisfying~\eqref{eq:bound_b_sigma_Lipschitz},~\eqref{eq:bound_b_sigma_gamma} and~\eqref{eq:bound:order1_2_Phi}. Let $\psi(x,\lambda,a):= \frac{1}{2}|a|^2$. We seek for a field $\left\{w_t(\lambda): \lambda\in\Nc(\R^d), t\in[0,T]\right\}$ such that
\begin{align*}
    w_t(\lambda)= \int_{\R^d}h(t,x)\lambda(dx)\eqsp,
\end{align*}
for a certain function $h$. From~\eqref{SDE:strong}, applying Itô's formula, we have
\begin{align}
\label{eq:LS:diffusion_w:kinetic}
    w\left(t, \mu_{t}\right) &+ \int_0^t\int_{\R^d} \psi(x,\mu_u,\mathfrak{a}_u(x))\mu_u(dx)du  = w\left(0, \mu_{0}\right) +
    \int_0^{t}\int_{\R^d}\Dc_u(x,\mu_u,\mathfrak{a}_u(x), h)\mu_u(dx)du + \mathbb{M}_t\eqsp, 
\end{align}
where
\begin{align*}
    \Dc_t(x,\lambda,a, h):= \partial_t h + b(t,x)^\top Dh + a^\top Dh + \frac{1}{2}\Delta h + \frac{1}{2}|a|^2 + \phi(t,x)h\eqsp,
\end{align*}
with $\phi(x):=\gamma_t(x)\Big(\sum_{k\geq0}kp_k(x) - 1\Big)$, $\mathbb{M}$ a martingale (after an eventual localization), and $\Delta$ the Laplacian operator. Operating as in the previous example, we see that whenever $h$ satisfies the following PDE
\begin{align}\label{eq:kinetic:semiparabolic}
    \begin{cases}
        \partial_t h + b(t,x)^\top Dh - \frac{1}{2}|Dh|^2 + \frac{1}{2}\Delta h + \phi(t,x)h=0\eqsp,
        \\
        h(T,x)=0\eqsp,
    \end{cases}
\end{align}
we have
\begin{align*}
    \Dc_u(x,\lambda,a, h)=\frac{1}{2}|a+Dh|^2\eqsp.
\end{align*}
This means that under~\eqref{eq:kinetic:semiparabolic}, $\Dc\geq 0$, for $a\in A$ and is zero for $a=-Dh$. Therefore, under~\eqref{eq:kinetic:semiparabolic}, we get property \textit{(ii)} of~\Cref{Prop:verification2}, and property \textit{(iii)}, for $\mathfrak{a}_s(x):= -Dh(s,x)$, showing that this control is an optimal one. The solution of~\eqref{eq:kinetic:semiparabolic} is standard and is an application of the Hopf--Cole transformation.

%% file: main_text/conclusion.tex
\section{Conclusion}
\label{Section:Conclusion}

In this work, we studied the stochastic control of interacting branching diffusion processes within a general configuration framework. The first main contribution is the formulation of the associated \emph{HJB equation,} obtained through a bijection with the topological union \(\sqcup_{\Vc \in \Igen}\R^{d|\Vc|}\). This structure allowed us to leverage differential tools in finite-dimensional Euclidean spaces to analyze the control problem.

We then provided a \emph{viscosity characterization} of the value function, including the proof of a \emph{comparison principle}, ensuring uniqueness within the class of functions satisfying the prescribed growth conditions. This establishes a rigorous link between the control problem and its PDE characterization.

Finally, we considered the \emph{mean-field reduction}, showing how the symmetry of interactions simplifies the problem by restricting the analysis to empirical measures and symmetric controls, while preserving the essential probabilistic features of the general setting.

This work also opens the way to further developments. In particular, the \emph{relaxed formulation} of the mean-field control problem, developed in the companion paper \citet{ocello2023controlled}, represents a fundamental step in extending the theory. Such a formulation is crucial to address compactness and existence issues and paves the way to studying \emph{scaling limits}, as in \citet{ocello2025controlled}, where superprocesses naturally arise as limiting objects. Furthermore, combining this path for mean-field interactions with the framework developed in \citet{de2024mean,de2025linear} would provide a natural avenue to study scaling limits in heterogeneous systems.

%% file: appendix/well-posed-ctrl-pb.tex
\section{Well-posedness of the optimization problem}

\subsection{Proof of~\Cref{prop:existence_strong_branching}}
\label{Appendix:Proof:existence_strong_branching}

    Fix $(t,\lambda=\sum_{i\in \Vc}\delta_{(i,x^i)}) \in\R_+\times E$, and $\beta\in\CtrlStandard$. Using induction, we build the branching events of the population. We later show that such a process satisfies~\eqref{SDE:strong} and is well-posed. Since for each branch, the diffusion $\sigma$ and the jump rate $\gamma$ are bounded and the drift $b$ is linear in $(x,a)$, to ensure a well-posedness we must have that the mass does not explode in finite time, \ie,~\eqref{eq:non-explosion-moment1_mass}, and the first moment bounded, \ie,~\eqref{eq:non-explosion-moment1_population}.
    
    Define by induction an increasing sequence of stopping time $(\tau_k)_{k\in\N}$, a sequence of random variables $(V_k)_{k\in\N}$ valued in the set of finite subsets of $\Ic$ and a sequence of processes $(Y^{i,\beta}, i \in V_k)_{k\in\N}$ such that
    \begin{align*}
        \xi^{t,\lambda;\beta}_s = \sum_{k\geq1}\1_{\tau_{k-1}\leq s<\tau_k}\sum_{i\in V_k}\delta_{(i,Y^{i,\beta}_s)}\eqsp.
    \end{align*}
    We set $\tau_0=t$, $V_0=\Vc$, and $Y^{i,\beta}_t :=x^i$, for $i\in \Vc$. Then, given $\tau_{k-1}$ and $V_{k-1}$, define $\tau_k$ as
    \begin{align*}
        \tau_k =\inf\left\{s\in\left(\tau_{k-1},T\right]: \exists i \in V_{k-1}, ~ Q^i((\tau_{k-1},s]\times[0,C_\gamma])=1 \right\}\eqsp.
    \end{align*}
    Define $\mathcal{Y}^k$, $\mathfrak{b}^k(\mathcal{Y}^k,\beta_s)$, $\Sigma^k(\mathcal{Y}^k,\beta_s)$, and $\mathcal{W}^k$, as
    \begin{align*}
        & \mathcal{Y}^k_s := 
        \begin{pmatrix}
            Y^{i_1,\beta}_s\\
            \vdots
            \\
            Y^{i_{|V_{k-1}|},\beta}_s
        \end{pmatrix}\eqsp,
        \quad
        \mathfrak{b}^k(\mathcal{Y}^k_s,\beta_s) := 
        \begin{pmatrix}
            b\left(i_1,Y^{i_1,\beta}_s, \sum_{i\in V_{k-1}} \delta_{(i,Y^{i,\beta}_s)},\beta_s^{i_1}\right)\\
            \vdots
            \\
            b\left(i_{|V_{k-1}|},Y^{i_{|V_{k-1}|},\beta}_s, \sum_{i\in V_{k-1}} \delta_{(i,Y^{i,\beta}_s)},\beta_s^{i_{|V_{k-1}|}}\right)
        \end{pmatrix}, \\
        &\Sigma^k(\mathcal{Y}^k_s,\beta_s) :=
        \text{Diag}_{|V_{k-1}|}\begin{pmatrix}
            \sigma\left(i_1,Y^{i_1,\beta}_s, \sum_{i\in V_{k-1}} \delta_{(i,Y^{i,\beta}_s)},\beta_s^{i_1}\right)\\
            \vdots
            \\
            \sigma\left(i_{|V_{k-1}|},Y^{i_{|V_{k-1}|},\beta}_s, \sum_{i\in V_{k-1}} \delta_{(i,Y^{i,\beta}_s)},\beta_s^{i_{|V_{k-1}|}}\right)
        \end{pmatrix},\quad
        \mathcal{W}^k_s = 
        \begin{pmatrix}
        W^{i_1}_s\\
        \vdots
        \\
        W^{i_{|V_{k-1}|}}_s
        \end{pmatrix}\eqsp,
    \end{align*}
    taking values in $\R^{d|V_{k-1}|}$, $\R^{d|V_{k-1}|}$, $\R^{d|V_{k-1}|\times d^\prime|V_{k-1}|}$, and $\R^{d^\prime|V_{k-1}|}$ respectively, where 
    the matrix $\text{Diag}_m$ is a diagonal matrix of size $dm\times d^\prime m$, for $m\in\N$,
    and the indices $i_1,\dots,i_{|V_{k-1}|}\in V_{k-1}$ are taken w.r.t.\ the total order $\leq$ in $\Ic$. From Assumption H\ref{hypH:model_parameters}, together with~\eqref{eq:bound_d_1_Rd-measures-to_R_d-2}, we have that the coefficients 
    $\mathfrak{b}^k$ and $\Sigma^k$ are Lipschitz continuous in $\R^{d|V_{k-1}|}$ uniformly in the control, with Lipschitz constant that may depend on $|V_{k-1}|$. Therefore, $\mathcal{Y}^k$ is uniquely (up to indistinguishability) defined as the continuous and adapted process satisfying
    \begin{align*}
    \mathcal{Y}^k_s= \mathcal{Y}^k_{\tau_{k-1}} +\int^s_{\tau_{k-1}}\mathfrak{b}^k(\mathcal{Y}^k_u,\beta_u) du + \int^s_{\tau_{k-1}} \Sigma^k(\mathcal{Y}^k_u,\beta_u)d\mathcal{W}^k_u\eqsp, \quad \P-\text{a.s.}
    \end{align*}
    
    Describing what happens at branching events $\tau_k$, we can conclude the construction of the branching process. Given the definition of $\tau_k$, there is an (almost surely) unique label, that we denote $\hat i_k \in V_{k-1}$, such that
    \begin{align*}
    Q^{\hat i_k} \left((\tau_{k-1},\tau_{k}]\times[0,C_\gamma]\right)=1\eqsp.
    \end{align*}
    Let $\chi_k$ the $[0,C_\gamma]$-valued random variable such that $(\tau_{k},\chi_k)$ belongs to the support of $Q^{\hat i_k}$. We set $V_k$ as
    \begin{align*}
        V_k:= \begin{cases}
            V_{k-1}, &\text{ if }\chi_k\in \left[\gamma\left(\hat i_k,Y^{\hat i_k,\beta}_{\tau_{k}}, \sum_{i\in V_{k-1}} \delta_{(i,Y^{i,\beta}_{\tau_{k}})},\beta_{\tau_{k}}^{\hat i_k}\right), C_\gamma\right],\\
            V_{k-1}\backslash \left\{\hat i_k\right\},&\text{ if }\chi_k\in I_0\left(\hat i_k,Y^{\hat i_k,\beta}_{\tau_{k}}, \sum_{i\in V_{k-1}} \delta_{(i,Y^{i,\beta}_{\tau_{k}})},\beta_{\tau_{k}}^{\hat i_k}\right),\\
            V_{k-1}\backslash \left\{\hat i_k\right\}\cup \left\{\hat i_k 0,\dots,\hat i_k (\ell-1)\right\},&\text{ if }\chi_k\in I_\ell\left(\hat i_k,Y^{\hat i_k,\beta}_{\tau_{k}}, \sum_{i\in V_{k-1}} \delta_{(i,Y^{i,\beta}_{\tau_{k}})},\beta_{\tau_{k}}^{\hat i_k}\right)\text{ for }\ell\geq 1\eqsp,
        \end{cases}
    \end{align*}
    where we impose the continuity for the flow for the off-spring, \ie, $Y^{i,\beta}_{\tau_{k}} := Y^{\hat i_k,\beta}_{\tau_{k}}$, for $i\in V_{k}\backslash V_{k-1}$.
    
    We prove that this process satisfies the SDE~\eqref{SDE:strong} by induction. Since $\tau_0 = t$, it is trivially satisfied. If it holds true up to $\tau_{k-1}$, we have
    \begin{align}\label{eq:prop_non_expl:SDE_between_tau_k-tau_k_1}
        \langle\varphi,\xi^{t,\lambda;\beta}_{s\wedge \tau_k}\rangle =
        \1_{s\leq \tau_{k-1}} \langle\varphi,\xi^{t,\lambda;\beta}_{s}\rangle + \1_{\tau_{k-1}<s<\tau_{k}} \sum_{i\in V_{k-1}} \varphi \left(i,Y^{i,\beta}_{s}\right) + \1_{s\geq\tau_{k}} \sum_{i\in V_{k}} \varphi \left(i,Y^{i,\beta}_{\tau_{k}}\right)\eqsp.
    \end{align}
    The first term on the r.h.s.\ satisfies~\eqref{SDE:strong} by the induction hypothesis. We apply Itô's formula for each branch to deal with the second one. Finally, the third term is equal to
    \begin{align*}
        \sum_{i\in V_{k}}
        \varphi\left(i,Y^{i,\beta}_{\tau_{k}-}\right) 
        =& 
        \sum_{i\in V_{k-1}}
        \varphi \left(i,Y^{i,\beta}_{\tau_{k}-}\right)
        - 
        \1_{
            \chi_k\in \left[0,\gamma\left(\hat i_k,Y^{\hat i_k,\beta}_{\tau_{k}-}, \sum_{i\in V_{k-1}} \delta_{(i,Y^{i,\beta}_{\tau_{k}-})},\beta_{\tau_{k}-}^{\hat i_k}\right)\right)
        } \varphi \left(\hat i_k,Y^{\hat i_k,\beta}_{\tau_{k}-}\right)\\
        &\qquad\qquad\qquad
        +\sum_{\ell\geq 1}
        \1_{
            \chi_k\in I_\ell\left(Y^{\hat i_k,\beta}_{\tau_{k}-}, \sum_{i\in V_{k-1}} \delta_{(i,Y^{i,\beta}_{\tau_{k}-})},\beta_{\tau_{k}-}^{\hat i_k}\right)
        } \sum_{l= 0}^{\ell-1}
        \varphi \left(\hat i_k l,Y^{\hat i_k l,\beta}_{\tau_{k}-}\right)
        \eqsp,
    \end{align*}
    which coincides with the integral w.r.t. the Poisson random measures over $(\tau_{k-1},\tau_k]$. Therefore,~\eqref{SDE:strong} is satisfied up to $\tau_k$ and we conclude by induction.
    
    As previously recalled, to achieve a well-posedness of the population, the last missing ingredients are~\eqref{eq:non-explosion-moment1_mass} and~\eqref{eq:non-explosion-moment1_population}. Let $\left\{\theta_n\right\}_{n\in\N}$ be
    \begin{align*}
        \theta^1_n:=\inf\left\{s\geq t: |V_s|\geq n \right\}
        \eqsp,\quad
        \theta^2_n:=\inf\left\{s\geq t: \sum_{i\in \Vc^{t,\lambda;\beta}_u}\left|Y^{i,\beta}_u\right|\geq n \right\}
        \eqsp,\quad\text{ and }\quad
        \theta_n:=\theta^1_n\wedge \theta^2_n\eqsp.
    \end{align*}
    The first part of the proof ensures that $\xi^{t,\lambda;\beta}$ is well-posed and satisfies~\eqref{SDE:strong} up to $\theta_n$. Let us first focus on~\eqref{eq:non-explosion-moment1_mass} and apply~\eqref{SDE:strong} to the function $(i,x)\mapsto 1$, obtaining
    \begin{align*}
        |\Vc^{t,\lambda;\beta}_{s\wedge\theta_n}| =& |\Vc^{t,\lambda;\beta}_{t}|
        +
        \int_{\left(t,s\wedge \theta_n\right]\times \R_+}
        \sum_{i\in \Vc^{t,\lambda;\beta}_{u-}}
        \sum_{k\geq0}
        (k-1)
        \1_{I_k\left(i,Y^{i,\beta}_{u-},\xi^{t,\lambda;\beta}_{u-},\beta^i_{u}\right)}(z)
        Q^i(dudz)
        \eqsp,
    \end{align*}
    for $s\geq t$.
    Applying Itô's formula, we also obtain
    \begin{align*}
        |\Vc^{t,\lambda;\beta}_{s\wedge\theta_n}|^2
        =&~
        |\Vc^{t,\lambda;\beta}_{t}|^2
        +
        \int_{\left(t,s\wedge \theta_n\right]\times \R_+}
        \sum_{i\in \Vc^{t,\lambda;\beta}_{u-}}
        \sum_{k\geq0}
        \left(
            \left(|\Vc^{t,\lambda;\beta}_{u-}| + k-1\right)^2 - |\Vc^{t,\lambda;\beta}_{u-}|^2
        \right)\1_{I_k\left(i,Y^{i,\beta}_{u-},\xi^{t,\lambda;\beta}_{u-},\beta^i_{u}\right)}(z)
        Q^i(dudz)
        \\
        =&~
        |\Vc^{t,\lambda;\beta}_{t}|^2+
        \int_{\left(t,s\wedge \theta_n\right]\times \R_+}
        \sum_{i\in \Vc^{t,\lambda;\beta}_{u-}}
        \sum_{k\geq0}\left(
        2(k-1)|\Vc^{t,\lambda;\beta}_{u-}| + (k-1)^2\right) \1_{I_k\left(i,Y^{i,\beta}_{u-},\xi^{t,\lambda;\beta}_{u-},\beta^i_{u}\right)}(z)
        Q^i(dudz)\eqsp.
    \end{align*}
    Therefore, we get
    \begin{align*}
        \sup_{u\in[t,s]}|\Vc^{t,\lambda;\beta}_{u\wedge\theta_n}| \leq & |\Vc^{t,\lambda;\beta}_{t}|+ \int_{\left(t,s\wedge \theta_n\right]\times \R_+}
        \sum_{i\in \Vc^{t,\lambda;\beta}_{u-}}\sum_{k\geq 1}
        (k-1)
        \1_{I_k\left(i,Y^{i,\beta}_{u-},\xi^{t,\lambda;\beta}_{u-},\beta^i_{u}\right)}(z)
        Q^i(dudz),\\
        \sup_{u\in[t,s]}|\Vc^{t,\lambda;\beta}_{u\wedge\theta_n}|^2 \leq & |\Vc^{t,\lambda;\beta}_{t}|^2+ \int_{\left(t,s\wedge \theta_n\right]\times \R_+}\sum_{i\in \Vc^{t,\lambda;\beta}_{u-}}\sum_{k\geq 1}\left(
            2(k-1)|\Vc^{t,\lambda;\beta}_{u-}| + (k-1)^2\right)\1_{I_k\left(i,Y^{i,\beta}_{u-},\xi^{t,\lambda;\beta}_{u-},\beta^i_{u}\right)}(z)
            Q^i(dudz)\eqsp,
    \end{align*}
    and, taking the expectation,
    \begin{align*}
        \E\left[\sup_{u\in[t,s]}|V_{u\wedge\theta_n}|\right] \leq & |\Vc^{t,\lambda;\beta}_{t}|+ \E\left[\int_t^{s\wedge \theta_n}\sum_{i\in \Vc^{t,\lambda;\beta}_{u}}\gamma\left(i,Y^{i,\beta}_{u}, \xi^{t,\lambda;\beta}_u,\beta_{u}^{i}\right)\sum_{k\geq 1} (k-1) p_k\left(i,Y^{i,\beta}_{u}, \xi^{t,\lambda;\beta}_u,\beta_{u}^{i}\right)du\right]\\
        \leq & |\Vc^{t,\lambda;\beta}_{t}|+ C_\gamma C^1_\Phi\E\left[\int_t^{s\wedge \theta_n}\sup_{z\in[t,u]}|V_{z\wedge\theta_n}|\right],\\
        \E\left[\sup_{u\in[t,s]}|V_{u\wedge\theta_n}|\right] \leq & |\Vc^{t,\lambda;\beta}_{t}|+ C_\gamma(C^1_\Phi+C^2_\Phi)\E\left[\int_t^{s\wedge \theta_n}\sup_{z\in[t,u]}|\Vc^{t,\lambda;\beta}_{z\wedge\theta_n}|^2\right]
        \eqsp.
    \end{align*}
    Applying Grönwall's lemma, we obtain
    \begin{align*}
        \E\left[\sup_{u\in[t,s]}|\Vc^{t,\lambda;\beta}_{u\wedge\theta_n}|\right] \leq |\Vc^{t,\lambda;\beta}_{t}|e^{C_\gamma C^1_\Phi (s-t)}\eqsp,
        \qquad 
        \E\left[\sup_{u\in[t,s]}|\Vc^{t,\lambda;\beta}_{u\wedge\theta_n}|^2\right] \leq |\Vc^{t,\lambda;\beta}_{t}|^2 e^{C_\gamma (C^1_\Phi+C^2_\Phi)(s-t)}\eqsp.
    \end{align*}
    Since the bound is uniform in $n$, $\theta^1_n$ converges almost surely to infinity, and by Fatou's lemma, we retrieve~\eqref{eq:non-explosion-moment1_mass} and~\eqref{eq:non-explosion-moment2_mass}. This implies also~\eqref{eq:non-explosion-moment1_control}, since
    \begin{align*}
        \E\left[\int_t^s\sum_{i\in \Vc^{t,\lambda;\beta}_{u}}|\beta^i_u|du\right]
        \leq
        \E\left[\int_t^s|V_{u}|\sup_{i\in\Ic}|\beta^i_u|du\right]\leq
        \E\left[\sup_{u\in[t,s]}|V_{u}|\int_t^s\sup_{i\in\Ic}|\beta^i_u|du\right] \leq C\eqsp,
    \end{align*}
    where in the last inequality we used Cauchy--Schwartz inequality,~\eqref{eq:bound_sup_beta_2} and~\eqref{eq:non-explosion-moment2_mass}.
    
    Proving~\eqref{eq:non-explosion-moment1_population} is more subtle, as the SDE~\eqref{SDE:strong} cannot be applied directly. We see that~\eqref{eq:prop_non_expl:SDE_between_tau_k-tau_k_1} is still valid for $\varphi(i,x)=|x|$. Itô's formula yields, for $s \in(\tau_{k-1},\tau_k)$,
    \begin{align*}
        \sum_{i\in \Vc^{t,\lambda;\beta}_{k-1}}
        \left|Y^{i,\beta}_{s}\right|
        =&~
        \sum_{i\in \Vc^{t,\lambda;\beta}_{k-1}}
        \left|
            Y^{i,\beta}_{\tau_{k}} + 
            \int_{\tau_{k-1}}^s b\left(i,Y^{i,\beta}_{u}, \xi^{\beta}_{u},\beta^i_{u}\right)du +\int_{\tau_{k-1}}^s \sigma\left(i,Y^{i,\beta}_{u}, \xi^{\beta}_{u},\beta^i_{u}\right)dW^i_u
        \right|\\
        \leq& \sum_{i\in \Vc^{t,\lambda;\beta}_{k-1}}
        \left|Y^{i,\beta}_{\tau_{k}} \right| + \sum_{i\in \Vc^{t,\lambda;\beta}_{k-1}}\int_{\tau_{k-1}}^s \left|b\left(i,Y^{i,\beta}_{u}, \xi^{\beta}_{u},\beta^i_{u}\right)\right|du +
        \sum_{i\in \Vc^{t,\lambda;\beta}_{k-1}} \left|\int_{\tau_{k-1}}^s \sigma\left(i,Y^{i,\beta}_{u}, \xi^{\beta}_{u},\beta^i_{u}\right)dW^i_u \right|\\
        \leq& \sum_{i\in \Vc^{t,\lambda;\beta}_{k-1}}
        \left|Y^{i,\beta}_{\tau_{k}} \right| + C_b \int_{\tau_{k-1}}^s|V_{u}|du + C_b \sum_{i\in \Vc^{t,\lambda;\beta}_{k-1}}\int_{\tau_{k-1}}^s \left(\left|Y^{i,\beta}_{u}\right| + \left|\beta^i_{u}\right|\right)du +
        \sum_{i\in \Vc^{t,\lambda;\beta}_{k-1}}\left|\int_{\tau_{k-1}}^s \sigma\left(i,Y^{i,\beta}_{u}, \xi^{\beta}_{u},\beta^i_{u}\right)dW^i_u \right|
        \eqsp,
    \end{align*}
    where we have used the bound~\eqref{eq:bound_b_sigma_gamma} over the coefficient $b$ in the last inequality.
    Since the family of Brownian motions $\{W^i\}_{i\in\Ic}$ are indipendent from the one of Poisson measures $\{Q^i\}_{i\in\Ic}$, we have that taking the conditional expectation with respect to $\Fc_{\tau_{k-1}}$, we can apply the Burkholder--Davis--Gundy's inequalities \citep[see, \eg, Theorem 92,][]{Dellacherie:Meyer:B}.  This means that there exists a constant $C>0$ (which may change from line to line) such that
    \begin{align*}
        &\E\left[\sup_{u\in \left[\tau_{k-1}\wedge\theta_n,s\wedge\tau_k\wedge\theta_n\right]}\sum_{i\in \Vc^{t,\lambda;\beta}_{k-1}}\bigg|\int_{\tau_{k-1}\wedge\theta_n}^u \sigma\left(i,Y^{i,\beta}_{r}, \xi^{\beta}_{r},\beta^i_{r}\right)dW^i_{r} \bigg|\Bigg|\Fc_{\tau_{k-1}} \right] 
        \\
        &
        \leq C
        \E\left[\sum_{i\in \Vc^{t,\lambda;\beta}_{k-1}} \left(\int_{\tau_{k-1}\wedge\theta_n}^{s\wedge\tau_k\wedge\theta_n}\text{Tr}\left(\sigma \sigma^\top\left(i,Y^{i,\beta}_{u}, \xi^{\beta}_{u},\beta^i_{u}\right)\right) du \right)^{1/2}\Bigg|\Fc_{\tau_{k-1}} \right]\\
        & \leq C
        \E\left[\left(s\wedge\tau_k\wedge\theta_n - \tau_{k-1}\wedge\theta_n\right)|V_{k-1}|\Bigg|\Fc_{\tau_{k-1}} \right] = C
        \E\left[\int_{\tau_{k-1}\wedge\theta_n}^{s\wedge\tau_k\wedge\theta_n}|V_u|du\Bigg|\Fc_{\tau_{k-1}} \right]
        \eqsp,
    \end{align*}
    where we have used~\eqref{eq:bound_b_sigma_gamma} in the last line. Therefore, by induction, we have that there exists a constant $C>0$ (which may change from line to line) such that
    \begin{align*}
        \E\left[\sup_{u\in[t,s]}\sum_{i\in \Vc^{t,\lambda;\beta}_{u\wedge\theta_n}}\left|Y^{i,\beta}_{u\wedge\theta_n}\right|\right] \leq &
        \sum_{i\in V}|x^i|+ 
        C \left(\E\left[ \int_t^{s\wedge \theta_n} |V_{u}|du \right] +\E\left[ \int_t^{s\wedge \theta_n} \sum_{i\in \Vc^{t,\lambda;\beta}_{u}}\left|Y^{i,\beta}_{u}\right|du \right]
        +\E\left[ \int_t^{s\wedge \theta_n} \sum_{i\in \Vc^{t,\lambda;\beta}_{u}}\left|\beta^{i}_{u}\right|du \right] \right)\eqsp,
    \end{align*}
    where we have used~\eqref{eq:non-explosion-moment1_mass} and~\eqref{eq:non-explosion-moment1_control} to bound the term depending on the mass of the population. Applying Grönwall's lemma, we obtain
    \begin{align*}
    \E\left[\sup_{u\in[t,s]}\sum_{i\in \Vc^{t,\lambda;\beta}_{u\wedge\theta_n}}\left|Y^{i,\beta}_{u\wedge\theta_n}\right|\right] \leq 
    C \left(\sum_{i\in V}|x^i| +\E\left[ \int_t^{s} |V_{u}|du \right] +\E\left[ \int_t^{s} \sum_{i\in \Vc^{t,\lambda;\beta}_{u}}\left|\beta^{i}_{u}\right|du \right] \right)\eqsp.
    \end{align*}
    Since the bound is uniform in $n$, $\theta^2_n$ converges almost surely to infinity, and by Fatou's lemma, we retrieve~\eqref{eq:non-explosion-moment1_population}.

\subsection{Proof of~\Cref{Lemma:bound_moment2_population}}
\label{Appendix:Proof:bound_moment2_population}

    Fix $\big(t,\lambda=\sum_{i\in V}\delta_{(i,x^i)}\big) \in[0,T]\times E$, and $\beta\in\CtrlStandard$. Let $\left\{\theta_n\right\}_{n\in\N}$ be
    \begin{align*}
    \theta_n:=&\inf\left\{s\geq t: |\Vc^{t,\lambda;\beta}_s|\geq n \right\}\wedge \inf\left\{s\geq t: \sum_{i\in \Vc^{t,\lambda;\beta}_u}\left|Y^{i,\beta}_u\right|\geq n \right\}\eqsp.
    \end{align*}
    We have that $\xi^{t,\lambda;\beta}$ is satisfied~\eqref{SDE:strong} up to $\theta_n$. Applying~\eqref{SDE:strong} to the function $(i,x)\mapsto |x|^2$, we get
    \begin{align*}
        &\sum_{i\in \Vc^{t,\lambda;\beta}_{s\wedge\theta_n}}\left|Y^{i,\beta}_{s\wedge\theta_n}\right|^2
        \\
        =& \sum_{i\in V}|x^i|^2+ 
        \int_t^{s\wedge\theta_n}\sum_{i\in \Vc^{t,\lambda;\beta}_u} 2 \left(Y^{i,\beta}_u\right)^\top\sigma\left(i,Y^{i,\beta}_u,\xi^{t,\lambda;\beta}_u,\beta^i_u\right)dB^i_u
        +\int_t^{s\wedge\theta_n}\sum_{i\in \Vc^{t,\lambda;\beta}_u} 2 \left(Y^{i,\beta}_u\right)^\top b\left(i,Y^{i,\beta}_u,\xi^{t,\lambda;\beta}_u,\beta^i_u\right)du \\
        &
        +\int_t^{s\wedge\theta_n}\sum_{i\in \Vc^{t,\lambda;\beta}_u}\text{Tr}\left(\sigma \sigma^\top\left(i,Y^{i,\beta}_{u}, \xi^{\beta}_{u},\beta^i_{u}\right)\right) du
        + \int_{(t,s\wedge\theta_n]\times\R_+}\sum_{i\in \Vc^{t,\lambda;\beta}_{u-}}\sum_{k\geq0}
        (k-1)\left|Y^{i,\beta}_{u-}\right|^2\1_{I_k\left(i,Y^{i,\beta}_{u-},\xi^{t,\lambda;\beta}_{u-},\beta^i_{u}\right)}(z)
        Q^i(dudz)\eqsp.
    \end{align*}
    Taking the supremum in the interval $[t,s]$ and taking the expectation, we bound each term in the r.h.s. Applying Burkholder--Davis--Gundy's inequalities \citep[see, \eg, Theorem 92,][]{Dellacherie:Meyer:B} to the second term, there exists a constant $C>0$ (which may change from line to line) such that
    \begin{align*}
        &\E\left[\sup_{u\in[t,s]}\int_t^{u\wedge\theta_n}\sum_{i\in \Vc^{t,\lambda;\beta}_{r} } 2 \left(Y^{i,\beta}_{r} \right)^\top\sigma\left(i,Y^{i,\beta}_{r} ,\xi^{t,\lambda;\beta}_{r} ,\beta^i_{r} \right)dB^i_{r} \right]\\
        &~\leq C \E\left[\left(\int_t^{s\wedge\theta_n}\sum_{i\in \Vc^{t,\lambda;\beta}_u}  \left|Y^{i,\beta}_{u} \right|^2 \text{Tr}\left(\sigma \sigma^\top\left(i,Y^{i,\beta}_{u}, \xi^{\beta}_{u},\beta^i_{u}\right)\right) du \right)^{1/2}\right]\leq C \E\left[\int_t^{s\wedge\theta_n}\sum_{i\in \Vc^{t,\lambda;\beta}_u}  \left|Y^{i,\beta}_{u} \right|^2  du \right]\eqsp.
    \end{align*}
    From~\eqref{eq:bound_b_sigma_gamma} on the growth of $b$ and $\sigma$, the third and the fourth terms can be bounded as follows
    \begin{align*}
        &\E\left[\sup_{u\in[t,s]}\int_t^{u\wedge\theta_n}\sum_{i\in \Vc^{t,\lambda;\beta}_{r} }\left( 2 \left(Y^{i,\beta}_{r} \right)^\top b\left(Y^{i,\beta}_{r} ,\xi^{t,\lambda;\beta}_{r} ,\beta^i_{r} \right) + \text{Tr}\left(\sigma \sigma^\top\left(i,Y^{i,\beta}_{r}, \xi^{\beta}_{r},\beta^i_{r}\right)\right)\right)dr\right]
        \\
        &
        \leq C \E\left[\int_t^{s\wedge\theta_n}|V_u| + \sum_{i\in \Vc^{t,\lambda;\beta}_u}  \left|Y^{i,\beta}_u\right|^2 + \left|\beta^{i}_u\right|^2 du \right]\eqsp,
    \end{align*}
    using that $a^\top b \leq \frac{1}{2}\left(|a|^2 + |b|^2\right)$, for $a,b\in \R^d$. Finally, the last term gives
    \begin{align*}
        &\E\left[\sup_{u\in[t,s]}
        \int_{(t,u\wedge\theta_n]\times\R_+}\sum_{i\in \Vc^{t,\lambda;\beta}_{r-}}\sum_{k\geq0}
        (k-1)\left|Y^{i,\beta}_{r-}\right|^2\1_{I_k\left(i,Y^{i,\beta}_{r-} ,\xi^{t,\lambda;\beta}_{r-} ,\beta^i_{r}\right)}(z)
        Q^i(drdz)
        \right]
        \\
        &
        \leq \E\left[\int_t^{s\wedge\theta_n} \sum_{i\in \Vc^{t,\lambda;\beta}_{u}}\gamma\left(i,Y^{i,\beta}_u,\xi^{t,\lambda;\beta}_u,\beta^i_u\right) \sum_{k\geq1}
        (k-1)\left|Y^{i,\beta}_{u}\right|^2 p_k\left(i,Y^{i,\beta}_u,\xi^{t,\lambda;\beta}_u,\beta^i_u\right)du\right]
        \leq C \E\left[\int_t^{s\wedge\theta_n}\sum_{i\in \Vc^{t,\lambda;\beta}_u}  \left|Y^{i,\beta}_u\right|^2 du\right]\eqsp.
    \end{align*}
    Combining all the terms and using Gronwall's inequality first and Fatou's lemma then, we obtain~\eqref{eq:non-explosion-moment2_population}.

%% file: appendix/dpp.tex
\section{Dynammic programming principle}
\label{Appendix:DPP}

We prove in this section the dynamic programming principle (DPP) for the controlled branching process $\xi^{t,\lambda;\beta}$ introduced in~\Cref{Section:strong_form}. We closely follow the presentation in \citet{claisse18-v1} and \citet{kharroubi2024stochastic}, restating only the key results needed to establish \Cref{Prop:DPP}.

\paragraph{Canonical space and representation.}
Let $\mathcal{W} := C(\mathbb{R}_+, \mathbb{R}^d)$ be the space of continuous functions from the non-negative real line to $\mathbb{R}^d$, endowed with the topology of locally uniform convergence. This topology induces a Borel $\sigma$-algebra on $\mathcal{W}$, which we denote by $\mathscr{W}$, and which coincides with the $\sigma$-algebra generated by the canonical filtration $(\mathscr{W}_s)_{s \geq 0}$, where $\mathscr{W}_s$ is the smallest $\sigma$-algebra making the evaluation maps at time $t \leq s$ measurable \citep[see, \eg, Section 1.3 of][]{SV97}.

Now, let $\mathcal{M}$ be the space of integer-valued Borel measures defined on $\mathbb{R}_+ \times [0, C_\gamma]$, which are locally finite—that is, each measure assigns finite mass to any bounded Borel subset. Equipped with the vague topology, $\mathcal{M}$ becomes a Polish space, see, \eg, Section 4.1 of \citet{book:KALLENBERG-RM} or Appendix A2 of \citet{daley2003introduction}.
Let $(\mathscr{M}_s)_{s \geq 0}$ denote the canonical filtration on $\mathcal{M}$, where each $\mathscr{M}_s$ is the smallest $\sigma$-algebra such that the mappings $\nu \mapsto \nu(C)$, for all $C \in \mathcal{B}([0,s] \times [0, C_\gamma])$, are measurable. Equivalently, $\mathscr{M}_s$ can be described as the $\sigma$-algebra generated by the mappings $\nu \mapsto \nu([0,s] \times \cdot)$. The corresponding Borel $\sigma$-algebra on $\mathcal{M}$ is denoted by $\mathscr{M}$, and satisfies $\mathscr{M} = \bigvee_{s \geq 0} \mathscr{M}_s$.


Similarly, we define the space $\mathcal{H}$, its Borel $\sigma$-algebra $\mathscr{H}$, and its filtration $\{\mathscr{H}_t\}_{t \geq 0}$ by
\begin{align*}
    \mathcal{H} := \prod_{i \in \Ic} (\mathcal{W}\times \mathcal{M}), \quad
    \mathscr{H} := \bigotimes_{i \in \Ic} \left( \mathscr{W} \otimes \mathscr{M} \right), \quad
    \mathscr{H}_s := \bigotimes_{i \in \Ic} \left( \mathscr{W}_s \otimes \mathscr{M}_s \right).
\end{align*}
As countable products of Polish spaces, each component space $\mathcal{W}\times \mathcal{M}$ and the full product space $\mathcal{H}$ are also Polish. 

We define the canonical probability space by setting
\begin{align*}
    \Omega := \mathcal{H}, \quad \mathscr{F}_s := \mathscr{H}^\P_s, \quad \Pb := \bigotimes_{i \in \mathcal{I}} (\mathbb{W} \otimes \mathbb{Q}),
\end{align*}
where $ (\mathscr{H}^\P_s)_{s \geq 0} $ denotes the usual $ \Pb $-augmentation of the filtration $ (\mathscr{H}_s)_{s \geq 0} $, and $ \mathbb{W} $ (resp. $ \mathbb{Q} $) is the Wiener measure on $ \mathcal{W} $ (resp. the distribution of a Poisson random measure on $ \mathbb{R}_+ \times [0, C_\gamma] $) with Lebesgue intensity.

Given an element $ (w^j, \nu^j)_{j \in \mathcal{I}} \in \mathcal{H} $, for any $ s \geq 0 $ and $ U \in \mathcal{B}(\mathbb{R}_+ \times [0, C_\gamma]) $, we define the coordinate mappings by
\begin{align}
\label{eq:App-DPP:def-BM-PP}
    B^i_s\left( (w^j, \nu^j)_{j \in \mathcal{I}} \right) := w^i(s), \qquad
    Q^i\left( (w^j, \nu^j)_{j \in \mathcal{I}}, U \right) := \nu^i(U).
\end{align}

The proof of the DPP relies on the ability to work within the canonical space $\mathcal{H}$, which enables a measurable construction of the controlled branching process $\xi^{t,\lambda;\beta}$. To this end, we state the following results without proof, as they are direct generalizations of those in \citet{claisse18-v1}. Once \Cref{prop:existence_strong_branching} is established, the similarity in the setting makes the adaptation straightforward.

\begin{proposition}[Proposition 3.4 of \citet{claisse18-v1}]
\label{prop:App-DPP:admissible_control}
    For a process $\beta = (\beta_i)_{i\in\Ic}$, we have that $\beta\in\CtrlStandard$  if and only if, for every $i \in \Ic$, there exists a process $\beta^{\mathcal{H},i} : \mathbb{R}_+ \times \mathcal{H} \to A$,
    which is predictable with respect to the filtration $ (\mathscr{H}_s)_{s \geq 0} $, such that for all $ s \geq 0 $ and $ \omega \in \Omega $,
    \begin{align*}
        \alpha^i_s(\omega) = \alpha^{\mathcal{H},i}_s\left( \left( B^j(\omega), Q^j(\omega) \right)_{j\in\Ic} \right)
        = \alpha^{\mathcal{H},i}_s\left( \left( B^j_{s \wedge \cdot}(\omega), Q^j([0, s) \times \cdot)(\omega) \right)_{j\in\Ic} \right).
    \end{align*}
\end{proposition}

\begin{proposition}[Proposition 3.5 of \citet{claisse18-v1}]
\label{prop:App-DPP:canonical_process}
    Fix $(t,\lambda)\in\R_+\times E$. Let $ (\widetilde{\Omega}, (\widetilde{\mathcal{F}}_s)_{s \geq 0}, \widetilde{\mathbb{P}}) $ be a filtered probability space satisfying the usual conditions equipped with $ (\widetilde{B}^i, \widetilde{Q}^i)_{i \in \mathcal{I}} $ a family of independent Brownian motions and Poisson random measure on $\R_+\times[0,C_\gamma]$ with Lebesgue intensity measure. Fix $\widetilde{\beta}$ a control on $\widetilde{\Omega}$ defined by
    \begin{align*}
        \widetilde{\beta}^i_s(\widetilde{\omega}) := \beta^{\mathcal{H},i}_s\left( \left( \widetilde{B}^j(\widetilde{\omega}), \widetilde{Q}^j(\widetilde{\omega}) \right)_{j\in\Ic} \right)\eqsp,
        \qquad \text{ for } i \in \Ic\eqsp,\eqsp s \geq 0\eqsp,\eqsp \widetilde{\omega} \in \widetilde{\Omega}\eqsp.
    \end{align*}
    
    Then, there exists a unique (up to indistinguishability) $(\widetilde{\mathcal{F}}_s)_{s \geq 0}$-adapted càdlàg process $ \widetilde{\xi}^{t,\lambda,\widetilde{\beta}} $ satisfying the same semimartingale decomposition as in~\eqref{SDE:strong}, w.r.t.\ $ (\widetilde{B}^i, \widetilde{Q}^i)_{i \in \mathcal{I}} $. Moreover, there exists a Borel-measurable map
    \begin{align*}
        F^{t,\lambda,\beta^{\mathcal{H},\cdot}} : \mathcal{H} \to \mathbb{D}([t, +\infty), E)
    \end{align*}
    such that, for any , we have
    \begin{align*}
        \widetilde{\xi}^{t,\lambda,\widetilde{\beta}} = F^{t,\lambda,\beta^{\mathcal{H},\cdot}} \left( (\widetilde{B}^i, \widetilde{Q}^i)_{i \in \mathcal{I}} \right), \qquad \widetilde{\mathbb{P}} \text{--a.s.}
    \end{align*}
\end{proposition}

\paragraph{Pseudo-Markov property.}
We now introduce the \emph{pseudo-Markov property} required to establish the DPP. This property is derived using the canonical space formulation provided in \Cref{prop:App-DPP:admissible_control} and \Cref{prop:App-DPP:canonical_process}. Working within the canonical space proves especially beneficial in this setting, as it enables a rigorous and tractable framework in which filtrations, stopping times, and control processes are naturally and consistently defined. 

We consider the version of the pseudo-Markov property that is derived from the formulation in \citet{claisse18-v1}, which generalizes the original result of \citet{claisse2016pseudo} to the setting of controlled branching diffusions. This property plays a central role in rigorously handling conditional expectations with respect to filtrations at stopping times. It is a key technical tool for establishing the recursive structure that underlies the dynamic programming principle, which is fundamental to the stochastic control framework.

First, we define the concatenation in the canonical space. Fix $t\geq0$. For $(w_1,w_2) \in \mathcal{W}^2$, let $w_1 \otimes_t w_2 $ be defined by $(w_1 \otimes_t w_2)(s) := w_1(t \wedge s) + \left( w_2(s) - w_2(t) \right) \1_{s \geq t}$, for all $s \geq 0$. Similarly, for $(\nu_1, \nu_2) \in \mathcal{M}^2$, let $\nu_1 \otimes_t \nu_2 $ be $\pi_1 \otimes_t \pi_2 := \restr{{\pi_1}}{[0, t]} + \restr{{\pi_2}}{(t, +\infty)}$, for $s \geq 0$.
Fix, now, $\beta \in \CtrlStandard$. For $t \geq 0$ and $\bar{\omega} \in \Omega$, the \emph{shifted control} $\beta^{t,\bar{\omega}} $ is defined as
\begin{align*}
    \beta^{t,\bar{\omega}}_s(\omega) := \beta_s\left( \left( B^i(\bar{\omega}) \otimes_t B^i(\omega),\; Q^i(\bar{\omega}) \otimes_t Q^i(\omega) \right)_{i \in \mathcal{I}} \right)\eqsp, \qquad \text{ for } s \geq 0\eqsp,\eqsp\omega \in \Omega\eqsp.
\end{align*}
It is important to note that, for a fixed $\bar{\omega}$, the control $\beta^{t,\bar{\omega}} $ is admissible and independent of $\mathscr{F}_t$.


We can now state the pseudo-Markov property as follows.

\begin{lemma}[Lemma 3.7 of \citet{claisse18-v1}]
\label{lemma:App-DPP:pseudo_markov}
    Fix $(t, \lambda) \in \mathbb{R}_+ \times E$, $\beta\in\CtrlStandard$, and a stopping time $\tau \in \mathcal{T}_{t, +\infty}$. Then, for any Borel-measurable function $\varphi : \D([t, +\infty), E) \to \mathbb{R}_+$, it holds that
    \begin{align*}    
        \mathbb{E}\left[ \varphi\left( \xi^{t, \lambda, \beta} \right) \middle| \mathcal{F}_\tau \right](\bar{\omega})
        = \mathbb{E}\left[ \varphi\left( \xi^{\tau(\bar{\omega}),\, \xi^{t, \lambda, \beta}_{\tau \wedge \cdot}(\bar{\omega}),\, \beta^{\tau(\bar{\omega}), \bar{\omega}}} \right) \right], \qquad \P(d\bar{\omega})\text{--a.s.}\eqsp,
    \end{align*}
    with the notation
    \begin{align*}
        \xi^{
            \tau(\bar{\omega}),\eqsp
            \xi^{t, \lambda, \beta}_{\tau \wedge \cdot}(\bar{\omega}),\eqsp
            \beta^{\tau(\bar{\omega}), \bar{\omega}}
        }_s
        :=
        \xi^{t, \lambda, \beta}_s(\bar{\omega})
        \1_{s < \tau(\bar{\omega})}
        +
        \xi^{
            \tau(\bar{\omega}),\eqsp
            \xi^{t, \lambda, \beta}_{\tau}(\bar{\omega}),\eqsp
            \beta^{\tau(\bar{\omega}), \bar{\omega}}
        }_s
        \1_{s \geq \tau(\bar{\omega})}
        \eqsp.
    \end{align*}
\end{lemma}

The proof of this pseudo-Markov property remains unchanged from that given in \citet{claisse18-v1}, even in the present setting involving fully interacting particle systems like ours. This robustness stems from the fact that the result relies solely on the structure and decomposability of the underlying sources of randomness—namely, the family of Brownian motions and Poisson random measures—rather than on the specifics of the interaction mechanisms between particles.

This property has two key implications. First, it provides a \emph{conditioning principle} analogous to the classical tower property for conditional expectations, adapted to the controlled setting. Second, it allows us to restrict the optimization problem to controls that are \emph{independent of the past}, without any loss of generality. This simplification is formalized in the following corollary.

\begin{corollary}[Proposition 5.2
    of \citet{claisse18-v1}]
\label{corollary:App-DPP:pseudo_markov}
    Fix $(t, \lambda) \in \mathbb{R}_+ \times E$, $\beta\in\CtrlStandard$, and a stopping time $\tau \in \mathcal{T}_{t, T}$. Then, it holds that
    \begin{align*}
        J(t,\lambda;\beta) =&~ 
        \int_\Omega
        \left(
            \int_t^{\tau(\omega)} 
            \sum_{i\in \Vc^{t,\lambda;\beta}_s(\omega)}
            \psi\left(
                i,Y^{i,\beta}_s(\omega),\xi^{t,\lambda;\beta}_s(\omega),(\beta^{\tau(\omega), \omega})^i_s
            \right)ds + J\left(
                \tau(\omega),\xi^{t,\lambda;\beta}_\tau(\omega);
                \beta^{\tau(\omega), \omega}
            \right)
        \right)\P(d\omega)
        \eqsp.
    \end{align*}
\end{corollary}

\paragraph{Measurable selection.} One last ingredient needed to establish the DPP is a measurable selection principle, which allows us to select an $\varepsilon$-optimal control from the set of admissible controls. This is crucial for proving the existence of optimal controls and establishing the recursive structure of the value function.

We endow $\CtrlStandard$ with the Borel $\sigma$-algebra related to the distance $d_{\CtrlStandard,\Ic}$ defined by
\begin{align*}
    d_{\CtrlStandard,\Ic}(\beta_1, \beta_2) := \sum_{i \in \Ic} 2^{-|i|} \E\left[\int_{0}^{T}
        |\beta^i_1(s) - \beta^i_2(s)|ds
    \right]\eqsp.
\end{align*}

First, we retrieve a stability result for the branching system.

\begin{proposition}[Proposition 2.2 of \citet{kharroubi2024stochastic}]
\label{prop:App-DPP:stability}
    Suppose that Assumption H\ref{hypH:model_parameters} holds. Fix $(t,\lambda) \in \mathbb{R}_+ \times E$. Let $(t_n)_{n \geq 1} \subset \mathbb{R}_+$, $(\lambda_n)_{n \geq 1} \subset E$, and $(\beta_n)_{n \geq 1} \subset \CtrlStandard$ be sequences such that $(t_n, \lambda_n)$ converges to $(t, \lambda)$ as $n\to\infty$ and
    \begin{align*}
        \E\left[
            \int_{0}^{T}|
                \beta^i_{n,s} - \beta_s^i|ds
        \right]\xrightarrow[n\to\infty]{} 0\eqsp,
    \end{align*}
    for $i \in \Ic$. Then, the following convergence holds:
    \begin{align}
    \label{eq:stability-Ys}
        \E\left[
            \int_{0}^{T}
            \left| Y^{t_n,\lambda_n;\eqsp\beta_n,\eqsp i}_s \eqsp\1_{\Vc_s^{t_n,\lambda_n;\eqsp\beta_n}} - Y^{t,\lambda;\eqsp\beta,\eqsp i}_s \eqsp\1_{\Vc_s^{t,\lambda;\eqsp\beta}}\right|^2
        ds
        \right] \xrightarrow[n\to\infty]{} 0\eqsp,
        \qquad\text{ for }s \in [t,T],\eqsp i \in \Ic\eqsp.
    \end{align}
\end{proposition}

Fix $\varepsilon>0$. Let $\mathcal{U}_\varepsilon$ denote the set of all $\varepsilon$-optimal controls associated with a given initial condition, \ie,
\begin{align*}
    \mathcal{U}_\varepsilon(t,\lambda) := \left\{ \beta \in \CtrlStandard^t
    \eqsp:\eqsp
    J(t,\lambda;\beta) \leq v(t,\lambda) + \varepsilon
    \right\}\eqsp,
    \qquad \text{for } (t,\lambda) \in \mathbb{R}_+ \times E\eqsp,
\end{align*}
with $\CtrlStandard^t$ the collection of admissible controls independent of $\mathscr{F}_t$.

We aim to exhibit a function that associates to each $(t,\lambda)$ a control $\beta\in\mathcal{U}(t,\lambda)$ in a measurable way. To this purpose, we follow the outline of the proof of Lemma 3.1 in \citet{kharroubi2024stochastic}.

\begin{proposition}[Measurable selection]
\label{prop:App-DPP:measurable_selection}
    Suppose that Assumption H\ref{hypH:model_parameters} holds. Fix $\varepsilon > 0$. Then, for each $\nu\in\Pc([0,T]\times E)$, there exists a Borel-measurable function
    \begin{align*}
        \phi_\nu : \big([0,T] \times E,\eqsp\Bc([0,T]) \otimes \Bc(E)\big) \to \big(\CtrlStandard,\eqsp\Bc(\CtrlStandard)\big)
    \end{align*}
    such that $\phi_\nu(t,\lambda)\in\mathcal{U}_\varepsilon(t,\lambda)$, for $(t,\lambda) \in [0,T] \times E$.
\end{proposition}
\begin{proof}
    As noted in the proof of Lemma 3.1 of \citet{kharroubi2024stochastic}, first, $\CtrlStandard$ equippend with $\Bc(\CtrlStandard)$ forms a Borel space, generalizing Theorem 13.6 and 4.28 of \citet{aliprantis2006infinite}. Moreover, as a consequence of Theorem 3.4.1 and Theorem 3.4.5 in \citet{cohn2013measure} and Theorem 4.13 of \citet{brezis2011functional}, the set of predictable processes valued in $A$ is also separable for the $L^1$ distance. Let now $\mathscr{C}_\varepsilon$ (resp. $\bar{\mathscr{C}}$) defined by
    \begin{align*}
        \mathscr{C}_\varepsilon:= \left\{ (t,\lambda,\beta) \in [0,T] \times E\times \CtrlStandard : \beta\in\bar{\mathcal{U}}_\varepsilon(t,\lambda) \right\}
        \qquad
        \text{(resp. }
        \bar{\mathscr{C}}:= \left\{ (t,\lambda,\beta) \in [0,T] \times E\times \CtrlStandard : \beta\in\CtrlStandard_t \right\}
        \text{ )}
        \eqsp,
    \end{align*}
    with
    \begin{align*}
        \bar{\mathcal{U}}_\varepsilon(t,\lambda) := \left\{ \beta \in \CtrlStandard
        \eqsp:\eqsp
        J(t,\lambda;\beta) \leq v(t,\lambda) + \varepsilon
        \right\}\eqsp,
        \qquad \text{ for } (t,\lambda) \in \mathbb{R}_+ \times E
        \eqsp.
    \end{align*}
    As a consequence of the continuity of $\psi$ and $\Psi$, we have that \cref{prop:App-DPP:stability} the set $\mathscr{C}_\varepsilon$ is closed and a fortiori a Borel subset of $[0,T] \times E\times \CtrlStandard$. Arguing like in Lemma 3.1 of \citet{kharroubi2024stochastic}, $\bar{\mathscr{C}}$ is Borel, thus $\mathscr{C}_\varepsilon\cap \bar{\mathscr{C}}$ is a Borel. Therefore, Proposition 7.36 and Propositions 7.49 of \citet{bertsekas1996stochastic} imply there exists an analytically measurable function $\phi: [0,T]\times E\to\CtrlStandard$ such that $(t, \lambda, \phi(t,\lambda))\in \mathscr{C}_\varepsilon\cap \bar{\mathscr{C}}$ for all $(t,\lambda)\in[0,T]\times E$. Fix $\nu\in\Pc([0,T]\times E)$ and $\Bc_\nu([0,T]\times E)$ the completion of the Borel $\sigma$-algebra $\Bc([0,T]\times E)$ under $\nu$. Applying Corollary 7.42.1 of \citet{bertsekas1996stochastic}, $\phi$ is universally measurable, yielding the existence of a Borel measurable map $\phi_\nu$ such that $\phi_\nu(t,\lambda)\in \bar{\mathcal{U}}_\varepsilon(t,\lambda)$ for $\nu$--almost every $(t,\lambda)\in[0,T]\times E$.
\end{proof}


%% file: appendix/verification-thm.tex
\section{Verification Theorem}
\label{Appendix:verification-thm}
\begin{proof}[Proof of~\Cref{Prop:verification2}]
    By the local submartingale property in condition \textit{(ii)}, there exists a nondecreasing sequence of stopping times $(\tau_n)_n$ such that $\tau_n\uparrow T$ $\P$--a.s. and 
    \begin{align}
    \label{eq:verification_thm2:lower_bound}
        \E\left[
            w\left(
                T\wedge \tau_n, \xi^{\bar{t},\bar\lambda;\beta}_{T\wedge \tau_n}
            \right)
            +
            \int_{\bar t}^{T\wedge \tau_n} \sum_{i\in \Vc^{\bar{t},\bar\lambda;\beta}_u}
            \psi\left(i,Y^{i,\beta}_u,\xi^{\bar{t},\bar\lambda;\beta}_u,\beta^i_u\right)du
        \right]
        \geq 
        w(\bar{t},\bar \lambda)\eqsp, \quad \text{ for }\beta \in \CtrlStandard\eqsp.
    \end{align}
    Fix now $\eps>0$. Let $\beta \in \CtrlStandard^{\eps}_{(\bar t,\bar \lambda)}$, as defined in \Cref{prop:bound_eps_opt_control}. From~\eqref{eq:verification_thm2:growth_w} and~\eqref{eq:coercivity_hyp:psi}, we see that for all $n$ and $\beta \in \CtrlStandard^{\eps}_{(\bar t,\bar \lambda)}$, the l.h.s.\ is integrable and bounded by an integrable quantity. Applying dominated convergence theorem, by sending $n$ to infinity into~\eqref{eq:verification_thm2:lower_bound}, we get
    \begin{align*}
        w(\bar{t},\bar \lambda)
        \leq&\eqsp
        \E\left[w\left(
            T,\xi^{\bar{t},\bar\lambda;\beta}_{T}
        \right)+\int_{\bar{t}}^{T} \sum_{i\in \Vc^{\bar{t},\bar\lambda;\beta}_u}\psi\left(i,Y^{i,\beta}_u,\xi^{\bar{t},\bar\lambda;\beta}_u,\beta^i_u\right)du\right]
        \\
        = &\eqsp
        \E\left[\Psi\left(\xi^{\bar{t},\bar\lambda;\beta}_{T}\right)+\int_{\bar{t}}^{T} \sum_{i\in \Vc^{\bar{t},\bar\lambda;\beta}_u}\psi\left(i,Y^{i,\beta}_u,\xi^{\bar{t},\bar\lambda;\beta}_u,\beta^i_u\right)du\right]=J(\bar{t},\bar \lambda;\beta)\eqsp,
    \end{align*}
    using the terminal condition \textit{(i)} and~\eqref{eq:def:cost_function}. 
    Since $\beta$ is arbitrary in $\CtrlStandard^{\eps}_{(\bar t,\bar \lambda)}$, this shows that $v(\bar{t},\bar\lambda)\geq w(\bar{t},\bar\lambda)$. Moreover, we retrieve the reverse inequality when the local martingale property for $\bar \beta$ as condition \textit{(iii)} implies that all the previous inequalities holds as equality. Therefore, by applying Fatou's lemma we can conclude.
\end{proof}

\begin{proof}[Proof of~\Cref{Thm:Verification_Theorem}]
    \textit{(i)} 
    Fix $\Vc\in\Igen$, an initial condition $\left(t,\Vec{x}_\Vc\right) \in[0,T]\times \R^{d|\Vc|}$, and an admissible control $\beta\in\CtrlStandard$. Define $\lambda:=\iota^{-1}(\Vec{x}_\Vc)$. Consider the stopping times $\tau_k$ and $\theta_n$ defined as follows:
    \begin{align*}
        \tau_k :=&~ \inf\left\{s\in\left(\tau_{k-1},T\right]: \exists i \in \Vc^{t,\lambda;\beta}_{k-1}, ~ Q^i((\tau_{k-1},s]\times[0,C_\gamma])=1 \right\}
        \eqsp,\\
        \theta_n:=&~
        \inf\left\{s\in\left[t,T\right]: |V_s|\geq n \right\}\wedge \inf\left\{s\in\left[t,T\right]: \sum_{i\in \Vc^{t,\lambda;\beta}_u}\left|Y^{i,\beta}_u\right|\geq n \right\}\eqsp.
    \end{align*}
    W.r.t.\ these stopping times, the population $\xi^{t,\lambda;\beta}$ is equal to
    \begin{align*}
        \xi^{t,\lambda;\beta}_s 
        = 
        \sum_{k\geq1}\1_{\{\tau_{k-1}\leq s<\tau_k\}}\sum_{i\in \Vc^{t,\lambda;\beta}_{\tau_k}}\delta_{(i,Y^{i,\beta}_s)}
        = 
        \sum_{k\geq1}\1_{\{\tau_{k-1}\leq s<\tau_k\}}
        \eqsp 
        \iota^{-1}\left((Y^{\beta,i}_s)_{i\in\Vc^{t,\lambda;\beta}_{\tau_k}}\right)\eqsp.
    \end{align*}
    As noted in~\Cref{Rmk:diffusion_between_branching_events} and in the proof of \cref{prop:existence_strong_branching}, between two branching events $\tau_{k-1}$ and $\tau_k$, the population behave like a controlled diffusion living in $\R^{d|\Vc^{t,\lambda;\beta}_{\tau_{k-1}}|}$. Therefore, Itô's formula describes here the evolution of a function valued in $\xi^{t,\lambda;\beta}$ in each interval $\left[\tau_{k-1}\wedge \theta_n, \tau_k\wedge\theta_n\right)$.
    
    Denote $V^n_k:= \Vc^{t,\lambda;\beta}_{\tau_{k}\wedge \theta_n}$, $\Vec{Y}^{\beta,V^n_k}_{s} := (Y^{\beta,i}_{s})_{i\in V^n_k}$, and $\Vec{\beta}^{V^n_k}_{s} := (\beta^{i}_{s})_{i\in V^n_k}$, for $s\geq t$. Using $\iota$, we have that the semimartingale decomposition~\eqref{SDE:strong} translates into
    \begin{align*}
        &\E\left[
            w_{V^n_k}
                \left(s\wedge\tau_k\wedge\theta_n, \Vec{Y}^{\beta,V^n_k}_{s\wedge\tau_k\wedge\theta_n} \right) - 
                w_{V^n_{k-1}}
                \left(s\wedge\tau_{k-1}\wedge\theta_n, \Vec{Y}^{\beta,V^n_{k-1}}_{s\wedge\tau_{k-1}\wedge\theta_n} \right)
        \right]
        \\
        &
        =
        \E\left[
            \int_{s\wedge\tau_{k-1}\wedge\theta_n}^{s\wedge\tau_k\wedge\theta_n}\left\{\partial_t w_{V^n_{k-1}}\left(u,\Vec{Y}^{\beta,V^n_{k-1}}_u\right) + \Lb_{V^n_{k-1}} w_{V^n_{k-1}}\left(\Vec{Y}^{\beta,V^n_{k-1}}_u, \Vec{\beta}^{V^n_{k-1}}_u\right)\right\} du
        \right]\eqsp.
    \end{align*}
    Therefore, we have that
    \begin{align}
        \label{eq:verification_thm:flow}
        \begin{split}
            &\E\left[
            w_{\Vc^{t,\lambda;\beta}_{s\wedge \theta_n}}
                \left(s\wedge\theta_n, \Vec{Y}^{\beta,\Vc^{t,\lambda;\beta}_{s\wedge \theta_n}}_{s\wedge\theta_n} \right) 
            \right] - w_{\Vc}\left(t, \Vec{x}_\Vc\right)  
        \\
        &=\E\left[\sum_{k\geq1}\left(
            w_{V^n_k}
                \left(s\wedge\tau_k\wedge\theta_n, \Vec{Y}^{\beta,V^n_k}_{s\wedge\tau_k\wedge\theta_n} \right) - 
                w_{V^n_{k-1}}
                \left(s\wedge\tau_{k-1}\wedge\theta_n, \Vec{Y}^{\beta,V^n_{k-1}}_{s\wedge\tau_{k-1}\wedge\theta_n} \right)\right)
            \right]
        \\
        &=\E\left[\sum_{k\geq1}
            \int_{s\wedge\tau_{k-1}\wedge\theta_n}^{s\wedge\tau_k\wedge\theta_n}\left\{\partial_t w_{V^n_{k-1}}\left(t,\Vec{Y}^{\beta,V^n_{k-1}}_u\right) + \Lb_{V^n_{k-1}} w_{V^n_{k-1}}\left(\Vec{Y}^{\beta,V^n_{k-1}}_u, \Vec{\beta}^{V^n_{k-1}}_u\right)\right\} du
        \right]\eqsp.
        \end{split}
    \end{align}
    
    Since $w$ satisfies~\eqref{eq:verification_thm:inf_condition}, we have
    \begin{align*}
        \partial_t w_{V^n_{k}}\left(t,\Vec{Y}^{\beta,V^n_{k}}_u\right) + \Lb_{V^n_{k}} w_{V^n_{k}}\left(\Vec{Y}^{\beta,V^n_{k}}_u, \Vec{\beta}^{V^n_{k}}_u\right)+\sum_{i \in \Vc^{t,\lambda;\beta}_{\tau_k\wedge\theta_n}} \psi\left(i,Y^{i,\beta}_u, \xi^{t,\lambda;\beta}_u,\beta^i_u\right) \geq 0\eqsp,
    \end{align*}
    for $\beta\in\CtrlStandard$, $k\geq0$, and $u\in \left[\tau_{k}\wedge\theta_n,\tau_{k+1}\wedge\theta_n\right)$.
    Thus,
    \begin{align}\label{eq:verification_thm:HJB-traj}
        \E\left[
            w_{\left|V_{s\wedge\theta_n}\right|}
                \left(s\wedge\theta_n, \Vec{Y}^{\beta,\left|V_{s\wedge\theta_n}\right|}_{s\wedge\theta_n} \right) 
            \right] - w_\Vc\left(t, \Vec{x}_\Vc\right)
        \geq -
        \E\left[
            \int_t^{s\wedge\theta_n}\sum_{i\in \Vc^{t,\lambda;\beta}_u} \psi\left(i,Y^{i,\beta}_u, \xi^{t,\lambda;\beta}_u,\beta^i_u\right)du 
        \right]\eqsp.
    \end{align}
    From~\eqref{eq:coercivity_hyp:psi}, we have
    \begin{align*}
        \left|\int_t^{s\wedge\theta_n}\sum_{i\in \Vc^{t,\lambda;\beta}_u} \psi\left(i,Y^{i,\beta}_u, \xi^{t,\lambda;\beta}_u,\beta^i_u\right)du \right|
        \leq
        C_\Psi \left(1+ \int_t^{T}\left(|V_u|^2 + \sum_{i\in \Vc^{t,\lambda;\beta}_u}\left|Y^{i,\beta}_u\right|^2+ \sum_{i\in \Vc^{t,\lambda;\beta}_u}\left|\beta^{i}_u\right|^2 \right)du\right)
        \eqsp,
    \end{align*}
    therefore the r.h.s.\ in~\eqref{eq:verification_thm:HJB-traj} is integrable for
    $\beta\in\CtrlStandard^\varepsilon_{(t,\lambda)}$
    using~\eqref{eq:non-explosion-moment2_mass},~\eqref{eq:non-explosion-moment2_population} and~\eqref{eq:bound_eps_opt_control}. Analogously, from~\eqref{eq:verification_thm:growth_w}, we also have that l.h.s.\ in~\eqref{eq:verification_thm:HJB-traj} is integrable for
    $\beta\in\CtrlStandard^\varepsilon_{(t,\lambda)}$. We can then apply the dominated convergence theorem, and send $n$ to infinity into~\eqref{eq:verification_thm:HJB-traj}:
    \begin{align*}
        \E\left[
            w_{\Vc^{t,\lambda;\beta}_s}
                \left(s, \Vec{Y}^{\beta,\Vc^{t,\lambda;\beta}_s}_{s} \right) 
            \right] - w_{\Vc}\left(t, \Vec{x}_\Vc\right)
        \geq
        - \E\left[
            \int_t^{s}\sum_{i\in \Vc^{t,\lambda;\beta}_u} \psi\left(i,Y^{i,\beta}_u, \xi^{t,\lambda;\beta}_u,\beta^i_u\right)du 
        \right]\eqsp,\qquad
        \text{ for }
        \beta\in\CtrlStandard^\varepsilon_{(t,\lambda)}
        \eqsp.
    \end{align*}
    Since $w$ is continuous on $[0,T]\times E$, by sending $s$ to $T$, we obtain by the dominated convergence theorem and by~\eqref{eq:verification_thm:terminal_cond}
    \begin{align*}
    \E\left[
        \Psi \left(\xi^{t,\lambda;\beta}_T \right) 
        \right] - w_\Vc\left(t, \Vec{x}_\Vc\right)
    \geq
    - \E\left[
        \int_t^{T}\sum_{i\in \Vc^{t,\lambda;\beta}_u} \psi\left(i,Y^{i,\beta}_u, \xi^{t,\lambda;\beta}_u,\beta^i_u\right)du 
    \right], \qquad\text{ for }\beta\in\CtrlStandard^\varepsilon_{(t,\lambda)}\eqsp.
    \end{align*}
    From the arbitrariness of $\beta\in\CtrlStandard^\varepsilon_{(t,\lambda)}$, we deduce that $w_\Vc(t, \Vec{x}_\Vc)\leq v_\Vc(t, \Vec{x}_\Vc)$, which entails $w(t, \lambda)\leq v(t, \lambda)$, for $(t,\lambda) \in[0,T]\times E$.

    \textit{(ii)} From \eqref{eq:verification_thm:inf_condition}, we have that
    \begin{align*}
        -\partial_t w_\Vc\left(t,\Vec{x}_\Vc\right)
        - \left\{\Lb_\Vc v_\Vc
        \big(\Vec{x}_\Vc, \Vec{\mathfrak{a}}_\Vc\left(t,\Vec{x}_\Vc\right)\big) 
        - \sum_{i \in\Vc} \psi\big(i,x_i, \iota^{-1}(\Vec{x}_\Vc),(\Vec{\mathfrak{a}}_\Vc)_i\left(t,\Vec{x}_\Vc\right)\big) \right\}=0\eqsp.
    \end{align*}
    Applying this to~\eqref{eq:verification_thm:flow}, we get
    \begin{align*}
        w_\Vc\left(t, \Vec{x}_\Vc\right) = 
        \E\left[
            w_{\Vc^{t,\lambda;\hat\beta}_{s\wedge\theta_n}}
                \left(s\wedge\theta_n, \Vec{Y}^{\hat\beta,\Vc^{t,\lambda;\hat\beta}_{s\wedge\theta_n}}_{s\wedge\theta_n} \right) 
            +
            \int_t^{s\wedge\theta_n}\sum_{i\in \Vc^{t,\lambda;\hat\beta}_u} \psi\left(i,Y^{i,\hat\beta}_u, \xi^{t,\lambda;\hat\beta}_u,\hat\beta^i_u\right)du 
        \right]\eqsp,
    \end{align*}
    for $n\geq 1$. For Fatou's lemma, we obtain
    \begin{align*}
    w_\Vc\left(t, \Vec{x}_\Vc\right) \geq \E\left[
        w_{\Vc^{t,\lambda;\hat\beta}_s}
        \left(s, \Vec{Y}^{\hat\beta,\left|V_{s}\right|}_{s} \right)
        +
        \int_t^{s}\sum_{i\in \Vc^{t,\lambda;\hat\beta}_u} \psi\left(i,Y^{i,\hat\beta}_u, \xi^{t,\lambda;\hat\beta}_u,\hat\beta^i_u\right)du 
    \right]\eqsp.
    \end{align*}
    Sending $s$ to $T$ and using again Fatou's lemma, together with the fact $w_{\Vc^\prime}\left(T,\Vec{y}_{\Vc^\prime}\right) =\Psi\left(\iota^{-1}\left(\Vec{y}_{\Vc^\prime}\right)\right)$, for $\Vc^\prime\in \Igen$, and $\Vec{y}_{\Vc^\prime}\in \R^{d|\Vc^\prime|}$, we see that
    \begin{align*}
        w_\Vc\left(t, \Vec{x}_\Vc\right) \geq \E\left[
            \Psi\left(\xi^{t,\lambda;\hat\beta}_T\right)+
            \int_t^{s}\sum_{i\in \Vc^{t,\lambda;\hat\beta}_u} \psi\left(i,Y^{i,\hat\beta}_u, \xi^{t,\lambda;\hat\beta}_u,\hat\beta^i_u\right)du 
        \right]
        = J\left(t, \iota^{-1}\left(\Vec{x}_\Vc\right); \hat\beta\right)
        \eqsp.
    \end{align*}
    This shows that $w_\Vc(t, \Vec{x}_\Vc) \geq J(t, \iota^{-1}(\Vec{x}_\Vc); \hat\beta) \geq v_\Vc(t, \Vec{x}_\Vc)$, and finally that $w = v$ with $\hat\beta$ as an optimal Markovian control.
\end{proof}